\documentclass{elsarticle}
\pdfoutput=1
\usepackage{latexsym, amssymb, enumerate, amsmath,subcaption}
\usepackage{url}
\usepackage{graphics,graphicx,subfig,color}
\usepackage{bm}
\usepackage[numbers]{natbib}

\newcommand{\kone}[1]{{ #1}}
\newcommand{\ktwo}[1]{{ #1}}
\numberwithin{equation}{section}
\newtheorem{theorem}{Theorem}[section]
\newtheorem{lemma}[theorem]{Lemma}

\newtheorem{proposition}[theorem]{Proposition}
\newtheorem{remark}[theorem]{Remark}

\newproof{proof}{Proof}
\newcommand{\eps}{{\varepsilon}}
\renewcommand{\d}{\mathrm{d}}
\newcommand{\D}{\mathrm{D}}

\renewcommand*{\d}{{\mathrm d}}

\newcommand*{\x}{{\boldsymbol{x}}}

\renewcommand*{\Xi}{{\boldsymbol{\xi}}}

\renewcommand*{\O}{{\mathcal{O}}}

\newcommand*{\R}{{\mathbb{R}}}

\newcommand{\LE}{{\mathcal{L}}_G}

\newcommand{\dt}[0]{{\delta t}}
\newcommand{\Dt}[0]{{\Delta t}}
\newcommand{\td}[0]{t_\Delta}
\newcommand{\gt}[0]{\tilde{g}}
\newcommand{\Gt}[0]{\tilde{G}}

\newcommand{\twoparts}[4]
{
	\left\{
		\begin{array}{lll}
			&\displaystyle#1 &\quad \textnormal{for } \displaystyle#2 \vspace{0pt}\\ 
			&{\displaystyle#3 }& {\quad \textnormal{for } \displaystyle#4 }
		\end{array}
	\right.
}
\newcommand{\uR}[1]{\uppercase\expandafter{\romannumeral#1}}

\allowdisplaybreaks

\begin{document}

\title{On convergence of higher order schemes for the projective integration method for stiff ordinary differential equations}
%\author[JM]{John~Maclean\corref{cor1}}
\author{John~Maclean\corref{cor1}}
\ead{j.maclean@maths.usyd.edu.au}

\author{Georg~A.~Gottwald\corref{cor2}}
%\author[GAG]{Georg~A.~Gottwald}
\ead{georg.gottwald@sydney.edu.au}

\cortext[cor1]{Corresponding author}

\address{School of Mathematics and Statistics, University of Sydney, NSW 2006. Australia.}
%\address[GAG]{School of Mathematics and Statistics, University of Sydney, NSW 2006. Australia.}
          %{Put the URL for your home page here if you have one}

          %Use \thanks statements for acknowledgements of grants and
          %support. They will appear below all the authors' addresses, so be
          %specific about which author is thanking whom:

          %\thanks{}
          
%    \subjclass is required.
%\subjclass[2010]{65LXX, 65PXX, 34E13, 37MXX }

\pagestyle{myheadings}\markboth{On convergence of the projective integration method}{John Maclean and Georg A. Gottwald}

\begin{abstract}
We present a convergence proof for higher order implementations of the projective integration method (PI) for a class of deterministic multi-scale systems in which fast variables quickly settle on a slow manifold. The error is shown to contain contributions associated with the length of the microsolver, the numerical accuracy of the macrosolver and the distance from the slow manifold caused by the combined effect of micro- and macrosolvers, respectively. %In particular, the error is shown to be independent of the order of the microsolver. 
We also provide stability conditions for the PI methods under which the fast variables will not diverge from the slow manifold. We corroborate our results by numerical simulations. 
\end{abstract}

\maketitle
\section{Introduction} Many problems in the natural sciences are modelled by multidimensional ordinary differential equations with entangled processes running on widely separated time scales. 
One is often interested in resolving the behaviour of the slow processes over a long, macro time scale. However, the fast processes prevent direct solution of the system by traditional numerical methods. %In response, multiscale methods have been developed to approximate - analytically or numerically - the dynamics of the slow processes without resolving the full behaviour of the system over long time scales. 
Recently two numerical methods designed to overcome the restriction to the small integration time step in these stiff dynamical systems have been much studied; the \emph{projective integration method} within the equation-free framework and the \emph{heterogeneous multiscale methods} (HMM). Each method exists in multiple formulations; in the PI method, we mention \cite{KevrekidisGearEtAl03,GearKevrekidis03,HummerKevrekidis03,GearLee05,LustRooseVandekerckhove06,GivonEtAl06,KevrekidisSamaey09,LafitteSamaey10}, and in the HMM, \cite{EEngquist03,VandenEijnden03,E03,EngquistTsai05,ELiuVandenEijnden05,EEtAl07,Liu10}. There is some debate on the similarities and differences between the methods; the interested reader is referred to \cite{VandenEijnden07,EVandenEijnden08} for a discussion. \\
Both methods assume that the fast variables in the full multiscale system quickly relax to a slow manifold, after which the dynamics of the slow variables is governed by a slow reduced system. Both methods estimate the effective influence of the fast variables on the dynamics of the slow variables by employing a \emph{microsolver} to perform short fine-scale computations with small time steps (microsteps). This information is used to propagate the dynamics on the slow manifold for large time steps (macrosteps) in the \emph{macrosolver}.\\
The philosophy behind each method is slightly different. The PI approach estimates the effective slow vector field via direct numerical evaluation, not assuming any knowledge on the form of the reduced vector field; this forms part of the equation-free approach. In contrast, the HMM philosophy utilises a priori analytical knowledge about the reduced vector field. \\%assumes that the fast variables in the full multiscale system quickly relax to a slow manifold, after which the dynamics of the slow variables is governed by a slow reduced system. By repeatedly initialising a microsolver to perform short fine-scale computations on the full system, much larger, macro length time steps can be extrapolated in a macrosolver. The HMM philosophy is to employ {\emph{a priori}} analytical knowledge about the reduced system (often conservation laws or other macro scale information) and identify the information missing from the observable data that would enable numerical computations using the reduced model. A microsolver is then employed to estimate the missing data, which is subsequently used in the macro scale model with an appropriate macrosolver. 

\noindent
In this paper, we focus on numerical methods that are \emph{seamless}; that is, the numerical methods do not explicitly separate the slow variables and the fast variables at any stage in the solver, but instead propagate all variables simultaneously. These methods are useful in systems where conceptually there exists a decomposition or transformation of the system into slow and fast variables, but where this transformation is unknown. 
The added complication of seamless numerical methods is that the fast variables are propagated simultaneously with the slow variables with the large time step of the macrosolver. This may lead to a more severe departure of the fast variables from the slow manifold over the macrosteps in comparison to nonseamless methods. \\

\noindent
\ktwo{In first order PI methods the micro- and macrosolver are applied sequentially, so the error accrued by the micro- and macrosolver can be analysed separately, as for example in \cite{E03,GottwaldMaclean13}. There are two different approaches to extend PI to higher order solvers. First, one can still apply the micro- and macrosolver sequentially, as in \cite{GearEtAl02,EngquistTsai05,LustRooseVandekerckhove06,RooseVandekerckhove06}. The analysis in \cite{EngquistTsai05,RooseVandekerckhove06} shows that such schemes can be accurate to second order in the size of the macrosolver. Alternatively, one can apply the microsolver multiple times during each time step of the macrosolver, as in \cite{E03, GearLee05, LafitteEtAl14}. The numerical schemes that we will consider take this approach. The analysis of such methods is complicated by the requirement that the errors accrued by the micro- and macrosolvers, which are intertwined due to the nonlinear nature of the dynamics, have to be estimated simultaneously. 
%These methods have the practical advantage that one is not limited by the second order of the macrosolver. 
In \cite{E03}, an error bound is proposed for a seamless HMM scheme of arbitrary order, albeit without proof. In \cite{GearLee05, LustRooseVandekerckhove06, RooseVandekerckhove06}, second order PI schemes are proposed and analysed. In \cite{LafitteEtAl14}, error bounds for the slow variables and stability conditions are derived for an arbitrary order Runge-Kutta macrosolver applied to a kinetic equation with linear relaxation.}\\

\noindent
In this paper we present a higher order seamless multiscale method as considered in \cite{E03,GearLee05}, \ktwo{for a system of nonlinear stiff ordinary differential equations}. We propose a slight modification of this method which, involving an additional application of the microsolver, constructs slow vector fields pointing towards the slow manifold. Both schemes reduce to Runge-Kutta methods if the microsolver is switched off. We establish rigorous convergence results for the slow variables of these methods. We find that both methods incur error terms propotional to the order of the macrosolver, the distance of the fast variables from the slow manifold, and an additional term due to the microsolver, independent of the order of the microsolver. \ktwo{This result confirms for the two methods we consider the error bound suggested in \cite{E03}.} Furthermore, we find that the error due to the microsolver is smaller in our proposed method when both methods are employed at the same computational cost.\\ \ktwo{A known problem in seamless methods is that the macrosolver may lead to a departure of the fast variables from the slow manifold. To combat this divergence of the fast variables, several methods have been introduced \cite{GearEtAl05, ZagarisEtAl09, VandekerckhoveEtAl11,ZagarisEtAl12}; analytical bounds on the departure of the fast variables from the slow manifold over a macrostep have received relatively little attention (with the notable exception of \cite{EngquistTsai05}). Estimates of the maximal deviation of the fast variables from the slow manifold are particularly important when bifurcations occur or when the dynamics transits to different solution branches (e.g. \cite{GearEtAl02,KevrekidisGearEtAl03,KevrekidisSamaey09,SiettosEtAl12}); if the departure from the slow manifold is too large, the transitions may be premature.}\\ We establish bounds on the departure of the fast variables from the slow manifold over the macrosolver.  The bounds show that the numerically induced departure of the fast variables from the slow manifold scales one order better in the macrostep size in our modified version of PI. Furthermore, these bounds allow us to derive stability conditions for both methods under which the departure of the fast variables from the slow manifold remains finite over the macrosteps.\\
%linearly with the macrostep in the existing method and quadratically with the macrostep in our proposed method. \\

%\noindent In this paper we analyse a numerical method consistent with existing formulations of seamless HMM and PI as treated in \cite{E03,EngquistTsai05,GearLee05}, and propose a modified version of the method used in \cite{E03}, which employs an additional application of the microsolver. Our work is an extension to higher order of results presented in \cite{GottwaldMaclean13} for first order PI. We establish the convergence of these methods to the continuous solution of the slow processes in a multiscale system of deterministic ordinary differential equations amenable to slow manifold theory, finding that both methods incur error terms propotional to the order of the macrosolver, the reiniti%alisation, and an additional term due to the microsolver which is independent of the order of the microsolver. We find that the error due to the microsolver is smaller in our proposed method. We then establish stability conditions for both methods under which the fast variables stay close to the slow manifold over the macrosteps, finding in particular that the reinitialis%ation scales linearly with the macrostep in the existing method and quadratically under our proposed method. \\

\noindent
The paper is organized as follows. In Section~\ref{sec:model} we discuss the class of dynamical systems studied, and briefly summarize in Section~\ref{sec:rk} classic Runge-Kutta methods for these systems. We then present two multiscale methods which enable the solution of these systems with macro length time steps in Section~\ref{sec:methods}. In Section~\ref{sec:proof}, the main part of this work, we derive rigorous error bounds for those numerical multiscale methods. In Section~\ref{sec:numerics} we present results from numerical simulations corroborating our analytical findings. We conclude with a discussion in Section~\ref{sec:summary}.

%%%%%%%%%%%%%%%%%%%%%%%%%%%%%%%%%%%%%%%%%

\section{Model}
\label{sec:model}

We consider deterministic multiscale systems of the form
\begin{align}
\label{baseneo} \dot z_\eps &= \mathcal{F}(z_\eps,\, \eps) \;\; ,
\end{align}
with $z_\eps \in \R^{n+m}$ and time scale separation parameter $0<\varepsilon \ll 1$. We assume there is a (possibly unknown) decomposition $z_\eps = (x_\eps, y_\eps)$ into fast variables $x_\eps\in\R^m$ and slow variables $y_\eps\in\R^n$ which evolve according to
\begin{align}
\dot y_\eps &=g(x_\eps,y_\eps)\; ,
\label{baseslow}
\\
\label{basefast.f} \dot x_\eps &= \frac{1}{\eps} f(x_\eps,y_\eps) \;\; .
\end{align}
We consider here the particular fast vector fields of the form
\begin{align}
f(x_\eps,y_\eps) &=\frac{\Lambda}{\varepsilon}(-  \, x_\eps+ h_0(y_\eps))\; .%f(x_\eps,y_\eps,\eps)\;,
\label{basefast}
\end{align}
%Without loss of generality we assume that $(x_\eps,y_\eps)=(0,0)$ is a fixed point.  
We assume there is a coordinate system such that the matrix $\Lambda\in \R^{m\times m}$ is diagonal with diagonal entries $\lambda_{ii}>0$. We further allow for a scaling of time such that $\min(\lambda_{ii}) =1 $ and define $\max(\lambda_{ii})=\lambda$. % We assume that the vectorfield of the slow variable $g(x,y)$ is purely nonlinear with $\det \D g(0,0) \kone{=} 0$, where $\D g$ denotes the Jacobian of $g(x,y)$, so
We assume that there exists a slow manifold $x=h_\eps(y) = h_0(y) + \mathcal{O}(\eps)$, towards which initial conditions are attracted exponentially fast. On the slow manifold, the dynamics slows down and is approximately determined by
\begin{align}
\dot Y = G(Y)\; ,
\label{e.CMT}
\end{align}
with $Y=y_\eps+{\mathcal{O}}(\eps)$ and reduced slow vectorfield
\begin{align}
\label{e.G0} G(y) &= g(h_\eps(y),y) \; .
\end{align}

\section{Runge-Kutta Solvers}
\label{sec:rk}
We denote by $z_\eps(t^n)$ the solution of \eqref{baseneo} evaluated at the discrete time $t^n=n\Dt$, and by $\bar{z}^n$ the numerical approximation of $z_\eps(t^n)$ given by a Runge-Kutta solver of order P. %Throughout the paper variables with superscripts denote discrete variables whereas variables with brackets denote continuous variables. 
Runge-Kutta solvers form approximations to the dynamics in terms of \emph{increments}. \ktwo{For simplicity, we restrict our analysis to Runge-Kutta methods in which increments are given recursively by}
\begin{align}
\label{rk.k}\bar{k}_j(\bar{z}^n) = \Dt \, \mathcal{F}(\bar{z}^n + a_j \bar{k}_{j-1},\eps)\;\; , 
\end{align}
for $j=1,2,\dots,P$. The values of the nodes $a_j$ depend on the order P (see for instance \cite{Iserles}), and satisfy $0\le a_j \le 1$, with $a_1=0$ so that the first increment is defined explicitly. Each increment $\bar{k}_j$ evaluates the vector field $\mathcal{F}$ of \eqref{baseneo} at the intermediate time $t^n + a_j\Dt$. The increments are averaged to define $\bar{z}^{n+1}$, with
\begin{align}
\label{rk.m} \bar{z}^{n+1} &= \bar{z}^n + \sum_{j=1}^P b_j \bar{k}_j(\bar{z}^n) \;\; ,
\end{align}
where the weights $b_j$ satisfy the condition $\textstyle\sum_{j=1}^P b_j = 1$, and depend on the order P and the particular choice of nodes $a_j$.
For instance, for $P=4$, the widely used fourth-order Runge-Kutta scheme, the nodes and weights may be given by
\begin{align}
\label{nodes} a_j &= \left\{ 0, \, \frac{1}{2},\, \frac{1}{2},\, 1\right\} \;\;, \\
\label{weights} b_j &= \left\{ \frac{1}{6}, \,\frac{1}{3},\, \frac{1}{3},\, \frac{1}{6} \right\} \;\; .
\end{align}

\noindent
For any P, the nodes and weights are determined such that the application of a single Runge-Kutta step of order P to a system with initial condition $\bar{z}_\eps(t)$ and time step $\Dt$ produces an approximation to $\bar{z}_\eps(t+\Dt)$ accurate to within $\mathcal{O}(\Dt^{P+1})$; see for instance \cite{Iserles}. In particular for linear systems 
 \begin{align*}
 \dot{x}_\eps &= -\frac{\Lambda}{\eps} x_\eps  \;\; ,
 \end{align*}
for which 
  \begin{align*}
  \frac{\d^j x_\eps}{\d t^j} &=  \left(-\frac{\Lambda}{\eps}\right)^j \;\; ,
  \end{align*}
  a single Runge-Kutta step of $P$-th order can be written as
\begin{align*}
x^{n+1} &= \rho\left(-\frac{\Lambda \Dt}{\eps}\right) x^n \;\;,
\end{align*}
where the \emph{linear amplification factor} $\rho$ is given by the Taylor polynomial to order P of an exponential function
\begin{align}
\label{linamp} \rho(\eta) & = \sum_{j=0}^P \frac{(\eta)^j}{j!} \;\; .
\end{align}

\noindent
A straightforward implementation of Runge-Kutta methods to simulate stiff dynamical systems such as \eqref{baseslow}--\eqref{basefast} would be computationally too costly, as the time step is restricted to $\Dt \le \mathcal{O}(\eps)$ to ensure numerical stability.\\

\noindent
In the next section we present two numerical multiscale schemes which are designed to overcome the problem of stiffness presented above. These schemes employ a microsolver to relax the fast variables towards the slow manifold. Utilising the slowness of the dynamics on the slow manifold allows for the application of Runge-Kutta methods with large macro time steps $\Dt \gg \eps$.

\section{Numerical Multiscale Methods}
\label{sec:methods}

%In this section we describe two numerical methods that can resolve the fast dynamics of a multiscale system and describe a class of standard numerical methods, followed by a related class of multiscale methods. We will be 
%\todo{Paragraph: mention the current abilities of standard numerical solvers and the problems applying them to multiscale systems. Follow with the design goal of ms numerical methods and the current existing methods}

\noindent
We consider two seamless projective integration methods. The first is a general order formulation of PI as proposed in \cite{E03,GearLee05,EngquistTsai05}. 
%This method consists of a Runge-Kutta scheme with an application of a microsolver before each increment is evaluated. 
We call this scheme PI1. 
The second is a modification of PI1, which employs information from the microsolver to define increments which point in the direction of the slow manifold, at the cost of one additional application of the microsolver\footnote{We ensure that the overall cost of PI1 and PI2 is the same when they are compared numerically by adjusting the total number of microsteps in each method (see Section~\ref{sec:numerics}).}. We call this method PI2. The PI1 and PI2 schemes differ in the definition of the increments. \\

\noindent
Denote by $z^n$ the numerical approximation given by the multiscale scheme to $z_\eps(t^n)$; using $z^n$ as the initial condition, both methods employ a microsolver with small microstep $\dt$, and then evaluate the vectorfield over a large macrostep $\Dt \gg \eps$. Iterating these steps enables one to construct increments which cover a macro time scale. The macrosolver then combines these increments in a weighted sum in Runge-Kutta fashion. \\
We denote by $\varphi^{m,\dt}$ the flow map for the microsolver run for $m$ microsteps with time step $\dt$ and assume that it describes an explicit numerical method of order p. We do not specify which particular numerical method is chosen; as we will see in Proposition~\ref{lemma_Gtwiddle}, increasing the order of the microsolver does not improve the predicted overall error scaling. \\
\noindent
%We will discuss the benefits of particular choices for microsolver in Section \ref{sec:proof}. We emphasise that the microsolver is \emph{seamless}; each step of the microsolver advances both slow and fast variables simultaneously.
In the following we detail PI1 and PI2 and highlight their differences. The procedures are illustrated in Figure~\ref{fig:PI1} for PI1 and in Figure~\ref{fig:PI2} for PI2.

\subsection{Projective Integration Scheme PI1}
We describe here a general order formulation of projective integration along the lines of \cite{E03,GearLee05,EngquistTsai05,GottwaldMaclean13}. We remark that this formulation is an instance where PI and HMM are essentially the same (see \cite{E03,VandenEijnden07}).
%We describe here a general order formulation of the scheme introduced in \cite{E03}, where it is called seamless HMM (see also the similar schemes in \cite{GearLee05,EngquistTsai05}). 
The scheme PI1 is a modified Runge-Kutta scheme in which the microsolver is employed to relax the fast variables close to the slow manifold before each increment is estimated.\\
%In this scheme, the vectorfield for each increment in the solver is evaluated after the microsolver has been employed to relax the fast variables to the slow manifold. 

We denote by $z^{n,j}_{m}$ the approximation of the fast and slow variables at the $m$-th microstep of the $j$-th increment at time step $n$, and denote by $M_j$ the integer number of microsteps taken before the $j$-th increment is estimated. We denote discrete times associated with microsolvers by subscripts and those associated with macrosolvers by superscripts.\\
%
%The scheme starts with an application of the microsolver with $M_1$ steps and time step $\dt$, with
%\begin{align}
%\label{def:znM}z^{n,1}_{M_1} & = \varphi^{M_1,\dt}(z^n) \;\; , 
%\label{def:ynM}\left(x^{n,1}_{M_1},y^{n,1}_{M_1}\right) & = \varphi^{M_1,\dt}(x^n,y^n) \;\; , 
%\end{align}
%illustrated in Figure~\ref{fig:PI1nM}. 
The increments cover a time step of $\Dt$ and are given by evaluating $\mathcal{F}$ after an application of the microsolver, with
\begin{align}
\label{hatkz} \hat{k}_{j}(z^{n})& = \Dt \; \mathcal{F} (z^{n,j}_{M_j},\, \eps ) \;\; ,
%\label{hatky} \hat{k}_{y,j}(x^{n},y^{n})& = \Dt \; g\left({x}^{n,j}_{M_j},{y}^{n,j}_{M_j} \right) \;\;,\\
%\label{hatkx} \hat{k}_{x,j}(x^{n},y^{n})& = \Dt \; f\left({x}^{n,j}_{M_j},{y}^{n,j}_{M_j},\eps \right) \;\; ,
\end{align}
%where $M_j \in \mathcal{N}$ is the number of microsteps taken before estimating the $j$-th increment, and 
for $j=1,2,\dots,P$, where we define $z^{n,j}_{m}$ for $j=1,2,\dots,P$, $m=1,2,\dots,M_j$, as the output of the microsolver
\begin{align}
\label{def:znjm} z^{n,j}_{m} &= \varphi^{m,\dt}\left(z^{n,j}_{0}\right) \;\; , 
\end{align}
with initial condition
\begin{align}
\label{def:znj0} z^{n,j}_{0} &= \twoparts{z^{n}}{j=1}
{z^{n,1}_{M_1} + a_j \hat{k}_{j-1}(z^n)}{j>1} \;\; .
%\label{def:ynjm} \left(x^{n,j}_{m},{y}^{n,j}_{m}\right) &= \varphi^{m,\dt}\left(x^{n,1}_{M_1} + a_j \hat{k}_{x,j-1},\;y^{n,1}_{M_1} + a_j \hat{k}_{y,j-1}\right) \;\; , 
\end{align}
The nodes $a_j$ are those used in the increments of a Runge-Kutta solver of order P; i.e for $P=4$, $a_j$ may be given by \eqref{nodes}. \ktwo{For more general Runge-Kutta solvers for PI methods, see \cite{LafitteEtAl14}}. Construction of the microsteps $z^{n,j}_{m}$ is illustrated in Figures~\ref{fig:PI1nM} and \ref{fig:k1hat_M}, and construction of the increments $\hat{k}_{j}$ in Figures~\ref{fig:k1hat} and \ref{fig:k2hat}.

\noindent The macrosolver is then given by the weighted sum
\begin{align}
\label{PI1macro} z^{n+1} &= z^{n,1}_{M_1} + \sum_{j=1}^P b_j \hat{k}_{j}(z^{n}) \;\; ,
%y^{n+1} &= y^{n,1}_{M_1} + \sum_{j=1}^P b_j \hat{k}_{y,j}(x^{n},y^{n}) \;\; , \\
%x^{n+1} &= x^{n,1}_{M_1} + \sum_{j=1}^P b_j \hat{k}_{x,j}(x^{n},y^{n}) \;\; .
\end{align}
%Comparing \eqref{hatky}--\eqref{def:ynjm} with \eqref{rk2}, we see that the increments $\hat{k}_y$ and $\hat{k}_x$ in PI resemble those in the standard Runge Kutta method, except that $M$ microsteps are taken to relax the fast variables to the slow manifold before the $j$-th increment is estimated. 
%
where the weights $b_j$ are appropriate to a Runge-Kutta solver of order P; i.e for $P=4$, $b_j$ may be given by \eqref{weights}. The macrosolver is illustrated in Figure~\ref{fig:PI1sum}. Note that for $M_j=0$ for all $j$, i.e. without the microsolver, the scheme reduces to a standard Runge-Kutta solver of order P applied to the system \eqref{baseneo}. \ktwo{It is not true that PI schemes in general reduce to a numerical discretisation of the underlying multi-scale dynamical system if the microsolver is switched off (see for example \cite{RooseVandekerckhove06}).}

\noindent
For the analysis of the PI1 scheme, it is helpful to explicitly identify the slow and fast variables. We therefore decompose the PI1 variables $z^n$ into fast and slow components $(x^n,y^n)$, and $z^{n,j}_{M_j}$ into the fast and slow components $(x^{n,j}_{M_j},y^{n,j}_{M_j})$. Furthermore, we split the PI1 increments $\hat{k}_j$ into fast components $\hat{k}_{x,j}$ and slow components $\hat{k}_{y,j}$, with
\begin{align}
\label{hatky} \hat{k}_{y,j}(x^{n},y^{n})& = \Dt \; g\left({x}^{n,j}_{M_j},{y}^{n,j}_{M_j} \right) \;\;,\\
\label{hatkx} \hat{k}_{x,j}(x^{n},y^{n})& = \Dt \; f\left({x}^{n,j}_{M_j},{y}^{n,j}_{M_j},\eps \right) \;\; .
\end{align}
The macrosolver is then written as
\begin{align}
\nonumber y^{n+1} &= y^{n,1}_{M_1} + \sum_{j=1}^P b_j \hat{k}_{y,j}(x^{n},y^{n}) \;\; , \\
\label{PI1_macrox} x^{n+1} &= x^{n,1}_{M_1} + \sum_{j=1}^P b_j \hat{k}_{x,j}(x^{n},y^{n}) \;\; .
\end{align}
 
%thesis  \begin{remark}
% The seamless HMM schemes presented in \cite{E03,EEngquist03} also employ \eqref{def:znM}--\eqref{def:znj0}, but the macrosolver \eqref{PI1macro} is instead written as
% \begin{align*}
% z^{n+1} &= z^{n} + \sum_{j=1}^P b_j \hat{k}_{j}(z^{n}) \;\; .
% \end{align*}
% The difference is not significant in the error analysis, as the microsolver used to form $z^{n,1}_{M_1} $ is a highly accurate simulation of the slow dynamics. Consequently, no %additional error accrues from initialising the macrostep from $z^{n,1}_{M_1} $ instead of $z^{n} $.
% \end{remark}
% \begin{remark}
% We do not make any attempt to modify the Runge-Kutta nodes or weights to account for the presence of the microsolver. By contrast, in \cite{GearLee05} the second-order macrosolvers are derived so that they match a Taylor expansion of the true dynamics over one step to second order, including the contributions from the microsolver.
% \end{remark}

%\begin{align*}
%\hat{k}_{x,j}(x,y)& = \left\{\begin{gathered} h(x,y,\eps) \;\; \text{for } j=1 \;, \\ %\frac{1}{\eps}(- \Lambda \, x+ f(y)) \;\; \text{for } j=1 \;, \\
% h(\varphi_x^{M,dt}(x+a_j k_{x,j-1},y+a_j k_{y,j-1}),\varphi^{M,dt}_y(x+a_j k_{x,j-1},y+a_j k_{y,j-1})) %\frac{1}{\eps}(- \Lambda \, \varphi_x^{M,dt}(x+a_j k_{x,j-1},y+a_j k_{y,j-1}) + f(\varphi^{M,dt}_y(x+a_j k_{x,j-1},y+a_j k_{y,j-1})) \;\; \text{for } j=1 \; 
%\end{gathered}\right.
%\end{align*}

\subsection{Projective Integration Scheme PI2}
\label{sec:PI2}
We present here a modification of the PI1 scheme in which the increments are given by differences between endpoints of the microsolver. %In the remainder of the paper, we will establish that this method initialises the fast variables closer to the slow manifold after a macrostep than the PI1 method at the same computational cost, provided the slow manifold is not too nonlinear. Furthermore, we establish that the predicted error between the approximation of the slow variables produced by the PI2 method and the true value of the slow dynamics is the same as that of the PI1 method, and discover situations in which the PI2 method has better error scaling.\\
We again denote by $z^{n,j}_{m}$ the approximation of the fast and slow variables at the $m$-th microstep of the $j$-th increment.
%As in PI1, the scheme starts with an application of the microsolver with $M_1$ steps and time step $\dt$, with
%\begin{align}
%\label{def:2znM}z^{n,1}_{M_1} & = \varphi^{M_1,\dt}(z^n) \;\; , 
%\label{def:ynM}\left(x^{n,1}_{M_1},y^{n,1}_{M_1}\right) & = \varphi^{M_1,\dt}(x^n,y^n) \;\; , 
%\end{align}
%illustrated in Figure~\ref{fig:PI2nM}. We will construct the PI2 increments from the output of the microsolvers in the PI1 method, $z^{n,j}_{m}$, with an additional application of the microsolver to form $z^n_{P+1,m}$. 
The PI1 increments are given by \eqref{hatkz}, which we recall here as
\begin{align}
\label{2hatkz} \hat{k}_{j}(z^{n})& = \Dt \; \mathcal{F} (z^{n,j}_{M_j},\, \eps ) \;\; ,
%\label{hatky} \hat{k}_{y,j}(x^{n},y^{n})& = \Dt \; g\left({x}^{n,j}_{M_j},{y}^{n,j}_{M_j} \right) \;\;,\\
%\label{hatkx} \hat{k}_{x,j}(x^{n},y^{n})& = \Dt \; f\left({x}^{n,j}_{M_j},{y}^{n,j}_{M_j},\eps \right) \;\; ,
\end{align}
%where $M_j \in \mathcal{N}$ is the number of microsteps taken before  estimating the $j$-th increment, and 
for $j=1,2,\dots,P$, where $z^{n,j}_{m}$ is now defined for $j=1,2,\dots,P+1$, $m=1,2,\dots,M_j$, as the output of the microsolver
\begin{align}
\label{def:2znjm} z^{n,j}_{m} &= \varphi^{m,\dt}\left(z^{n,j}_{0}\right) \;\; , 
\end{align}
with initial condition
\begin{align}
\label{def:2znj0} z^{n,j}_{0} &= \twoparts{z^{n}}{j=1}
{z^{n,1}_{M_1} + a_j \hat{k}_{j-1}(z^n)}{j>1} \;\; .
%\label{def:ynjm} \left(x^{n,j}_{m},{y}^{n,j}_{m}\right) &= \varphi^{m,\dt}\left(x^{n,1}_{M_1} + a_j \hat{k}_{x,j-1},\;y^{n,1}_{M_1} + a_j \hat{k}_{y,j-1}\right) \;\; , 
\end{align}
Construction of the microsteps $z^{n,j}_{m}$ is illustrated in Figures~\ref{fig:PI2nM}, \ref{fig:k1} and \ref{fig:k2}. The PI2 increments are constructed by approximating the vector field $\mathcal{F}$ according to $\mathcal{F}(z^{n_j,M_j}) \approx (z^{n,j+1}_{M_{j+1}} -z^{n,1}_{M_1} )/(a_{j+1}\Dt)$, leading to
\begin{align}
\label{def:2kz} k_{j}(z^{n}) &= \frac{1}{a_{j+1}}\left(z^{n,j+1}_{M_{j+1}} -z^{n,1}_{M_1} \right) \;\;, 
%\label{def:ky} k_{y,j}(x^{n},y^{n}) &= \frac{1}{a_{j+1}}\left(y^{n_{j+1},a_{j+1}M} -y^{n,1}_{M_1} \right) \;\;, \\
%\label{def:kx} k_{x,j}(x^{n},y^{n}) &= \frac{1}{a_{j+1}}\left(x^{n_{j+1},a_{j+1}M} - x^{n,1}_{M_1} \right) \;\; ,
\end{align}
%where $M_j \in \mathcal{N}$ is the number of microsteps taken before estimating the $j$-th increment, and 
for $j=1,2,\dots,P$. 
The nodes $a_j$ with $j=1,2,\dots,P$ are again those used in the increments of a Runge-Kutta solver of order P, and we set $a_{P+1}=1$. The construction of the PI2 increments ${k}_{j}$ is illustrated in Figures~\ref{fig:k1} and \ref{fig:k2}.

\noindent
Each PI2 increment covers a time step of $\Dt + {M_{j+1}\dt}/{a_{j+1}}$. We fix the total number of microsteps $M_j$ for $j>1$ with 
\begin{align}
\label{Mj} M_j = a_j M \;\;, 
\end{align}
for $j=1,2,\dots,P+1$ and for some $M$ satisfying $a_j M \in \mathbb{N}$, so that each increment covers a uniform time step of $\Dt + M\dt =: \td$. Note that \eqref{Mj} allocates more microsteps after larger increments and less after shorter increments.
 
%For comparitive purposes we rewrite the increments from PI1 as
%\begin{align*}
%\hat{k}_{y,j}(x^{n},y^{n}) &= \frac{1}{a_{j+1}}\left(y^{n_{j+1},0} -y^{n,1}_{M_1} \right) \;\;, \\
%\hat{k}_{x,j}(x^{n},y^{n}) &= \frac{1}{a_{j+1}}\left(x^{n_{j+1},0} - x^{n,1}_{M_1} \right) \;\; ,
%\end{align*}
%to illustrate that PI2 employs the microsteps $y^{n_{j+1},m}$, $x^{n_{j+1},m}$ which PI1 uses only to estimate the $(j+1)$-th increment, to further relax the $j$-th increment to the slow manifold. 
%The use of $M_j = a_{j}M$ microsteps in \eqref{def:kz} asserts that the time step covered during each increment in PI2 is $\Dt+M\dt$. \\
\noindent
The macrosolver is now constructed as a weighted sum over the relaxed increments $k_j$ rather than over $\hat{k}_j$, with
\begin{align}
\label{macroz} z^{n+1} &= z^{n,1}_{M_1} + \sum_{j=1}^P b_j {k}_{j}(z^{n}) \;\; ,
%\label{macroy} y^{n+1} &= y^{n,1}_{M_1} + \sum_{j=1}^P b_j {k}_{y,j}(x^{n},y^{n}) \;\; , \\
%\label{macrox} x^{n+1} &= x^{n,1}_{M_1} + \sum_{j=1}^P b_j {k}_{x,j}(x^{n},y^{n}) \;\; .
\end{align}
where the weights $b_j$ again correspond to a Runge-Kutta solver of order P. Note that $t^n=n(\td+M_1\dt)$ for PI2. The macrosolver is illustrated in Figure~\ref{fig:PI2sum}.

\noindent
%The 'standard' and 'relaxed' definitions of the PI increments, \eqref{hatky}--\eqref{hatkx} and \eqref{def:ky}--\eqref{def:kx} respectively, are illustrated graphically in figures \ref{fig:c1} and \ref{fig:c2}. \\

%The equations \eqref{ynM}--\eqref{xnM}, \eqref{macroy}--\eqref{kx} together with \eqref{ky}--\eqref{kx} define our proposed scheme.
\noindent
Again for $M_1 =M=0$, i.e. without the microsolver, the PI2 scheme reduces to a standard Runge-Kutta method of order P. 

\noindent
As with the PI1 scheme, it is helpful to explicitly identify the slow and fast variables in the solver. We therefore decompose the PI2 variables $z^n$ into fast and slow components $(x^n,y^n)$, and $z^{n,j}_{M_j}$ into the fast and slow components $(x^{n,j}_{M_j},y^{n,j}_{M_j})$, and we split the PI2 increments ${k}_j$ into fast components ${k}_{x,j}$ and slow components ${k}_{y,j}$, with
\begin{align}
\label{def:ky} k_{y,j}(x^{n},y^{n}) &= \frac{1}{a_{j+1}}\left(y^{n,j+1}_{M_{j+1}} -y^{n,1}_{M_1} \right) \;\;, \\
\label{def:kx} k_{x,j}(x^{n},y^{n}) &= \frac{1}{a_{j+1}}\left(x^{n,j+1}_{M_{j+1}} - x^{n,1}_{M_1} \right) \;\; ,
\end{align}
depending via \eqref{def:2znj0} on the PI1 increments $\hat{k}_j$. For completeness we recall these as
\begin{align}
\label{2hatky} \hat{k}_{y,j}(x^{n},y^{n})& = \Dt \; g\left({x}^{n,j}_{M_j},{y}^{n,j}_{M_j} \right) \;\;,\\
\label{2hatkx} \hat{k}_{x,j}(x^{n},y^{n})& = \Dt \; f\left({x}^{n,j}_{M_j},{y}^{n,j}_{M_j},\eps \right) \;\; .
\end{align}
The macrosolver is then given by
\begin{align}
\nonumber y^{n+1} &= y^{n,1}_{M_1} + \sum_{j=1}^P b_j k_{y,j}(x^{n},y^{n}) \;\;  \\
\label{macrox} x^{n+1} &= x^{n,1}_{M_1} + \sum_{j=1}^P b_j k_{x,j}(x^{n},y^{n}) \;\; .
\end{align}
For ease of exposition we write the slow dynamics as
\begin{align*}
y^{n+1} &= y^n + \tilde{g}(x^n,y^n) \;\; , \\
\end{align*}
where the vector field of the slow variables in the PI2 macrosolver, $\tilde{g}$, is given by
\begin{align}
\label{def:gtwiddle} \tilde{g}(x^n,y^n) &= y^{n,1}_{M_1} - y^n + \sum_{j=1}^P b_j {k}_{y,j}(x^{n},y^{n}) \;\;.
\end{align}

\noindent
Comparing Figures~\ref{fig:PI1} and \ref{fig:PI2}, in the PI2 method the increments point in the approximate direction of the slow manifold, so that the macrostep initialises the fast variables close to the slow manifold after a macrostep. By comparison, the PI1 increments can depart from the slow manifold with larger scale separations or for initial conditions off the slow manifold. 

%%%%%%%%%%%%%%%%%%%%%%%%%%%%%%%%%%%%%%%%%%%%%%%%%%%%%%%%%%%
%PI1 figures
\begin{figure}[hb]
              \centering
        \begin{subfigure}[b]{.5\textwidth}
                \centering
\includegraphics[width=\textwidth]{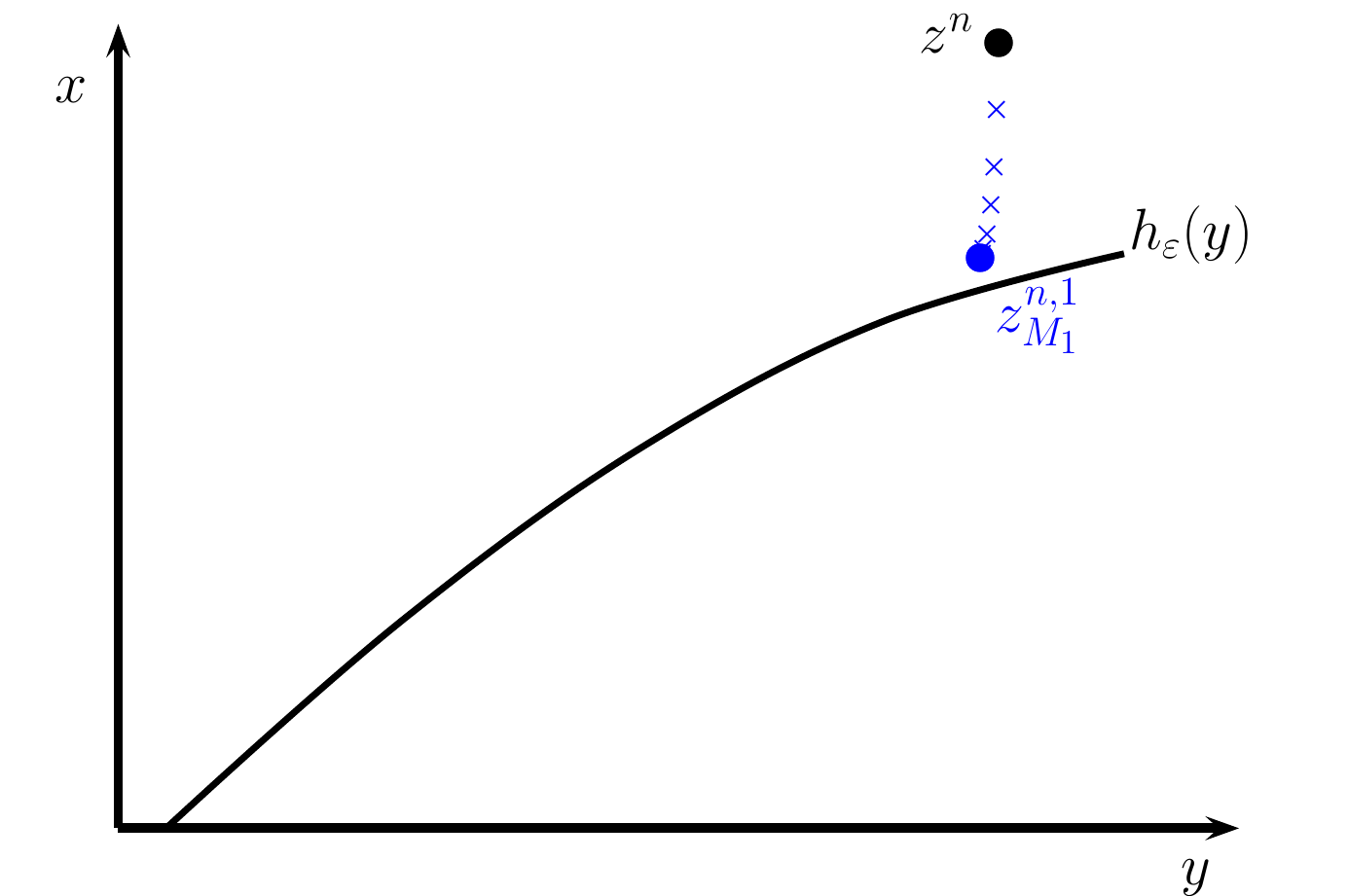}
\vspace{-20pt}
\caption{ }
\label{fig:PI1nM}
       \end{subfigure}%
  %      ~ %add desired spacing between images, e. g. ~, \quad, \qquad etc.
%\end{figure}
%\begin{figure}
        \begin{subfigure}[b]{.5\textwidth}
                \centering
\includegraphics[width=\textwidth]{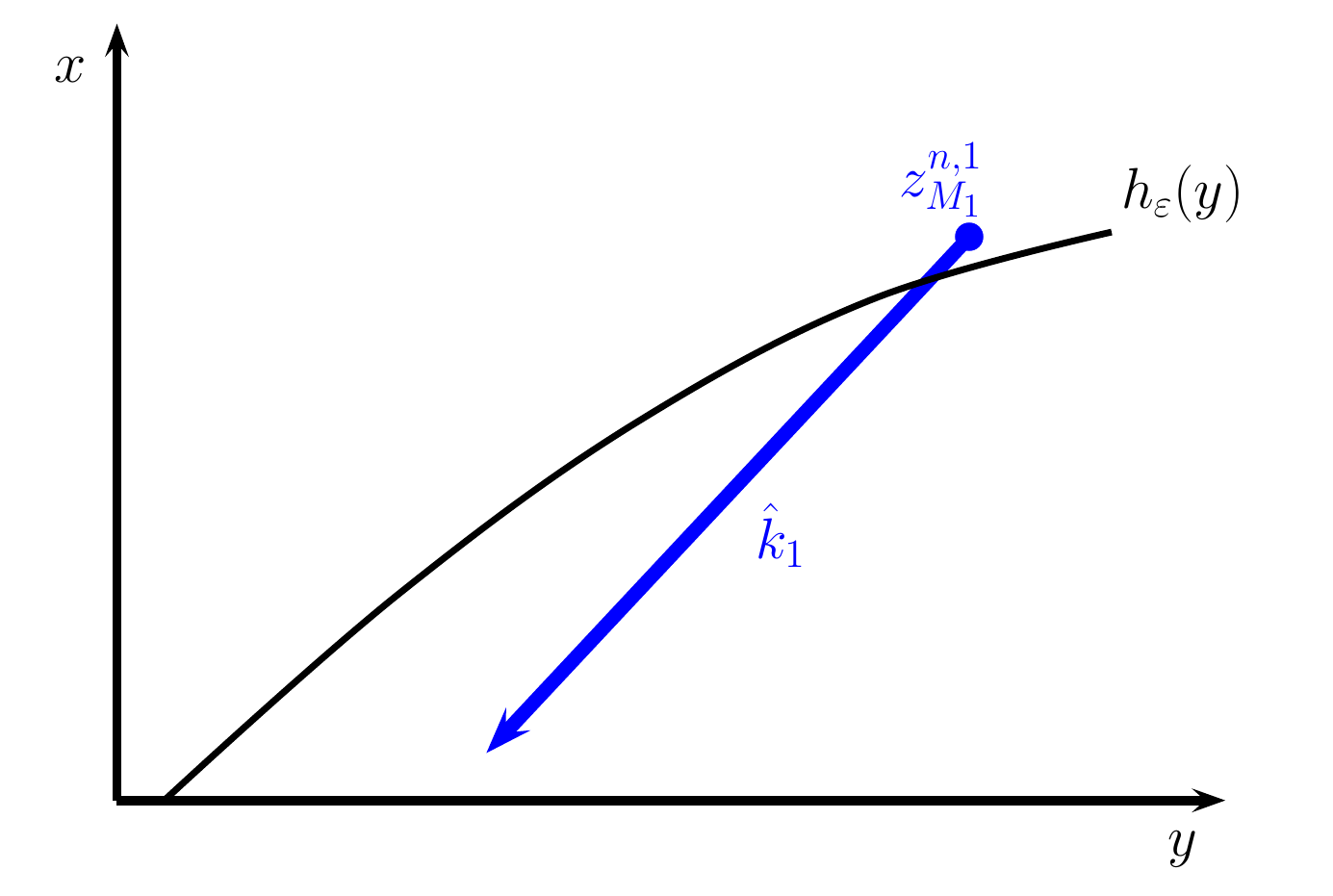}
\vspace{-20pt}
\caption{ }
\label{fig:k1hat}
        \end{subfigure}

        \begin{subfigure}[b]{.5\textwidth}
                \centering
\includegraphics[width=\textwidth]{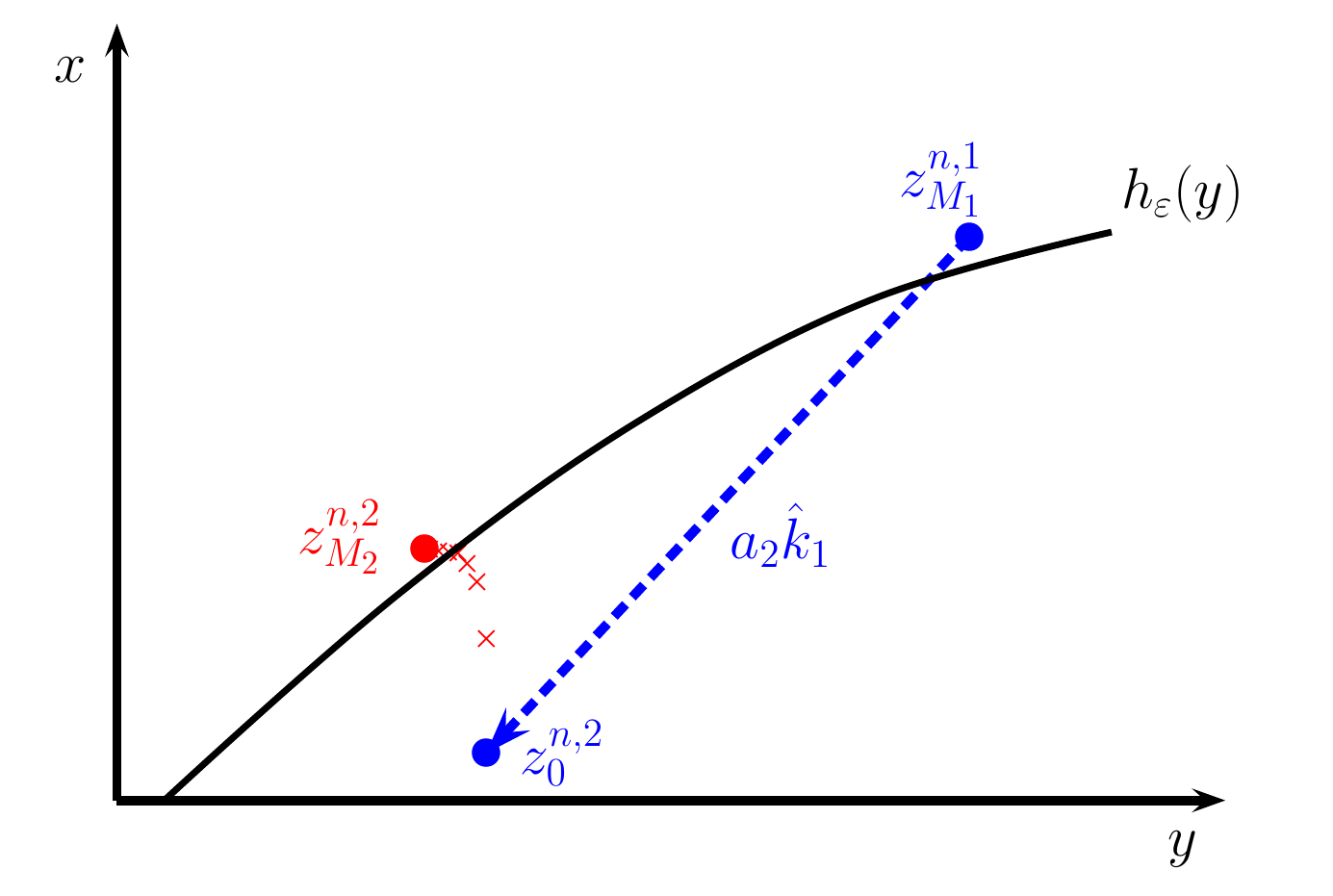}
\vspace{-20pt}
\caption{ }
\label{fig:k1hat_M}
        \end{subfigure}%
                \begin{subfigure}[b]{.5\textwidth}
                \centering
\includegraphics[width=\textwidth]{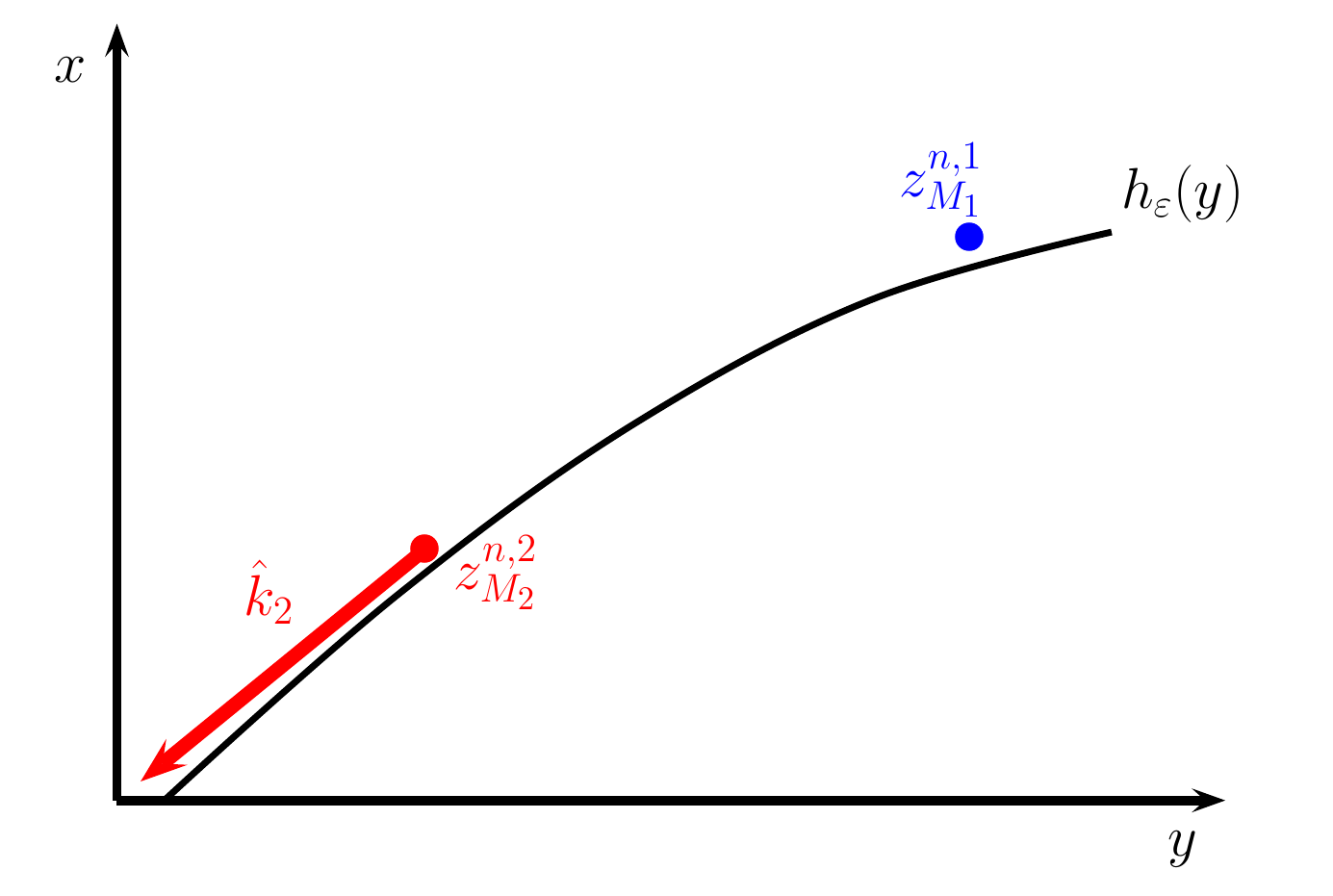}
  \vspace{-20pt}
 \caption{ }
  \label{fig:k2hat}
        \end{subfigure}

                \begin{subfigure}[b]{.5\textwidth}
                \centering
\includegraphics[width=\textwidth]{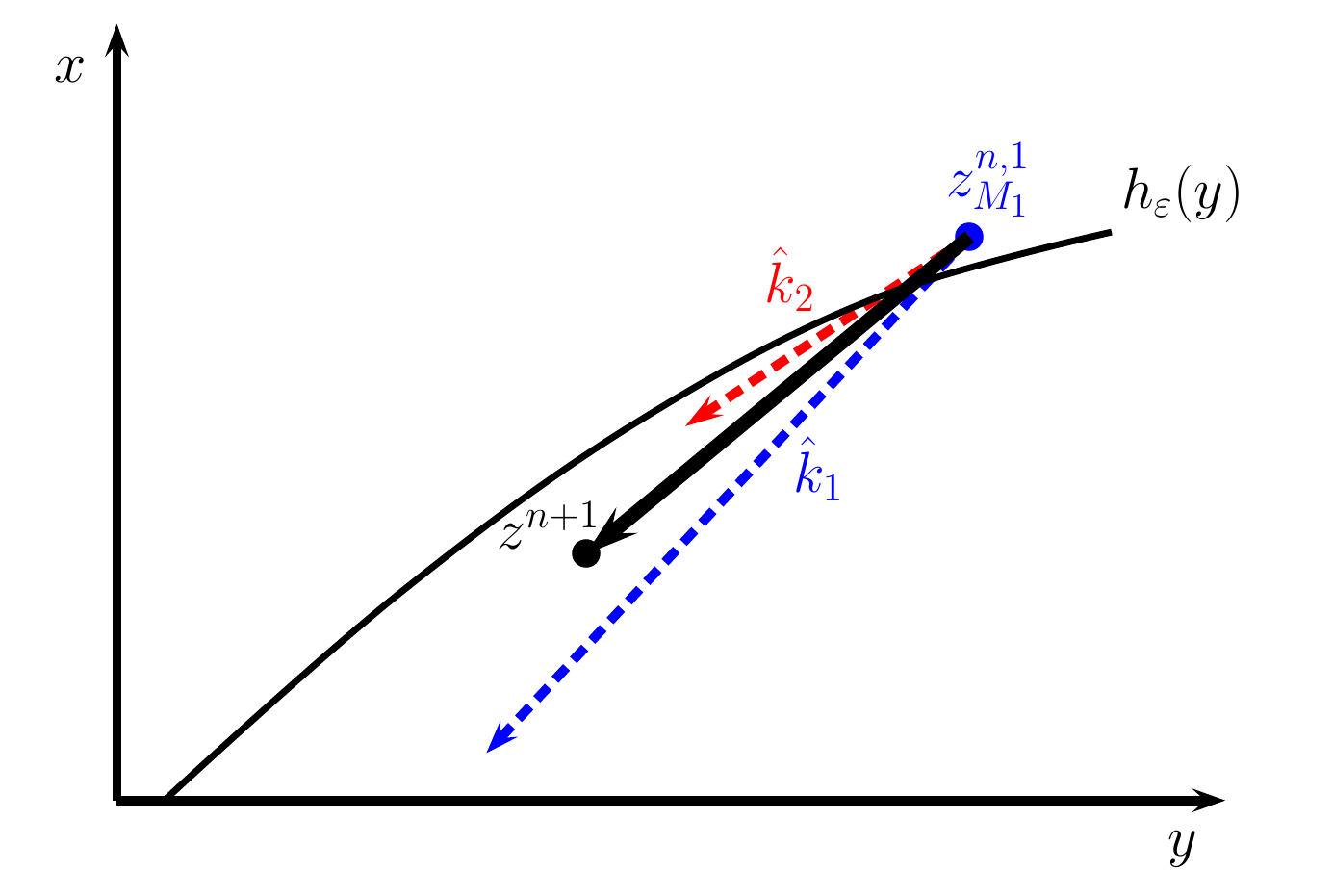}
 \vspace{-20pt}
  \caption{ }
  \label{fig:PI1sum}
        \end{subfigure}
        \caption{Sketch of the PI1 scheme for a second order Runge-Kutta macrosolver. The microsolver is employed in \ref{fig:PI1nM}, \ref{fig:k1hat_M}, the increments $\hat{k}_j$ are given by vectorfield evaluations in \ref{fig:k1hat}, \ref{fig:k2hat}, and the macrosolver is illustrated in \ref{fig:PI1sum}. For this scheme $a_1=0$, $a_2 = 1$ and $b_1=b_2={1}/{2}$.}
                 \label{fig:PI1}
\end{figure}

%%%%%%%%%%%%%%%%%%%%%%%%%%%%%%%%%%%%%%%%%%%%%%%%%%%%%%%%%%%%%%%%%%%%%%%%%%%%%%%%%%
%PI2 figures
\begin{figure}[hb]
              \centering
        \begin{subfigure}[b]{.5\textwidth}
                \centering
\includegraphics[width=\textwidth]{fig_n_nM}
\vspace{-20pt}
\caption{ }
\label{fig:PI2nM}
       \end{subfigure}%
    \begin{subfigure}[b]{.5\textwidth}
                \centering
\includegraphics[width=\textwidth]{fig_k1hat}
\vspace{-20pt}
\caption{ }
\label{fig:k1hat.2}
        \end{subfigure}

        \begin{subfigure}[b]{.5\textwidth}
                \centering
\includegraphics[width=\textwidth]{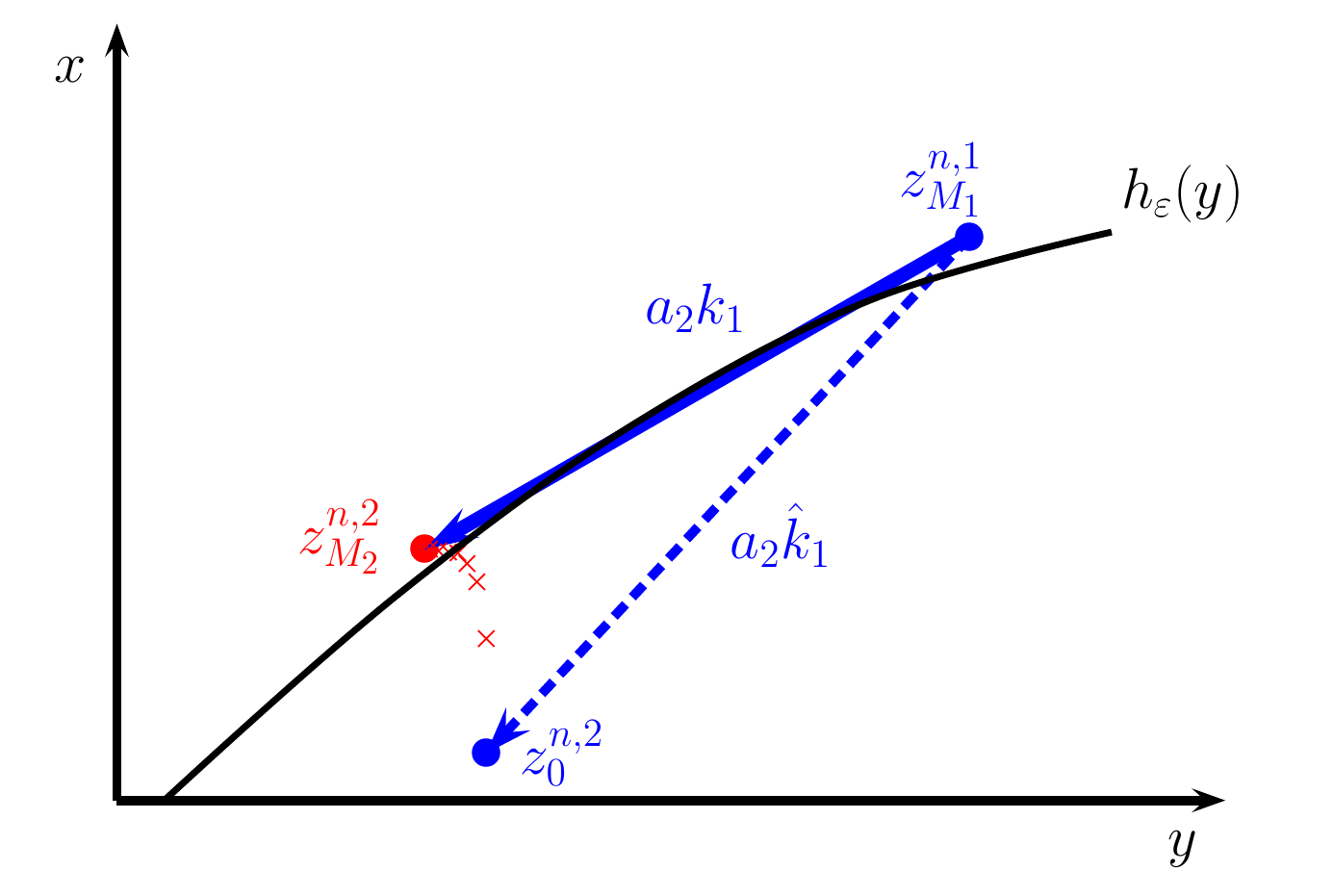}
\vspace{-20pt}
\caption{ }
\label{fig:k1}
        \end{subfigure}%
       \begin{subfigure}[b]{.5\textwidth}
                \centering
\includegraphics[width=\textwidth]{fig_k2hat}
  \vspace{-20pt}
 \caption{ }
  \label{fig:k2hat.2}
        \end{subfigure}
        
                \begin{subfigure}[b]{.5\textwidth}
                \centering
\includegraphics[width=\textwidth]{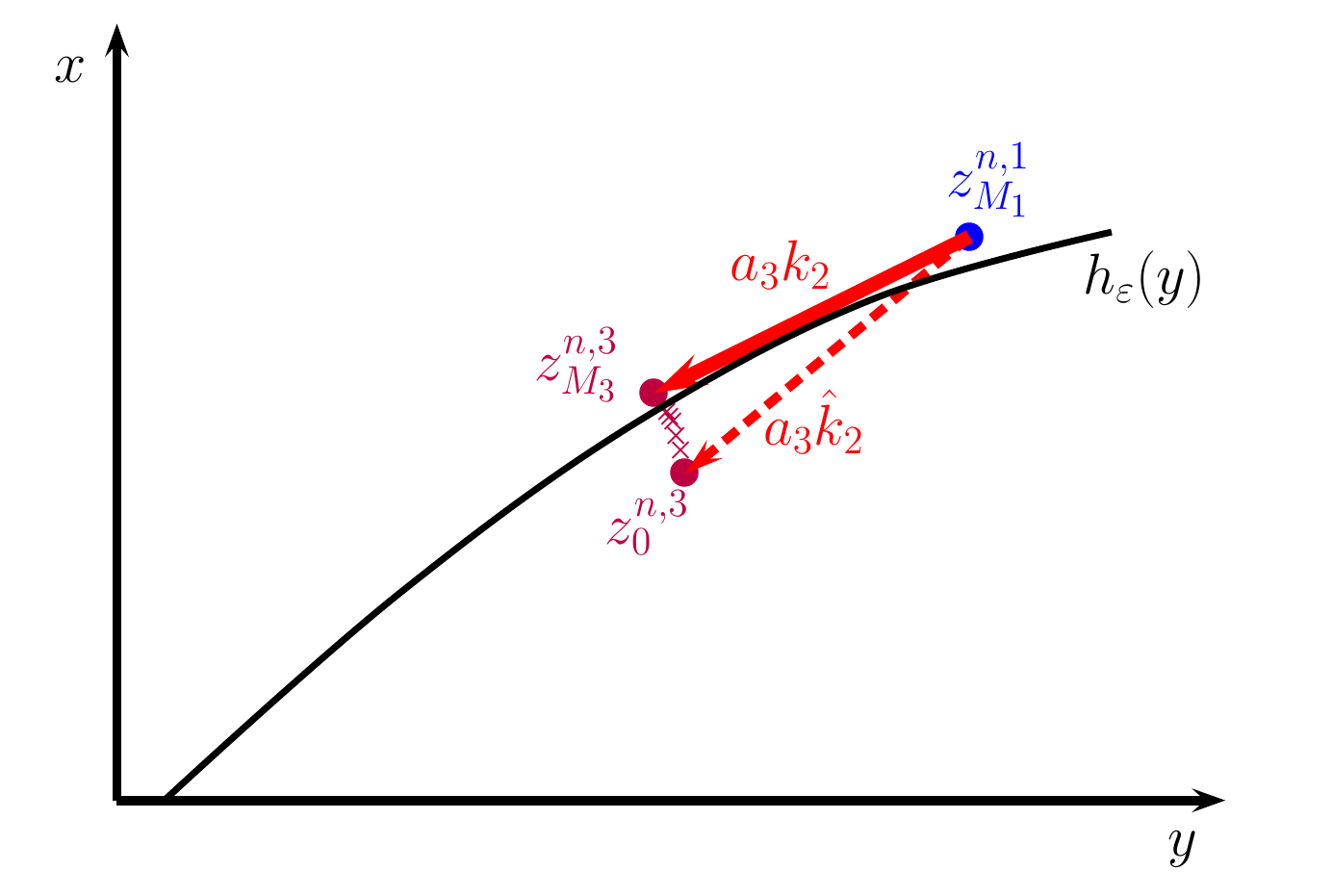}
  \vspace{-20pt}
 \caption{ }
  \label{fig:k2}
        \end{subfigure}%
                \begin{subfigure}[b]{.5\textwidth}
                \centering
\includegraphics[width=\textwidth]{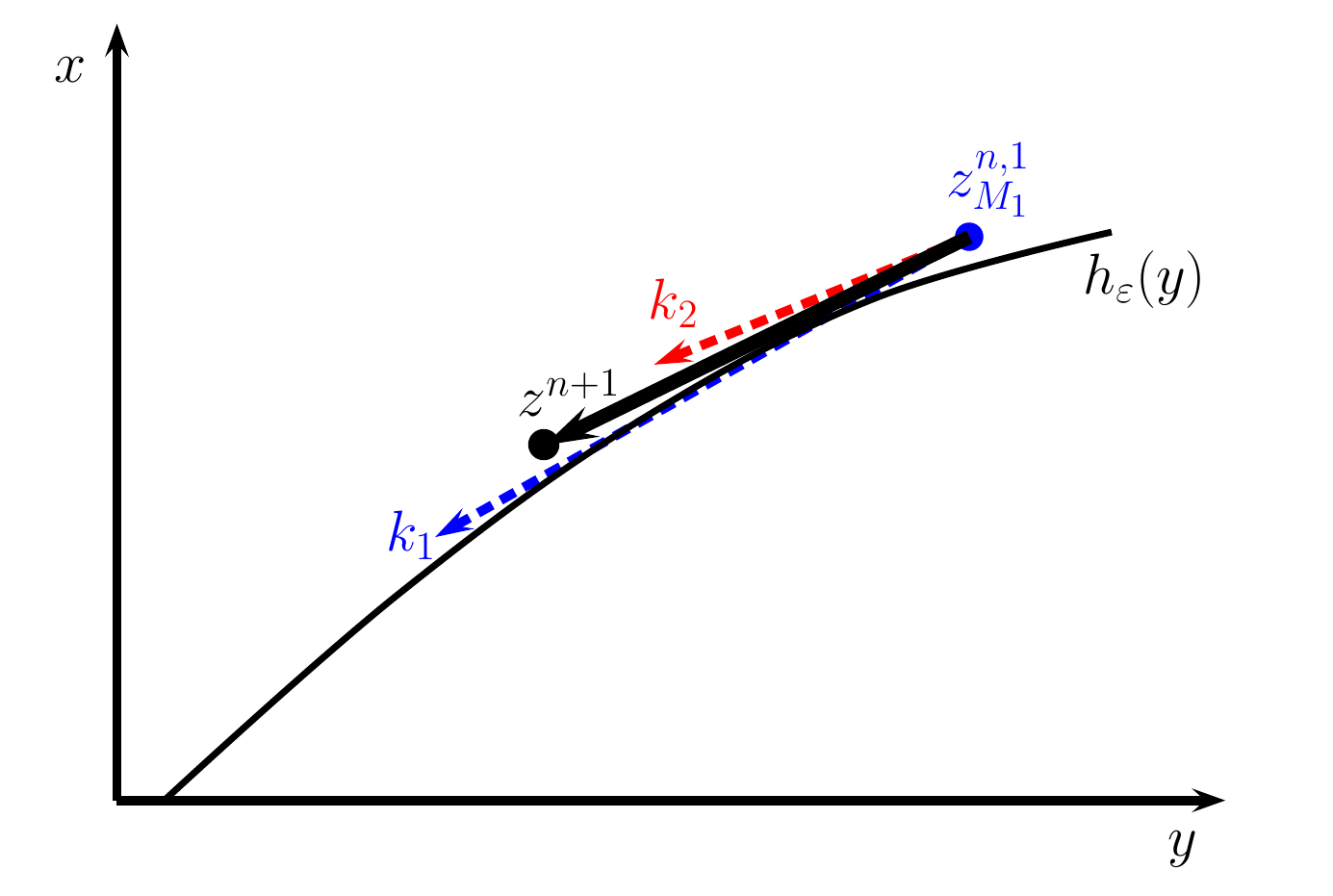}
\vspace{-20pt}
   \caption{ }
  \label{fig:PI2sum}
        \end{subfigure}
        \caption{Sketch of the PI2 scheme for a second order Runge-Kutta macrosolver. The microsolver is employed in \ref{fig:PI2nM}, \ref{fig:k1}, \ref{fig:k2} and the (now auxiliary) quantities $\hat{k}_j$ are given by vector field evaluations in \ref{fig:k1hat.2}, \ref{fig:k2hat.2}. The increments $k_j$ are estimated in \ref{fig:k1}, \ref{fig:k2}, and the macrosolver is illustrated in \ref{fig:PI2sum}. For this scheme $a_1=0$, $a_2 = a_3 = 1$ and $b_1=b_2={1}/{2}$.}
                 \label{fig:PI2}
\end{figure}

\section{Error analysis for Projective Integration}
\label{sec:proof}

We provide rigorous error bounds for the slow variables of PI in the formulations PI1 (\ref{hatkz})--(\ref{PI1macro}) and PI2 (\ref{def:2znjm})--(\ref{macroz}), following the general line of proof used in \cite{E03}. Therein the result for PI1 was stated, albeit without explicit proof. Furthermore, we establish bounds on the departure of the fast variables from the slow manifold over the macrosolver, yielding stability conditions for the fast variables.\\%, which to the best of our knowledge has not been done elsewhere. \\%The proof for PI1 is similar and is deferred to Appendix \ref{sec:PI1}.\\

\noindent
Throughout this work we assume the following conditions on the growth and smoothness of solutions of our system and on the numerical discretization parameters of PI.\\
%
%\begin{assumptions}

\noindent
{\bf Assumptions}
\vspace{-4mm}
\begin{itemize}
\renewcommand{\labelitemi}{}
\item 
\renewcommand{\labelitemi}{\it A1:}
 \item The zeroth order approximation of the slow manifold $h_0(y)$ is Lipschitz continuous; that is there exists a constant~$L_h$ such that 
\begin{align*}
|h_0(y_1)-h_0(y_2)| \leq L_h|y_1 - y_2|\; .
\end{align*}
\renewcommand{\labelitemi}{\it A2:}
\item The vectorfield $g(x,y)$ is Lipschitz continuous; that is there exists a constant~$L_g$ such that 
\begin{align*}
|g(x_1,y_1)-g(x_2,y_2)| \leq L_g(|x_1-x_2|+|y_1-y_2|)\; .
\end{align*}
\renewcommand{\labelitemi}{\it A3:}
\item The second order derivatives of $h_0$ are all bounded; that is there exists a constant~$L_{h'}$ such that 
\begin{align*}
\sup_{}\Big|\sum_{|\alpha|=2}{\partial^\alpha h_0(y)}{}\Big| \leq L_{h'}\; ,
\end{align*}
where we used multi-index notation.% to represent all second order derivatives of $h_0$.
%thesis with respect to the vector $y \in\mathbb{R}^n$. Here $\alpha= (\alpha_1, \alpha_2,\, \dots\,, \alpha_n)$ is a vector of non-negative integers, $|\alpha| = \alpha_1 + \alpha_2 + \dots + \alpha_n$, and $\partial^\alpha$ is given by the product of the partial derivatives
%\begin{align*}
%\partial^\alpha  &= \partial^{\alpha_1}_1 \partial^{\alpha_2}_2 \dots \partial^{\alpha_n}_n \; .
%\end{align*}
\renewcommand{\labelitemi}{\it A4:}
\item The vectorfield $g(x,y)$ is bounded for all $x,y$; that is there exists a constant~$C_g$ such that 
\begin{align*}
C_g =  \sup|g(x,y)| \;.
\end{align*}
\renewcommand{\labelitemi}{\it A5:}
\item The reduced slow dynamics $Y(t)$ is of class $C^{\max(P,\,p)}$; that is there exist constants~$C_P^*$~and~$C_p^*$ such that 
\begin{align*}
C_P^* &= \sup\left|\frac{\d^P{Y}(t)}{\d t^P}\right| \;, \\
C_p^* &= \sup\left|\frac{\d^p{Y}(t)}{\d t^p}\right| \;, 
\end{align*}
and in particular there exists a constant $C_2^*$ satisfying
\begin{align*}
C_2^* = \sup|\ddot{Y}(t)| \;.
\end{align*}
%\renewcommand{\labelitemi}{\it A6:}
%\item The microstep size $\dt$ resolves the fastest of the fast variables, while the macrostep size~$\Dt$ resolves the dynamics of the slow system, so that
%\begin{align*}
%0 < \dt \le \frac{2\eps}{\lambda}  < \Dt \; .
%\end{align*}
\renewcommand{\labelitemi}{\it A6:}
\item The total time $\Dt$ of the macrostep is sufficiently short so that, employing the practical constraint $M\dt \le \Dt$,
\begin{align*}
 L_G M \dt \leq L_G \Dt < \frac{1}{2}\;.
\end{align*}
%where $L_G \leq L_g(1+L_f)$ is the Lipschitz constant of the reduced dynamics \eqref{e.CMT}.\\
%\renewcommand{\labelitemi}{\it A8:}
%\item The macrostep size $\Dt$, number of microsteps $M$ and  microstep size $\dt$ are chosen such that 
%\begin{equation*}
%                   \begin{array}{lll}
%			&\displaystyle\Dt  \, \exp\left(-\frac{M\dt}{\eps}\right) < \frac{\eps}{\lambda}  &\qquad\text{if \;\;} \displaystyle0<\dt\le\frac{2\eps}{\lambda+1} \;,\\ 
%			&&or \\
%		&	\displaystyle \Dt  \, \exp\left(-\frac{M\dt^\star}{\eps}\right) < \frac{\eps}{\lambda} & \qquad \text{if \;\;} \displaystyle\frac{2\eps}{\lambda+1}<\dt<\frac{2\eps}{\lambda}\; .
%		\end{array}
%\end{equation*}
%with $\dt^\star=2\varepsilon-\lambda \dt$. The range $2\eps/(\lambda+1) < \dt<2\eps/\lambda$ corresponds to $0<\dt^\star<2\varepsilon/(\lambda+1)$. This assumption is %necessary to bound the distance of the fast variables from the slow manifold over the macrosteps.\\
\end{itemize}
%\end{assumptions}

\begin{remark}
Assumptions ({\it A1})--({\it A2}) imply that the reduced slow dynamics \eqref{e.CMT} is also Lipschitz continuous and there exists a constant $L_G \le L_g (1 + L_h)$ such that
\begin{align*}
|G(Y_1)-G(Y_2)| \leq L_G|Y_1 - Y_2|\; .
\end{align*}
Assumption ({\it A4}) implies that the reduced slow dynamics is also bounded and there exists a constant $C_G \le C_g$ such that
\begin{align*}
C_G = \sup|G(Y)| \; .
\end{align*}
\end{remark}

\noindent
The global Lipschitz conditions can be relaxed to local Lipschitz conditions by the usual means.\\
% Note that for the weights (\ref{W_HMM0}) or (\ref{e.WPI}) we have that $\tilde{g}$ is bounded by $C_g$.\\

\noindent
We will establish bounds for the error $\mathcal{E}^n$ between the PI1 and PI2 estimate $y^n$ and the solution of the full system $y_\eps(t^n)$,
\begin{align*}
{\mathcal{E}}^n = |y_\eps(t^n)-y^n|\; .
\end{align*}
%Our main result is provided by the following two Theorems, in which we establish a bound on the error between the PI1 and PI2 approximations of the slow variables $y^n$ and the true dynamics $y_\eps(t^n)$, and establish stability conditions for the PI1 and PI2 methods by bounding the distance of the fast variables from the slow manifold over the macrosteps.
\medskip
\begin{theorem}[Convergence]
\label{theorem.main}
Consider schemes PI1 and PI2 run with a Runge-Kutta method of order $P$ for the macrosolver and an explicit scheme of order $p$ for the microsolver. Given assumptions ({\it A1})--({\it A6}), there exists a constant $C$ such that  on a fixed time interval $T$, for each $n$ such that $n\td\le T$, the error between the PI1 and PI2 estimates and the exact solution of the full multiscale system (\ref{baseneo}) are bounded by
\begin{align*}
{\mathcal{E}}^n \le  C\left(
\td^P + M\dt + \varepsilon
+\left(\frac{\eps}{\td}+ \rho^{aM}\left(-\frac{\dt}{\varepsilon}\right) \right)|d^n_{\max}| 
\right) \;\; .
% \twoparts
%{\quad \;\;\;\;C\left(
%\td + \varepsilon
%+\left(\frac{\varepsilon}{\td}+\eps+e^{-\frac{M\dt}{\varepsilon}} \right)|d^n| 
%\right)}
%{C\frac{\dt}{\dt^*}\left(
%\td + \varepsilon
%+\left(\frac{\varepsilon}{\td}+\eps+e^{-\frac{M\dt^*}{\varepsilon}} \right)|d^n| 
%\right)}\;\; ,
\end{align*}
Here $a = \min_{j>1}a_j$, $\rho$ is the linear amplification factor \eqref{linamp} for the microsolver of order p measuring the attraction of the fast variables to the slow manifold over a microstep, and $\left|d^n_{\max}\right| := \max_{\substack{0 \le i <n\;, \\ 1 \le k \le P+1}} |x^{i,k}_{0} - h_0(y^{i,k}_{0})| $ is the maximal deviation of the fast variables from the approximate slow manifold accrued over the integration time.\end{theorem}

\medskip
It is worthwhile to briefly discuss the bound on $\mathcal{E}^n$. The term proportional to $\td^P$ reflects the convergence of the underlying Runge-Kutta numerical scheme of order P in the macrosolver. The term proportional to $M\dt$ is incurred by the drift of the slow variables over the microsteps before estimating the increments (regardless of the order p of the microsolver). The terms proportional to the time scale parameter $\varepsilon$ represent the error made by the reduction as well as an additional error incurred during the drift of the slow variable over the microsteps. The term proportional to $(\eps/\td + \rho^{aM}(-\dt/\varepsilon))|d^n_{\max}|$ measures the mismatch between the slow vector field $g(x,y)$  after an application of the microsolver and the reduced vector field $G(y)$.\\

\noindent
We also provide bounds on the deviation $|d^n|$ of the fast variables $x^n$ from the slow manifold $h_0(y^n)$ for PI1 and PI2, 
\begin{align*}
|d^n| = |x^n-h_0(y^n)|\; .
\end{align*}
\medskip
\begin{theorem}[Stability of the fast variables] 
\label{theorem.sec}
Consider schemes PI1 and PI2 run with a Runge-Kutta method of order $P$ for the macrosolver and a forward Euler scheme for the microsolver. Given assumptions {\it (A1)}, {\it (A3)} and {\it (A4)}, the fast variables do not diverge over the macrosolver, so that the largest deviation of the fast variables from the slow manifold $|d^n_{\max}|$ is finite, if 
\begin{align*}
\frac{\lambda \Dt}{\eps}\left(1-\frac{\dt}{\eps}\right)^{a M} < 1 \;\;.
\end{align*}
%and for the PI2 scheme
%\begin{align*}
%\frac{\left(1-\frac{\dt}{\eps}\right)^{aM}}{a} \left(\frac{\lambda \Dt}{\eps} \left(1-\frac{\dt}{\eps}\right)^{aM}\right)^j <1 \;\;\; \forall \; j.
%\end{align*} 
Then the distance of the fast variables from the slow manifold after the $n$-th macrostep
%, given by $|d^{n+1}| = |x^{n+1} - h_0(y^{n+1})|$, 
satisfies for the PI1 scheme the recurrence relation
\begin{align*}
|d^{n+1}| \le& \sum_{j=1}^P b_j \left( \frac{\lambda \Dt}{\eps} \left(1-\frac{\dt}{\varepsilon}\right)^{aM}\right)^j |d^{n}|  + L_h C_g (1+\lambda)\Dt \;\;, %\sum_{j=0}^{P-1} \left( \frac{\lambda\Dt }{\eps}\!\! \left(1-\frac{\dt}{\varepsilon}\right)^{aM} \right)^j\\
%&  +\mathcal{O}\left(\eps \, ,\, \left(1-\frac{\dt}{\varepsilon}\right)^{M_1} \left|d^{n}\right|\right) \;\;,
\end{align*}
and for the PI2 scheme
\begin{align*}
|d^{n+1}| \le &  \ \frac{\left(1-\frac{\dt}{\eps}\right)^{aM}}{a}\sum_{j=1}^P b_j \left(\frac{\lambda \Dt}{\eps} \left(1-\frac{\dt}{\eps}\right)^{aM}\right)^j |d^{n}|+2L_{h'} C_g^2 \td^2\;\;.
 %&+ L_h C_g (1+\lambda)\Dt \sum_{j=0}^{P-1} \left( \frac{\Dt \lambda}{\eps}\left(1-\frac{\dt}{\eps}\right)^{aM} \right)^j\Bigg] \\
 %& +\mathcal{O}\left(\eps \, ,\, \left(1-\frac{\dt}{\eps}\right)^{M_1} \left|d^{n}\right|\right)  \;\; .
\end{align*}
\end{theorem}
\ktwo{\begin{remark}
The stability condition for the fast variables of the PI1 and PI2 methods is identical to the corresponding stability condition for an Euler macrosolver given by Assumption 8 in \cite{GottwaldMaclean13} (see also \cite{LafitteEtAl14}).
\end{remark}}
We note that Theorem~\ref{theorem.sec} can be formulated for a microsolver of order $p>1$ but, as we shall see, optimal convergence results are given by a forward Euler microsolver. \\

We briefly discuss the stability condition and the bounds for $|d^{n+1}|$ established above. The stability condition can be understood as follows: $\left(1-{\dt}/{\eps}\right)^{a M}$ denotes the exponential contraction of the fast variables towards the slow manifold during the application of the microsolver; if this contraction rate does not bring the fast variables within a neighbourhood of $\epsilon/\lambda$ of the slow manifold, the fast variables will not have sufficiently relaxed and their dynamics remains stiff, possibly causing numerical instability over the subsequent integration steps.\\
%In contrast to the error bounds for the slow variables established in Theorem~\ref{theorem.main}, 
The bounds for the deviation of the fast variables from the slow manifold are different for PI1 and PI2. In particular, Theorem~\ref{theorem.sec} suggests that for a given macrostep size the fast variables deviate less from the slow manifold in our modified version PI2. This will be confirmed numerically in Section~\ref{sec:numerics}.\\ 

%We briefly discuss the stability conditions and bounds for $|d^{n+1}|$ established above, beginning with the PI1 scheme. The stability conditions for PI1 ensure that the fast variables converge to within $\mathcal{O}(\eps/(\lambda\Dt))$ of the slow manifold after an application of the microsolver, so that there is effectively no scale separation between the fast and slow variables for the large time steps of each increment. In the recursive estimate for $|d^{n+1}|$, the term proportional to $|d^n|$ measures the net change in the distance of the fast variables from the slow manifold over the increments in a macrostep; the additional term proportional to the PI1 macrostep $\Dt$ is incurred by the drift of the slow variable over the macrostep. For the PI2 scheme, the stability condition contains a factor ${\left(1-{\dt}/{\eps}\right)^{aM}}/a\ll1$, which reflects the increased stability from defining the PI2 increments as a difference between applications of the microsolver rather than as vector field evaluations. The recursive estimate for $|d^{n+1}|$ is now accurate to second order in the PI2 macrostep $\td$, reflecting that the PI2 increments anticipate the direction of the slow manifold.\\

%contains a factor $\left(1-{\dt}/{\varepsilon}\right)^{aM}$ representing the convergence of the fast variables to the slow manifold over the microsteps, and a factor $\lambda\Dt/\eps$ representing the departure of the fast variables from the slow manifold over the increments. 

\noindent
In the next section we prove Theorems~\ref{theorem.main} and \ref{theorem.sec}. We formulate the proofs for PI2 and point out where and how they will differ for PI1.
%{ \color{blue} At the larger values of $\dt$ with $\frac{2\eps}{\lambda+1}<\dt<\frac{2\eps}{\lambda}$, the fast variables converge towards the slow manifold with rate $\exp(-M\dt^*/\eps)$, which is a product of the forward Euler scheme used for the microsolver.}

%%%%%%%%%%%%%%%%%%%%%%%%%%%%%%%%%%%%%%%

\subsection{Error Analysis}
We split the error $\mathcal{E}^n$ between the PI approximation of the slow variables and their true value into two parts. Denote by $Y(t^n)$ the time-continuous solution of the reduced ordinary differential equation (\ref{e.CMT}) evaluated at time $t^n$, then
\begin{align*}
{\mathcal{E}}^n &= |y_\eps(t^n)-y^n|\\
&\le |y_\eps(t^n)-Y(t^n)| + |y^n-Y(t^n)|
\; ,
\end{align*}
where the first term describes the error between the exact solutions of the full system (\ref{baseslow})-(\ref{basefast}) and the reduced slow system (\ref{e.CMT}), which we label {\emph{reduction error}}, with
\begin{align}
E^n_r = y_\eps(t^n)-Y(t^n)\; ,
\end{align}
and the second term the error between PI and the exact solution of the reduced slow system (\ref{e.CMT}), which we label {\emph{discretization error}}, with
\begin{align}
E_d^n = y^n-Y(t^n)\; .
\end{align}
We will bound the two terms separately in the following. 

%%%%%%%%%%%%%%%%%%%%%%%%%%%%%%%%%%%%%%%

\subsection{Reduction error}
Setting the initial conditions close to the slow manifold with $y_\eps(0)=Y(0)+c_{0,y}\,\eps$ and $x_\eps(0)=h_\eps(y_\eps(0))+c_{0,x}$, we formulate the following theorem for the reduction error $E^n_r$.
\medskip
\begin{theorem}
\label{theorem.Ec}
Given assumptions ({\it A1})--({\it A3}), there exists a constant $C_1$ such that on a fixed time interval $T$, for each $t^n\le T$, the difference between the exact solutions of the reduced and the full system is bounded by
\begin{align*}
|E^n_r| \le C_1 \varepsilon \; ,
\end{align*}
with 
\[
C_1=
{\rm{max}}\left(|y_\eps(0)-Y(0)|,L_g|d_\eps(0)|) \right) 
e^{L_g\left(1+L_h\right)t^n}
\; ,
\]
where $d_\eps=x_\eps-h_\eps(y_\eps)$ measures the distance of the fast variables from the slow manifold.
\end{theorem}

\noindent
The proof is standard and is omitted here. The interested reader is referred to, for example, \cite{Carr,GottwaldMaclean13}.

\subsection{Discretization error}
We bound the discretization error $E^n_d = y^n - Y(t^n)$ 
%\ktwo{by splitting it into several pieces. 
\ktwo{in stages. We first give a proof for the convergence of a PI approximation of the reduced dynamics to the true reduced dynamics in Proposition~\ref{lemma_Gtwiddle}. We then compare the PI approximations of the reduced and the full multi-scale dynamics, and combine the two results to bound $E^n_d$.\\ 
To achieve the first bound we introduce} the auxiliary vector field $\tilde{G}$, which describes the PI2 method applied to the reduced slow system \eqref{e.CMT}. We first show that $\tilde{G}(Y(t^n))$ is close to a standard Runge-Kutta solver applied to $Y(t^n+M_1\dt)$; then we bound the difference between the auxiliary vectorfield $\tilde{G}(y^n)$ and the PI2 vectorfield for the slow variable $\tilde{g}(x^n,y^n)$. 

%\note{The philosophy of our proof is as follows: we split the discretization error into two parts. Define by $Y^n$ the numerical approximation to the reduced slow equation found by applying our numerical method to the reduced slow ode. We shall bound $|Y^n-Y(t^n)|$ and $|y^n-Y^n|$.}

\noindent
%In order to bound the discretization error $|y^n-Y(t^n)|$ we define an approximation to the reduced slow vectorfield, given by applying the PI method to the reduced slow system  \eqref{e.CMT}, by
Denote by $\phi^{m,\dt}$ the flow map for the microsolver of order p applied to the reduced system \eqref{e.CMT} for $m$ microsteps with time step $\dt$. Given initial condition $Y^n$ at $t=t^n$, we construct $\Gt$ analogously to the construction of $\gt$ used in PI2. We define
%\noindent
% Beginning with an application of the microsolver for the total number of $M_1$ steps with time %step $\dt$ to form
%\begin{align}
%\label{def:YnM} Y^{n,1}_{M_1} = \phi^{M_1,\dt}\left( Y^n \right) \;\;,
%\end{align} 
%analogous to \eqref{def:2znM}. 
$Y^{n,j}_{m}$ for $j=1,2,\dots,P+1$, $m=1,2,\dots,M_j$ as the output of the microsolver
\begin{align}
\label{def:Ynjm} Y^{n,j}_{m} &= \phi^{m,\dt}\left(Y^{n,j}_{0} \right) \;\; ,
\end{align}
analogous to \eqref{def:2znjm}, with initial condition
\begin{align}
\label{def:Ynj0} Y^{n,j}_{0} &= \twoparts{Y^n}{j=1}
{Y^{n,1}_{M_1} + a_j \Dt\,G(Y^{n,j-1}_{M_{j-1}})}{j>1}\;\;, 
%\label{def:ynjm} \left(x^{n,j}_{m},{y}^{n,j}_{m}\right) &= \varphi^{m,\dt}\left(x^{n,1}_{M_1} + a_j \hat{k}_{x,j-1},\;y^{n,1}_{M_1} + a_j \hat{k}_{y,j-1}\right) \;\; , 
\end{align}
analogous to \eqref{def:2znj0} and \eqref{2hatkz}. The increments are constructed by 
\begin{align}
\label{def:K} K_j(Y^n) &= \frac{1}{a_{j+1}}\left(Y^{n,j+1}_{M_{j+1}}-Y^{n,1}_{M_1}\right)\;\;,
\end{align}
analogous to \eqref{def:2kz}. Combining \eqref{def:Ynjm}--\eqref{def:K}, we form the auxiliary vectorfield
 \begin{align}
\label{def:Gtwiddle} \tilde{G}(Y^n) &= Y^{n,1}_{M_1} - Y^n + \sum_{j=1}^{P} b_j K_j(Y^n) \;\;,
\end{align}
analogous to the PI vectorfield \eqref{def:gtwiddle} of the macrosolver.\\

\noindent
In the following Proposition we demonstrate that $\tilde{G}$ evaluated at $Y(t^n)$ incurs an error of order $\mathcal{O}(\td^{P+1})$ over one macrostep, like standard Runge-Kutta methods, with an additional error term incurred by the applications of the microsolver.

\medskip
\begin{proposition}
\label{lemma_Gtwiddle} 
Given assumptions ({\it A1}), ({\it A2}), ({\it A4}) and ({\it A5}), $\tilde{G}(Y(t^n))$ provides a numerical estimate of the reduced slow vectorfield with
\begin{align*}
Y(t^{n+1}) = Y(t^{n}) + \tilde{G}(Y(t^{n})) + \mathcal{O}(\td^{P+1},\td M \dt) \;,
\end{align*}
where the error term $\mathcal{O}(\td^{P+1},\td M \dt)$ is bounded by $C_P^*\, \td^{P+1} + C_2^*\, \td M \dt$.
\end{proposition}

\begin{proof}
The increments $\bar{K}_j$ of a Runge-Kutta solver of order P applied to the reduced system \eqref{e.CMT}, initialised at $Y$ with time step $\td$, are given by
\begin{align}
\label{rk2}\bar{K}_j(Y) &= \td \, G(Y + a_j \bar{K}_{j-1}(Y)) \;\;.
\end{align}
For a Runge-Kutta solver of order P initialised at $Y(t^n+M_1\dt)$ we have
\begin{align}
\label{red.c}Y(t^{n+1}) =& Y(t^n+M_1\dt) + \sum_{j=1}^P b_j \bar{K}_j(Y(t^n+M_1\dt)) + \mathcal{O}\left(\td^{P+1}\right) \;\;,
\end{align}
where the $\mathcal{O}(\td^{P+1})$ term is bounded by $C_P^* \td^{P+1}$ \cite{Iserles}. Similarly, a microsolver $\phi$ of order p satisfies
\begin{align}
\label{m.p} \left|Y(t^n+M_1\dt) - Y^{n,1}_{M_1}\right| \le C_p^* M_1\dt^{p+1} \;\; .
\end{align}
 The Runge-Kutta solver \eqref{red.c} is rewritten as
\begin{align}
\label{red.micro1}Y(t^{n+1}) =& Y^{n,1}_{M_1} + \sum_{j=1}^P b_j \bar{K}_j(Y^{n,1}_{M_1}) \\
\nonumber &+ \sum_{j=1}^P b_j \left(\bar{K}_j(Y(t^n+M_1\dt))-\bar{K}_j(Y^{n,1}_{M_1})\right) + \mathcal{O}\left(\td^{P+1},{M_1\dt^{p+1}}\right) \;\;.
\end{align}
Employing assumptions ({\it A1})--({\it A2}) on the Lipschitz continuity of the reduced dynamics, \eqref{m.p} and the definition \eqref{rk2} of the Runge-Kutta increments $\bar{K}$ we bound
\begin{align*}
\left|\bar{K}_j(Y(t^n+M_1\dt)) - \bar{K}_j(Y^{n,1}_{M_1}) \right| \le& L_G \td \left| Y(t^n+M_1\dt) - Y^{n,1}_{M_1} \right|
\\
&+ L_G \td a_j \left| \bar{K}_{j-1}(Y(t^n+M_1\dt)) - \bar{K}_{j-1}(Y^{n,1}_{M_1}) \right|
\\
\le& L_G \td  C_p^* M_1 \dt^{p+1}\\ 
&+L_G \td a_j\left|\bar{K}_{j-1}(Y(t^n+M_1\dt)) - \bar{K}_{j-1}(Y^{n,1}_{M_1})\right| \;\; .
\end{align*}
Iterating this relationship with $a_1 = 0$ yields
\begin{align*}
\left|\bar{K}_j(Y(t^n+M_1\dt)) - \bar{K}_j(Y^{n,1}_{M_1}) \right| \le& L_G \td  C_p^* M_1 \dt^{p+1} + \mathcal{O}(\td^2M_1\dt^{p+1})\;\;.
\end{align*}
Upon substitution into \eqref{red.micro1} we obtain
\begin{align}
\label{red.micro}Y(t^{n+1}) =& Y^{n,1}_{M_1} + \sum_{j=1}^P b_j \bar{K}_j(Y^{n,1}_{M_1}) + \mathcal{O}\left(\td^{P+1},{M_1\dt^{p+1}}\right) \;\;,
\end{align}
which describes a Runge-Kutta method of order P, initialised at $Y^{n,1}_{M_1}$. The auxiliary vectorfield $\tilde{G}$ given by \eqref{def:Gtwiddle} with $Y^n = Y(t^n)$ is now constructed from \eqref{red.micro}. We write
\begin{align}
\nonumber Y(t^{n+1}) =& Y^{n,1}_{M_1} + \sum_{j=1}^P\! b_j\! \left({K}_j\left(Y(t^n)\right) +\bar{K}_j(Y^{n,1}_{M_1}) - {K}_j(Y(t^n))\!\right) + \mathcal{O}\left(\td^{P+1},{M_1\dt^{p+1}}\right) \\
\label{aux.main}=& Y(t^{n}) + \tilde{G}(Y(t^{n})) + \sum_{j=1}^P\! b_j\! \left(\bar{K}_j(Y^{n,1}_{M_1}) - {K}_j(Y(t^n))\!\right) + \mathcal{O}\left(\td^{P+1},{M_1\dt^{p+1}}\right) \;,
\end{align} 
where the increments $K_j$ are defined in \eqref{def:K}.

\noindent
We now bound $|\bar{K}_j(Y^{n,1}_{M_1}) - {K}_j(Y(t^n))|$ in \eqref{aux.main}. Rearranging the definition of $K_j$, \eqref{def:K}, we obtain
\begin{align}
\nonumber Y^{n,1}_{M_1}+a_{j+1}K_j(Y(t^n)) &= Y^{n,j+1}_{M_{j+1}} \\
\label{red.m}&= \phi^{M_{j+1},\dt}\left(Y^{n,j+1}_{0}\right)\;\;.
\end{align}
Similarly we use the definition of $\bar{K}_j(^{n,1}_{M_1})$, \eqref{rk2}, to obtain
\begin{align}
\nonumber Y^{n,1}_{M_1}+a_{j+1}\bar{K}_j(Y^{n,1}_{M_1}) =&\, Y^{n,1}_{M_1} + a_{j+1} \td \, G\left(Y^{n,1}_{M_1} + a_j \bar{K}_{j-1}  \right) \\
\nonumber =&\, Y^{n,1}_{M_1} + a_{j+1}\Dt\,G(Y^{n,1}_{M_1} + a_j K_{j-1})\\
\nonumber & +a_{j+1} \Dt \, \left( G\left(Y^{n,1}_{M_1} + a_j \bar{K}_{j-1}  \right) - G(Y^{n,1}_{M_1} + a_j K_{j-1})\right)\\
\nonumber &+a_{j+1} M\dt \, G\left(Y^{n,1}_{M_1} + a_j \bar{K}_{j-1}  \right)\\ \vspace{4pt}
\label{red.int} =&\, Y^{n,j+1}_{0} + a_{j+1}M\dt G(Y^{n,j+1}_{0} ) \\
\nonumber & +a_{j+1} \Dt \, \left( G\left(Y^{n,1}_{M_1} + a_j \bar{K}_{j-1}  \right) - G(Y^{n,1}_{M_1} + a_j K_{j-1})\right)\\
\nonumber &+a_{j+1} M\dt \, \left(G\left(Y^{n,1}_{M_1} + a_j \bar{K}_{j-1}  \right) -  G(Y^{n,j+1}_{0} )\right) \;\; ,
\end{align}
where we have suppressed the dependencies of $\bar{K}_j$ and $K_j$ on the right-hand side and used $\td = \Dt + M\dt$. Subtracting \eqref{red.m} from \eqref{red.int}, applying absolute values and dividing by $a_{j+1}$ yields the bound
\begin{align}
\label{red.ed} \left| \bar{K}_j(Y^{n,1}_{M_1}) - K_j(Y(t^n)) \right| \!\le&\! \frac{1}{a_{j+1}}\!\left|Y^{n,j+1}_{0} \!+\! a_{j+1}M\dt G(Y^{n,j+1}_{0} ) \!-\! \phi^{M_{j+1},\dt}\!\left(Y^{n,j+1}_{0}\right)\!\right|\\
\nonumber & + \Dt \, \left| G\left(Y^{n,1}_{M_1} + a_j \bar{K}_{j-1}  \right) - G(Y^{n,1}_{M_1} + a_j K_{j-1})\right|\\
\nonumber &+ M\dt \, \left|G\left(Y^{n,1}_{M_1} + a_j \bar{K}_{j-1}  \right) -  G(Y^{n,j+1}_{0} )\right| \;\; .
\end{align}
We now bound the three lines of \eqref{red.ed} separately. In the first line, we interpret $Y^{n,j+1}_{0} + a_{j+1}M\dt G(Y^{n,j+1}_{0} ) $ as a single Euler step with time step $a_{j+1}M\dt$ initialised at $Y^{n,j+1}_{0}$. The remaining term of the first line, $\phi^{M_{j+1},\dt}\left(Y^{n,j+1}_{0}\right)$, describes a microsolver of order p run for $M_{j+1} = a_{j+1}M$ steps, also initialised at $Y^{n,j+1}_{0}$. Therefore
\begin{align}
\label{red.t1} \left|Y^{n,j+1}_{0} \!+\! a_{j+1}M\dt G(Y^{n,j+1}_{0} ) \!-\! \phi^{M_{j+1},\dt}\!\left(\!Y^{n,j+1}_{0}\!\right)\!\right|\le \mathcal{O}(a_{j+1}M\dt^{p+1},(a_{j+1}M\dt)^2)\;.
\end{align}
The second line in \eqref{red.ed}, employing Assumptions ({\it A1})--({\it A2}) on the Lipschitz continuity of the reduced dynamics, is bounded by
\begin{align}
\label{red.t2} \Dt \, \left| G\left(Y^{n,1}_{M_1} + a_j \bar{K}_{j-1}  \right) - G(Y^{n,1}_{M_1} + a_j K_{j-1})\right| \leq&  a_j \Dt L_G \left| \bar{K}_{j-1}  - K_{j-1}\right|\;\; .
\end{align}
Upon using Assumption ({\it A4}) on the boundedness of the reduced dynamics and Assumptions ({\it A1})--({\it A2}), we bound the term in absolute values in the third line of \eqref{red.ed} by
\begin{align}
\nonumber  \left|G\left(Y^{n,1}_{M_1} + a_j \bar{K}_{j-1}  \right) -  G(Y^{n,j+1}_{0} )\right| \leq &  L_G \, \left|Y^{n,1}_{M_1} + a_j \bar{K}_{j-1} -  Y^{n,j+1}_{0}\right| 
\\
\nonumber = & L_G \, \big|a_j \td\, G\left(Y^{n,1}_{M_1}+a_{j-1}\bar{K}_{j-2}\right)\\
\nonumber &\qquad - a_{j+1}\Dt\,G\left(Y^{n,j}_{M_j}\right)\big|  \\
 \label{red.t3}\le&2C_2^* \td \;\;,
\end{align}
where we employed the bound on the nodes $a_j\le 1$, and Assumption ({\it A5}) on the smoothness of the reduced dynamics with $C_2^* = \sup|\ddot{Y}| = \sup|\D G(Y)\, G(Y)| = L_G C_G$.\\
Substituting \eqref{red.t1}--\eqref{red.t3} into \eqref{red.ed} yields the bound
\begin{align}
\label{red.KK} \left| \bar{K}_j(Y^{n,1}_{M_1}) - K_j(Y(t^n)) \right| \le & a_j \Dt L_G \left| \bar{K}_{j-1}  - K_{j-1}\right|+ 2C_2^*\td M\dt \\
\nonumber & +\mathcal{O}(M\dt^{p+1}, a_{j+1}(M\dt)^2)) \;\; .
\end{align}
Substituting $j=1$, noting that $a_1=0$ and neglecting higher order terms, we have
\begin{align}
\label{red.base} \left| \bar{K}_1(Y^{n,1}_{M_1}) - K_1(Y(t^n)) \right| \le &  2C_2^*\td M\dt \;\; .
\end{align}
Iteration of \eqref{red.KK}, seeded with \eqref{red.base} at $j=1$, yields
\begin{align*}
\left| \bar{K}_j(Y^{n,1}_{M_1}) - K_j(Y(t^n)) \right| \le &  2C_2^*\td M\dt \;\;.
\end{align*}
The Proposition now follows directly by substituting into \eqref{aux.main} and using the weighting condition $\textstyle\sum_{j=1}^P b_j =1$.
\end{proof}
\begin{remark}\label{PI1.p1}
Proposition~\ref{lemma_Gtwiddle} can be readily extended for PI1. This Proposition employs auxiliary increments $K_j$ designed to resemble the PI2 increments \eqref{def:2kz}. In order to prove this result for the PI1 method, one should instead employ auxiliary increments $\hat{K}_j = \Dt \, G(Y^{n,j}_{M_j})$, which resemble the PI1 increments \eqref{hatkz}. Following from \eqref{aux.main}, one can then readily bound $|\hat{K}_j - \bar{K}_j| \le 2C_2^* \Dt M \dt$ to obtain the same bound.\\
\end{remark}

\noindent
Proposition~\ref{lemma_Gtwiddle} establishes that using $\Gt$ to propagate the reduced dynamics incurs error proportional to $\td^{P+1} + \td M\dt$. In particular, these terms do not depend on the order p of the microsolver. To simplify the calculations, we therefore use a forward Euler method as the microsolver for the reduced system. In particular, we consider
\begin{align}
\label{eY}  Y^{n,j}_{m} &= Y^{n,j}_{m-1} + \dt\, G(Y^{n,j}_{m-1}) \;\;,
\end{align}
for $0\le m\le M_j$.\\

\noindent
We now use Proposition~\ref{lemma_Gtwiddle} to bound the error between the PI2 approximation of the slow variable $y^n$ in a full multiscale simulation and the reduced dynamics $Y(t^n)$. 
\begin{lemma}
\label{lemma_rk}
Given assumptions ({\it A1})--({\it A5}), the discretization error $|E^n_d| = \left|y^n-Y(t^n)\right|$ is bounded by
\begin{align*}
|E^n_d| \le& \frac{e^{L_G t^n}}{L_G} 
 \Bigg\{  
C_P^* \td^P + C_2^*M\dt
 + \max_{0\le i \le n-1}\left|\frac{\tilde{g}(x^i,y^i)-\tilde{G}(x^i,y^i)}{\td}\right|\Bigg\} \; .
\end{align*}

\end{lemma}
\begin{proof}
This result follows from \cite{EEngquist03}. Employing Proposition~\ref{lemma_Gtwiddle}, we have
\begin{align}
\nonumber E^n_d =& E^{n-1}_d + \gt(x^{n-1},y^{n-1}) - \Gt(Y(t^{n-1})) + \mathcal{O}(\td^{P+1},\td M \dt) \\
%thesis \nonumber =& E^{n-1}_d + \gt(x^{n-1},y^{n-1}) -\Gt(y^{n-1}) + \Gt(y^{n-1})- \Gt(Y(t^{n-1})) + \mathcal{O}(\td^{P+1},\td M \dt) \\
\label{gw} =& E^{n-1}_d + \mathcal{L}_{G}^{n-1} \, E^{n-1}_d + \gt(x^{n-1},y^{n-1}) -\Gt(y^{n-1})  + \mathcal{O}(\td^{P+1},\td M \dt) \;\;,
\end{align}
 where we used the mean value theorem for vector-valued functions to introduce 
\begin{align}
\label{rk_LE} \LE^n&=\int_0^1 \! \D \tilde{G}\big(Y(t^n)+\theta(y^n-Y(t^n))\big) \, d\theta \; ,
\end{align}
where $\D \tilde{G}$ is the Jacobian matrix of $\tilde{G}$. Recall that the $\mathcal{O}(\td^{P+1},\td M \dt)$ term in \eqref{gw} is bounded by $C_P^*\td^{P+1} + C_2^*\td M \dt$; taking absolute values of \eqref{gw} then yields
\begin{align}
\label{rk.int}
|E^n_d| \le & \left(1 + \left|\mathcal{L}_{G}^{n-1} \right|\right)\left| E^{n-1}_d \right| + \left|\gt(x^{n-1},y^{n-1}) -\Gt(y^{n-1})\right|  + C_P^*\td^{P+1} + C_2^*\td M \dt \;\;.
\end{align}
To bound $|\mathcal{L}_{G}^{n-1} |$, we first obtain an explicit formula for $\Gt$. Substituting \eqref{def:Gtwiddle}, \eqref{def:K} and the Euler microsolver \eqref{eY} into $\Gt$, \eqref{def:Gtwiddle}, we obtain
\begin{align*}
\tilde{G}(Y^n) =& Y^{n,1}_{M_1} - Y^n + \sum_{j=1}^{P} \frac{b_j}{a_{j+1}}\left(Y^{n,j+1}_{M_{j+1}} - Y^{n,1}_{M_1} \right) \\
=& Y^n + \dt \sum_{k=0}^{M_1-1}\! G(Y^{n,1}_{k}) - Y^n \\&+ \sum_{j=1}^{P} \frac{b_j}{a_{j+1}}\!\left(Y^{n,j+1}_{0} + \dt \!\sum_{k=0}^{M_{j+1}-1}\!G(Y^{n,j+1}_{k}) - Y^{n,1}_{M_1} \!\right) \,  .
\end{align*}
Substituting $Y^{n,j+1}_{0} = Y^{n,1}_{M_1} + a_{j+1} \Dt\,G(Y^{n,j}_{M_{j}})$ from \eqref{def:Ynj0}, we obtain the explicit formula
\begin{align*}
\tilde{G}(Y^n) =& \dt \sum_{k=0}^{M_1-1} G(Y^{n,1}_{k})  + \sum_{j=1}^{P} \frac{b_j}{a_{j+1}}\left(a_{j+1} \Dt \, G(Y^{n,j}_{M_{j}}) + \dt \sum_{k=0}^{M_{j+1}-1}G(Y^{n,j+1}_{k})  \right) \; \; .
\end{align*}
Substituting into \eqref{rk_LE} with $Y^n =Y(t^n)+\theta(y^n-Y(t^n))$, taking absolute values and employing Assumptions ({\it A1})--({\it A2}) on the Lipshitz constant $L_G$ of the reduced dynamics yields 
\begin{align*}
 |\LE^n| &\le \int_0^1 \! \left| \D \tilde{G}\big(Y(t^n)+\theta(y^n-Y(t^n))\big) \right| \, d\theta \\
 &\le \int_0^1   \dt \!\!\sum_{k=0}^{M_1-1} \left| \D G(Y^{n,1}_{k}) \right|  + \sum_{j=1}^{P} \frac{b_j}{a_{j+1}} a_{j+1} \Dt \, \left| \D G(Y^{n,j}_{M_{j}})\right| \\
 &+ \dt \sum_{k=0}^{M_{j+1}-1} \left|\D G(Y^{n,j+1}_{k})  \right|  \, d\theta \\
 &\le \int_0^1 \! \left(  L_G M_1 \dt   + \sum_{j=1}^{P} \frac{b_j}{a_{j+1}}\left( L_G a_{j+1} \Dt+ L_G M_{j+1} \dt  \right) \right) \, d\theta \;\; .
\end{align*}
Recalling $M_j = a_j M$ for $j>1$, and using the weighting condition $\textstyle\sum_{j=1}^P b_j =1$, we obtain the bound 
\begin{align*}
 |\LE^n| &\le L_G \left( \td + M_1 \dt \right) \;\; .
\end{align*}
We substitute this bound into \eqref{rk.int}, with 
\begin{align*}
|E^n_d| \le & \left(1 + L_G(\td+M_1\dt)\right)|E^{n-1}_d| + \left|\gt(x^{n-1},y^{n-1}) -\Gt(y^{n-1})\right|  \\&+ C_P^*\td^{P+1} + C_2^*\td M \dt \;\;.
\end{align*}
Iterating the recursive relationship with $E^0_d=0$ yields
\begin{align*}
|E^n_d| \le &\sum_{m=0}^{n-1} \!\!\left(1 + L_G(\td+M_1\dt)\right)^m \!\! \Bigg(\!\!C_P^*\td^{P+1} + C_2^*\td M \dt \\
&\qquad\qquad\qquad\qquad\qquad\qquad+\max_{0\le i \le n-1}\frac{\!|\gt(x^i,y^i) -\Gt(y^i)|}{\,\td} \!\! \Bigg) \\
%thesis = &\frac{\left(1 + L_G(\td+M_1\dt)\right)^n-1}{L_G(\td+M_1\dt)}  \left(C_P^*\td^{P+1} + C_2^*\td M \dt + \max_{0\le i \le n-1}\frac{\!|\gt(x^i,y^i) -\Gt(y^i)|}{\,\td}  \right)  \\
\le & \frac{e^{L_G\,t^n}}{L_G} \left(  C_P^*\td^{P} + C_2^* M \dt + \max_{0\le i \le n-1}\frac{\!|\gt(x^i,y^i) -\Gt(y^i)|}{\,\td}  \right)  \;\;,
\end{align*}
where $t^n=n(\td+M_1\dt)$.
\end{proof}

Lemma~\ref{lemma_rk} establishes that the error between the PI2 approximation of the slow variables of the full multiscale system \eqref{baseneo} and the true reduced dynamics \eqref{e.CMT} contains a term proportional to the order of the macrosolver $\td^P$, a term proportional to the length of the microsolver $M\dt$ and an additional term proportional to $\max_{0\le i \le n-1}|\gt(x^{n-1},y^{n-1}) -\Gt(y^{n-1})|$. The latter term measures the difference between the PI2 vector field $\tilde{g}$ of the slow variables, and the auxiliary vector field $\tilde{G}$ initialised at the same point $y^{n-1}$. In order to bound this term we define the deviation of the PI approximation of the fast variables from the approximate slow manifold over the increments,
\begin{align}
d^{n,j}_{m} = x^{n,j}_{m}-h_0(y^{n,j}_{m})\; ,
\end{align}
for $j=1,2,\dots,P+1$, $m=1,2,\dots,M_j$. The following Lemma bounds $|d^{n,j}_m|$.
\medskip
\begin{lemma}
\label{lemma_dnm}
Given Assumptions ({\it A1}) and ({\it A4}), the error between the fast variables and the approximate slow manifold during the application of a microsolver of order p is bounded for all $0\leq m\leq M_j$ by
\begin{align*}
|d^{n,j}_{m}|\le 
%\twoparts{
\rho^m\left(-\frac{ \dt}{\varepsilon}\right) |d^{n,j}_{0}|  + L_h C_g \varepsilon 
%}{
%\left( 1-\frac{\dt^*}{\varepsilon}\right)^m |d^{n,0}|  +\frac{\dt}{\dt^*}\varepsilon  L_fC_g
%}
\;\; ,
\end{align*}
\end{lemma}
where 
\begin{align*}
\rho\left(-\frac{\dt}{\varepsilon}\right)=\sum_{k=0}^p \frac{\left(-\frac{\dt}{\varepsilon} \right)^k}{k!} \;\;.
\end{align*}
The first term in the Lemma is a manifestation of the attraction of the fast variables towards the slow manifold along their stable eigendirection. The second term proportional to $\eps$ describes, as we will see below, the cumulative drift of the slow variables $y$ during the microsteps causing a departure from the slow manifold for nonconstant $h_0(y)$.

\noindent
To ensure convergence of the fast variables to the approximate slow manifold we require
\begin{align*}
0\le \rho\left(-\frac{ \dt}{\eps}\right) < 1 \;\; .
\end{align*}

%\begin{figure}
%\includegraphics[width=0.8\textwidth]{lemma_3}
%\caption{Illustration of equation \eqref{e.dnmonestep} for $x_\eps \in \mathbb{R}$, so that $\Lambda=1$. The two dots show one microstep in PI.}
%\label{fig:microstep}
%\end{figure}

\begin{proof}
Denote the increments of the microsolver as $\bar{k}_{x,i}$ and $\bar{k}_{y,i}$ for the fast and slow variables respectively, with nodes $a_i'$ and weights $b_i'$. Employing the fast vector field \eqref{basefast} we write the fast increments $\bar{k}_{x,i}$ analogously to \eqref{rk.k} as
\begin{align}
\nonumber \bar{k}_{x,i}(x^{n,j}_{m}, y^{n,j}_{m}) =& -\frac{\Lambda \dt }{\eps} \left( x^{n,j}_{m} + a'_i\bar{k}_{x,i-1}\right) + \frac{\Lambda \dt }{\eps}h_0(y^{n,j}_{m} + a'_i\bar{k}_{y,i-1})\\
\nonumber=& -\frac{\Lambda \dt }{\eps} \left( x^{n,j}_{m} - h_0(y^{n,j}_{m}) + a'_i\bar{k}_{x,i-1}\right) + \mathcal{O}\left(\frac{\dt^2}{\eps}\right)\\
\label{k.exp}=& -\frac{\Lambda \dt }{\eps} \left( d^{n,j}_{m} + a'_i\bar{k}_{x,i-1}\right) + \mathcal{O}\left(\frac{\dt^2}{\eps}\right)\;\;,
\end{align}
where we have used that $\bar{k}_{y,i} = \O(\dt)$. Introducing the increment associated with a linear system $\bar{k}_{\text{lin},i}(d^{n,j}_{m})=-{\Lambda \dt }{} \left( d^{n,j}_{m} + a'_i\bar{k}_{\text{lin},i-1}\right)/\eps$, we write
\begin{align*}
\bar{k}_{x,i}(x^{n,j}_{m}, y^{n,j}_{m}) -  \bar{k}_{\text{lin},i}(d^{n,j}_{m}) =& -\frac{\Lambda \dt }{\eps} a'_i \left( \bar{k}_{x,i-1} - \bar{k}_{\text{lin},i-1}\right) + \mathcal{O}\left(\frac{\dt^2}{\eps}\right)\\
=&\;\, \mathcal{O}\left(\frac{\dt^2}{\eps}\right) \;\;.
\end{align*}
The linear component $\bar{k}_{\text{lin},i}$ of the increment can be interpreted as the increment of the microsolver applied to the linear system $\dot d_\eps = -\Lambda d_\eps/\eps$ with initial condition $d^{n,j}_m$. Therefore, as discussed in Section~\ref{sec:rk}, a microstep taken with the linear increments $\bar{k}_{\text{lin},i}$ can be written as
\begin{align*}
d^{n,j}_{m} + \sum_{i=1}^p b_i' \bar{k}_{\text{lin},i}(d^{n,j}_{m}) =& \rho\left(-\frac{\Lambda\dt}{\eps}\right)d^{n,j}_{m} \;.
\end{align*}
Employing \eqref{rk.m} the microstep is expressed as
\begin{align*}
 x^{n,j}_{m+1} =& x^{n,j}_{m} + \sum_{i=1}^p b_i' \bar{k}_{x,i}(x^{n,j}_{m}, y^{n,j}_{m})  \\
 =& x^{n,j}_{m} + \sum_{i=1}^p b_i' \bar{k}_{\text{lin},i}(d^{n,j}_{m}) + \mathcal{O}\left(\frac{\dt^2}{\eps}\right)  \\
 =& d^{n,j}_{m} + \sum_{i=1}^p b_i' \bar{k}_{\text{lin},i}(d^{n,j}_{m}) + h_0(y^{n,j}_{m} )+ \mathcal{O}\left(\frac{\dt^2}{\eps}\right) \\
 =& \rho\left(-\frac{\Lambda\dt}{\eps}\right) d^{n,j}_{m} + h_0(y^{n,j}_{m} ) + \mathcal{O}\left(\frac{\dt^2}{\eps}\right) \;\;.
\end{align*}

%Let us first consider the linear system $\dot{x} = -\Lambda x /\eps$, i.e. the fast dynamics with $h_0=0$. The increments would then be 
%\begin{align}
%\label{kstupid} \bar{k}_{x,i}(x^{n,j}_{m}, y^{n,j}_{m}) =& -\frac{\Lambda \dt }{\eps} \left( x^{n,j}_{m} + a'_i\bar{k}_{x,i-1}\right) \;\;.
%\end{align}
%As discussed in Section~\ref{sec:rk}, substituting these simpler increments into \eqref{mstep} would produce 
%\begin{align*}
%x^{n,j}_{m+1} =& \rho\left(-\frac{\Lambda \dt}{\eps}\right)x^{n,j}_{m} \;\; .
%\end{align*}
%To understand the effects of the microsolver applied to the true dynamics \eqref{basefast}, we have only to estimate the effects of nonzero $h_0$. \\

%thesis We Taylor expand $h_0(y^{n,j}_{m} + a'_i\bar{k}_{y,i-1})$ in \eqref{kmicro} and employ assumptions ({\it A1}) and ({\it A4}) on the growth of $h_0$ and boundedness of g to obtain
%\begin{align*}
%h_0(y^{n,j}_{m} + a'_i\bar{k}_{y,i-1}) =& h_0(y^{n,j}_{m} ) + L_h C_g a'_i \dt + \mathcal{O}(\dt^2) \;\; .
%\end{align*}
%Substituting into \eqref{kmicro} produces
%\begin{align*}
%\bar{k}_{x,i}(x^{n,j}_{m}, y^{n,j}_{m}) =& -\frac{\Lambda \dt }{\eps} \left( x^{n,j}_{m} - h_0(y^{n,j}_{m}) + a'_i\bar{k}_{x,i-1}\right) + \mathcal{O}\left(\frac{\dt^2}{\eps}\right)\\
%=& -\frac{\Lambda \dt }{\eps} \left( d^{n,j}_{m} + a'_i\bar{k}_{x,i-1}\right) + \mathcal{O}\left(\frac{\dt^2}{\eps}\right)\;\;.
%\end{align*}

Then
\begin{align}
\nonumber \left|d^{n,j}_{m+1} \right|=& \left|x^{n,j}_{m+1}- h_0(y^{n,j}_{m+1})\right| \\
\nonumber \le& \left|\rho\left(-\frac{\Lambda\dt}{\eps}\right)\right| \left|d^{n,j}_{m}\right| + \left|h_0(y^{n,j}_{m} )- h_0(y^{n,j}_{m+1})\right| + \mathcal{O}\left(\frac{\dt^2}{\eps}\right) \\
\label{micro.r}  \le& \left|\rho\left(-\frac{\Lambda\dt}{\eps}\right)\right| \left|d^{n,j}_{m}\right| + L_h C_g \dt + \mathcal{O}\left(\frac{\dt^2}{\eps}\right) \;\; .
\end{align}
The first term in this bound represents the rate of convergence of the fast variables to the approximate slow manifold; for stability we require $\left|\rho\left(-\frac{\Lambda\dt}{\eps}\right)\right| < 1$. The term $L_h C_g \dt$ stems from the drift in the slow variables over a microstep. The slowest rate of convergence to the slow manifold is given by $\min(\lambda_{ii})=1$, so we obtain
 \begin{align*}
 \left|\rho\left(-\frac{\Lambda\dt}{\eps}\right)\right| \le & \rho\left(-\frac{\dt}{\eps}\right) \\
 =& 1 - \frac{\dt}{\eps} + \frac{1}{2}\left(-\frac{\dt}{\eps}\right)^2 + \dots + \frac{1}{p!}\left(-\frac{\dt}{\eps}\right)^p \\
 <& 1 \;\; .
 \end{align*} 
Iterating \eqref{micro.r} then yields
\begin{align*}
 \left|d^{n,j}_{m} \right|  \le& \rho^m\left(-\frac{\dt}{\eps}\right) \left|d^{n,j}_{m}\right| + \frac{L_h C_g \dt + \mathcal{O}\left(\frac{\dt^2}{\eps}\right) }{1-\left|\rho\left(-\frac{\dt}{\eps}\right)\right|} \\
  \le& \rho^m\left(-\frac{\dt}{\eps}\right) \left|d^{n,j}_{m}\right| + L_h C_g \eps + \mathcal{O}\left(\dt\right)\;\; ,
 \end{align*}
completing the proof of the Lemma.
 
\end{proof}
\begin{remark}
In the above Lemma we Taylor expand the terms in $h_0(y^{n,j}_{m} + a'_i\bar{k}_{y,i-1})$ up to $\mathcal{O}(\dt^2/\eps)$. However, higher-order terms may improve the error bound. For instance, one can show that for a fourth-order Runge-Kutta microsolver,
\begin{align*}
 \left|d^{n,j}_{m+1} \right| \le& \rho\left(-\frac{\dt}{\eps}\right) \left|d^{n,j}_{m}\right| + \frac{9}{24}L_h C_g \dt \;\; .
\end{align*}
\end{remark}
\begin{remark}
Optimal convergence of the fast variables to the approximate slow manifold during the application of the microsolver is given by a forward Euler microsolver with $0 < \dt \le \frac{2\eps}{\lambda+1}$, where
\begin{align*}
\rho\left(-\frac{ \dt}{\eps}\right) =& \; 1-\frac{\dt}{\eps}\;\; ,
\end{align*}
and the convergence rate $\rho^m\left(-{ \dt}/{\eps}\right)$ is bounded above by the exponential convergence $\exp{(-{m\dt}/{\eps})}$. The full stability region for the Euler microsolver is $0 < \dt < {2\eps}/{\lambda}$; for further details, see \cite{GottwaldMaclean13}.
\end{remark}

\medskip
\noindent
From Lemma~\ref{lemma_dnm} it follows that the rate of convergence of the fast variables to the approximate slow manifold is optimal for an Euler microsolver, and in Lemma~\ref{lemma_rk} we demonstrated that the dominant error terms between the PI2 approximation $y^n$ and the true reduced dynamics $Y(t^n)$ do not depend on the order p of the microsolver. We therefore choose as the microsolver for the PI2 scheme the forward Euler method to simplify the calculations, and write
%optimise the computational cost at no increase to the error, and write
\begin{align}
\label{ey} y^{n,j}_{m} &= y^{n,j}_{m-1} + \dt\, g(x^{n,j}_{m-1},y^{n,j}_{m-1}) \;\;, \\
\label{ex} x^{n,j}_{m} &= x^{n,j}_{m-1} + \dt\, f(x^{n,j}_{m-1},y^{n,j}_{m-1}) \;\;,
\end{align}

\noindent
We now bound the distance of the PI approximation $y^{n,j}_{m}$ of the slow variables over the microsteps of the full system \eqref{baseneo}, from $Y^{n,j}_{m}$, the PI approximation of the reduced dynamics over the microsteps. 

\medskip
\begin{lemma} Assuming ({\it A1}),({\it A2}) and ({\it A6}), the PI2 numerical estimate $y^{n,j}_{M_j}$ of the slow variable after the application of the microsolver at the $j$-th increment is close to the numerical estimate $Y^{n,j}_{M_j}$ of the reduced slow variable which was initialized at $Y^{n,1}_{0}=y^n$, with
\begin{align*}
%|y^{n,m} - \phiyn^m| \leq 3 L_g \eps|d^{n,0}| + \mathcal{O}(m\dt^2, m\dt\eps) \; .
|y^{n,j}_{M_j} - Y^{n,j}_{M_j}| \leq &
%\twoparts
%{\;\;\;\;\;\; 
2 L_g \left( 3\eps + a_j \Dt \left(1-\frac{\dt}{\eps}\right)^{aM} \right) \!\! \max_{1\le k \le j} \left|d^{n,k}_{0} \right| + 2  L_g L_h C_g a_j \td \eps  \\
&+ \mathcal{O}\left(\left(\Dt\eps + \Dt^2 (1-\frac{\dt}{\eps})^{aM}\right) \max_{1\le k \le j} \left|d^{n,k}_{0} \right|,  \td^2 \eps , M\dt\eps,M_1\dt \eps \right)  \;\; , %+ \mathcal{O}(\eps M\dt) 
%{\left( 3 L_g \eps|d^{n,0}| +  2 \frac{L_fC_g\eps}{L_f+1}\right)\frac{\dt}{\dt^*}  + \mathcal{O}(m\dt^2)} \;\; .
\end{align*}
for $2\le j \le P$, where $aM = \min_{j>1} M_j$, and
\begin{align*}
%|y^{n,m} - \phiyn^m| \leq 3 L_g \eps|d^{n,0}| + \mathcal{O}(m\dt^2, m\dt\eps) \; .
|y^{n,1}_{M_1} - Y^{n,1}_{M_1}| \leq &
%\twoparts
%{\;\;\;\;\;\; 
%2L_g \eps\left|d^{n,1}_{0}\right| + 2 L_g L_h C_g  M_1 \dt \eps \;\;,%+ \mathcal{O}(\eps M\dt) 
2L_g \eps\left|d^{n,1}_{0}\right| + \O( M_1 \dt\, \eps) \;\;,%+ \mathcal{O}(\eps M\dt) 
%{\left( 3 L_g \eps|d^{n,0}| +  2 \frac{L_fC_g\eps}{L_f+1}\right)\frac{\dt}{\dt^*}  + \mathcal{O}(m\dt^2)} \;\; .
\end{align*}
for $j=1$.
\label{phi.bound}
\end{lemma} 
\begin{proof}
Employing the definition of the Euler microsolvers \eqref{ey} for the PI2 scheme and \eqref{eY} for the reduced scheme, and Assumptions ({\it A1})--({\it A2}) on the Lipschitz continuity of the reduced dynamics gives
\begin{align}
%\label{lemma.aux}
\nonumber \left|y^{n,j}_{M_j} - Y^{n,j}_{M_j}\right| \leq& \left|y^{n,j}_{M_j-1} - Y^{n,j}_{M_j-1}\right| + \dt\left|g(x^{n,j}_{M_j-1},y^{n,j}_{M_j-1}) - G(Y^{n,j}_{M_j-1})\right| 
\\
\nonumber \le& \left|y^{n,j}_{M_j-1} - Y^{n,j}_{M_j-1}\right| + \dt\left|g(x^{n,j}_{M_j-1},y^{n,j}_{M_j-1}) - G(y^{n,j}_{M_j-1})\right|\\
\nonumber &+\dt\left|G(y^{n,j}_{M_j-1})-G(Y^{n,j}_{M_j-1})\right| 
\\
\nonumber \le& (1+L_G \dt)\left|y^{n,j}_{M_j-1} - Y^{n,j}_{M_j-1}\right| \\
\nonumber &+ \dt\left|g(x^{n,j}_{M_j-1},y^{n,j}_{M_j-1}) - g(h_\eps(y^{n,j}_{M_j-1}),y^{n,j}_{M_j-1})\right| 
\\
\nonumber \le& (1+L_G \dt)\left|y^{n,j}_{M_j-1} - Y^{n,j}_{M_j-1}\right| + L_g \dt\left|x^{n,j}_{M_j-1}- h_\eps(y^{n,j}_{M_j-1})\right| 
\\
\nonumber=& (1+L_G \dt)\left|y^{n,j}_{M_j-1} - Y^{n,j}_{M_j-1}\right| + L_g \dt\left|d^{n,j}_{M_j-1}\right|
\\
\nonumber &+L_g \dt\left| h_0(y^{n,j}_{M_j-1}) - h_\eps(y^{n,j}_{M_j-1})\right|
\\
\nonumber=& (1+L_G \dt)\left|y^{n,j}_{M_j-1} - Y^{n,j}_{M_j-1}\right| + L_g \dt\left|d^{n,j}_{M_j-1}\right| +L_g L_\eps\eps\,  \dt +\O(\dt\eps^2)\;\;,
\end{align}
%This establishes a discrete form of the bound from the reduction error, \eqref{e.yrederr}. 

\noindent
where we have defined $ h_0(y) - h_\eps(y) = L_\eps \eps + \O(\eps^2)$. Employing Lemma~\ref{lemma_dnm} on $|d^{n,j}_{M_j-1}|$ for a forward Euler microsolver yields the recursive bound
\begin{align*}
\left|y^{n,j}_{M_j} - Y^{n,j}_{M_j}\right| \le& (1+L_G \dt)\left|y^{n,j}_{M_j-1} - Y^{n,j}_{M_j-1}\right| \\
&+ L_g \dt\left(1-\frac{\dt}{\eps}\right)^{M_j-1}\left|d^{n,j}_{0}\right| + L_g (L_h C_g+L_\eps) \eps\, \dt   \;\;,
\end{align*}
which, upon iterating, gives 
\begin{align*}
\left|y^{n,j}_{M_j} - Y^{n,j}_{M_j}\right| \le& (1+L_G \dt)^{M_j}\left|y^{n,j}_{0} - Y^{n,j}_{0}\right| \\
&+ L_g \dt\left|d^{n,j}_{0}\right|\sum_{k=0}^{M_j-1}\left(1-\frac{\dt}{\eps}\right)^{k}(1+L_G\dt)^{M_j-1-k} \\
%&+ L_g L_h C_g\, \eps\, \dt \sum_{k=0}^{M_j-1}(1+L_G\dt)^k 
&+ L_g (L_h C_g+L_\eps)\eps\, \dt\sum_{k=0}^{M_j-1}(1+L_G\dt)^k 
\\
=& (1+L_G \dt)^{M_j}\left|y^{n,j}_{0} - Y^{n,j}_{0}\right| \\
&+ L_g \dt\left|d^{n,j}_{0}\right|\frac{(1+L_G\dt)^{M_j} - \left(1-\frac{\dt}{\eps}\right)^{M_j}}{L_G\dt+\frac{\dt}{\eps}} \\
%&+ L_g L_h C_g\, \eps\, \dt \frac{(1+L_G\dt)^{M_j}-1}{L_G\dt} 
&+ L_g (L_h C_g+L_\eps) \eps\, \dt \frac{(1+L_G\dt)^{M_j}-1}{L_G\dt} 
\\
\le& e^{L_G M_j \dt} \left|y^{n,j}_{0} - Y^{n,j}_{0}\right| 
%+ L_g \eps\left|d^{n,j}_{0}\right|\frac{e^{L_G M_j \dt} }{L_G\eps+1} + L_g L_h C_g\, \eps\, \dt \frac{e^{L_GM_j \dt}-1}{L_G\dt} \;\; .
+ L_g \eps\left|d^{n,j}_{0}\right|\frac{e^{L_G M_j \dt} }{L_G\eps+1}
\\
& + L_g (L_h C_g+L_\eps) \eps\, \dt \frac{e^{L_GM_j \dt}-1}{L_G\dt} \;\; .
\end{align*}
Realising that $e^{L_G M_j \dt}-1\leq 2 L_G M_j \dt$ under Assumption ({\it A6}), we obtain
\begin{align}
\label{mdiff}
%\left|y^{n,j}_{M_j} - Y^{n,j}_{M_j}\right| \le& 2\left|y^{n,j}_{0} - Y^{n,j}_{0}\right| + 2L_g \eps\left|d^{n,j}_{0}\right| + 2L_g L_h C_g \eps M_j \dt \;\; .
\left|y^{n,j}_{M_j} - Y^{n,j}_{M_j}\right| \le& 2\left|y^{n,j}_{0} - Y^{n,j}_{0}\right| + 2L_g \eps\left|d^{n,j}_{0}\right| + \O(\eps M_j \dt) \;\; .
\end{align}
At $j=1$ we initialize at $y^{n,1}_{0} = Y^{n,1}_{0} = y^n$, obtaining the desired bound
\begin{align}
\label{mdiff_1}
%\left|y^{n,1}_{M_1} - Y^{n,1}_{M_1}\right| \le& 2L_g \eps\left|d^{n,1}_{0}\right|  + 2 L_g L_h C_g \eps M_1 \dt\;\; .
\left|y^{n,1}_{M_1} - Y^{n,1}_{M_1}\right| \le& 2L_g \eps\left|d^{n,1}_{0}\right|  + \O( \eps M_1 \dt)\;\; .
\end{align}
For $j>1$, we have 
\begin{align*}y^{n,j}_{0} =& y^{n,1}_{M_1} + a_j \hat{k}_{y,j-1} \\
=& y^{n,1}_{M_1} + a_j \Dt \,g(x^{n,j-1}_{M_{j-1}},y^{n,j-1}_{M_{j-1}}) \;\; ,
\end{align*} 
using the definitions \eqref{def:2znj0} and \eqref{2hatky}, and 
\begin{align*}
Y^{n,j}_{0} = Y^{n,1}_{M_1} + a_j \Dt\,G(Y^{n,j-1}_{M_{j-1}}) \;\; ,
\end{align*}
using \eqref{def:Ynj0}. Substituting these into \eqref{mdiff} and employing assumptions {\it (A1)}--{\it (A2)}, we obtain
\begin{align*}
\left|y^{n,j}_{M_j} - Y^{n,j}_{M_j}\right| \le& 2\left|y^{n,1}_{M_1} - Y^{n,1}_{M_1}\right| 
+2 a_j \Dt \left|g(x^{n,j-1}_{M_{j-1}},y^{n,j-1}_{M_{j-1}}) - G(Y^{n,j-1}_{M_{j-1}}) \right| \\
%&+ 2L_g \eps\left|d^{n,j}_{0}\right| + 2L_g L_h C_g a_j M \dt \eps
&+ 2L_g \eps\left|d^{n,j}_{0}\right| + \O( M \dt \eps)
\\
%\le& 4L_g \eps\left|d^{n,1}_{0}\right|  + 4 L_g L_h C_g  M_1 \dt \eps \\
\le& 4L_g \eps\left|d^{n,1}_{0}\right| 
+2 a_j \Dt \left|g(x^{n,j-1}_{M_{j-1}},y^{n,j-1}_{M_{j-1}}) - G(y^{n,j-1}_{M_{j-1}}) \right| \\
&+2 a_j \Dt \left|G(y^{n,j-1}_{M_{j-1}}) - G(Y^{n,j-1}_{M_{j-1}}) \right| 
%&+ 2L_g \eps\left|d^{n,j}_{0}\right| + 2L_g L_h C_g a_j M \dt \eps
+ 2L_g \eps\left|d^{n,j}_{0}\right| + \O( M \dt \eps, M_1 \dt\eps)
\\
\le& 4L_g \eps\left|d^{n,1}_{0}\right|  
+2 L_g a_j \Dt \left| d^{n,j-1}_{M_{j-1}} \right| 
+2 L_G a_j \Dt \left|y^{n,j-1}_{M_{j-1}} - Y^{n,j-1}_{M_{j-1}}\right|
\\
&+ 2L_g \eps\left|d^{n,j}_{0}\right| 
+ \O( M \dt \eps, M_1 \dt\eps)
\\
\le&4L_g \eps\left|d^{n,1}_{0}\right|  
+ 2 L_g a_j \Dt \left(1-\frac{\dt}{\eps}\right)^{M_j} \left|d^{n,j-1}_{0} \right| + 2 L_g L_h C_g a_j \Dt \, \eps 
\\
&+2 L_G a_j \Dt \left|y^{n,j-1}_{M_{j-1}} - Y^{n,j-1}_{M_{j-1}} \right| + 2L_g \eps\left|d^{n,j}_{0}\right|+ \O( M \dt \eps, M_1 \dt\eps) \;\; ,
\end{align*}
where we have employed \eqref{mdiff_1} to bound $\left|y^{n,1}_{M_1} - Y^{n,1}_{M_1}\right|$ and Lemma~\ref{lemma_dnm} to bound $|d^{n,j-1}_{M_{j-1}} |$. Rearranging and taking the maximum over all increments in the terms in $|d^{n,j}_{0}|$ and $|d^{n,1}_{0}|$ produces
\begin{align*}
\left|y^{n,j}_{M_j} - Y^{n,j}_{M_j}\right| \le& 2 L_g \left( 3\eps + a_j \Dt \left(1-\frac{\dt}{\eps}\right)^{aM} \right) \!\! \max_{1\le k \le j} \left|d^{n,k}_{0} \right| + 2 L_g L_h C_g a_j \td \eps  \\
&+2 L_G a_j \Dt \left|y^{n,j-1}_{M_{j-1}} - Y^{n,j-1}_{M_{j-1}} \right| + \O( M \dt \eps, M_1 \dt\eps) \;\;.
\end{align*}
Iterating this relation yields to lowest order
\begin{align*}
\left|y^{n,j}_{M_j} - Y^{n,j}_{M_j}\right| \le& 2 L_g \left( 3\eps + a_j \Dt \left(1-\frac{\dt}{\eps}\right)^{aM} \right) \!\! \max_{1\le k \le j} \left|d^{n,k}_{0} \right| + 2 L_g L_h C_g a_j \td \eps  \\
&+ \mathcal{O}\left(\left(\Dt\eps + \Dt^2 (1-\frac{\dt}{\eps})^{aM}\right) \max_{1\le k \le j} \left|d^{n,k}_{0} \right|,  \td^2 \eps , M\dt\eps, M_1\dt \eps \right)  \;\; .
\end{align*}
\end{proof} 

\noindent
Lemma~\ref{phi.bound} provides bounds on the difference between solutions of the PI approximation of the slow variables in the full multiscale system and those of the PI approximation of the reduced system during the application of the microsolver. We use this result to bound the difference between the vectorfield of the PI2 method $\gt$ given by \eqref{def:gtwiddle} and the auxiliary vectorfield $\Gt$ given by \eqref{def:Gtwiddle}.

\medskip
\begin{lemma} \label{g-G}
Assuming ({\it A1})-({\it A6}), the auxiliary vectorfield $\tilde{G}$ is close to the vectorfield $\tilde{g}$ with
\begin{align*}
|\tilde{g}(x^n,y^n)-\tilde{G}(y^n)| \leq& 
%\twoparts
%{\begin{aligned}&
2 L_g \left(\frac{5}{a} \eps + \Dt \left(1-\frac{\dt}{\eps}\right)^{aM} \!  \right)\max_{1\le k \le P+1}|d^{n,k}_{0}| + 2 L_g L_h C_g \td \eps \\
&+\mathcal{O}(M_1 \dt \eps)\;\;,
%\end{aligned} \;\;\vspace{10pt}}
%{\begin{aligned}&L_g \bigg(\frac{\dt}{\dt^*}\frac{\varepsilon}{\td} + 3\frac{\dt}{\dt^*}L_g(1+L_f)\eps \bigg) |d^{n,0}| 
%\\&+ L_g e^{-\frac{M\dt^*}{\varepsilon}} |d^{n,0}|+ 3 \frac{\dt}{\dt^*}  L_g L_f C_g \eps
%\end{aligned}\;\;} \;\; .
\end{align*}
\end{lemma}
\begin{proof}
\noindent
We write
\begin{align*}
 |\tilde{g}(x^n,y^n)-\tilde{G}(y^n)| =& \Big|y^{n,1}_{M_1} - Y^{n,1}_{M_1} + \sum_{j=1}^P\frac{b_j}{a_{j+1}}\left(y^{n,j+1}_{M_{j+1}} - y^{n,1}_{M_1} - (Y^{n,j+1}_{M_{j+1}} - Y^{n,1}_{M_1})\right) \Big| \\
 \le& \left(\frac{1}{a}-1\right)\left|y^{n,1}_{M_1} - Y^{n,1}_{M_1} \right| +  \sum_{j=1}^P\frac{b_j}{a_{j+1}}\left|y^{n,j+1}_{M_{j+1}} - Y^{n,j+1}_{M_{j+1}} \right|\;\; ,
 \end{align*}
 where we have used that $1 < 1/a_{j} < 1/a$ for $j>1$. Employing Lemma~\ref{phi.bound} yields
 \begin{align*}
 |\tilde{g}(x^n,y^n)-\tilde{G}(y^n)|\le&(\frac{1}{a}-1)\left(2L_g \eps\left|d^{n,1}_{0}\right| + 2 L_g L_h C_g  M_1 \dt \eps\right) \\
 & +  \sum_{j=1}^P b_j \Bigg( 2 L_g \left( \frac{3}{a_{j+1}}\eps + \Dt \left(1-\frac{\dt}{\eps}\right)^{aM} \right) \!\! \max_{1\le k \le j} \left|d^{n,k}_{0} \right| \\
 &\hspace{1.5cm}+ 2 L_g L_h C_g  \td \eps \Bigg)
\\
\le&2 L_g \left((\frac{4}{a} - 1)\eps + \Dt \left(1-\frac{\dt}{\eps}\right)^{aM} \right)\max_{1\le k \le P+1}|d^{n,k}_{0}| \\
&+ 2 L_g L_h C_g \td \eps  + 2 (1-\frac{1}{a}) L_g L_h C_g  M_1 \dt \eps \;\; .
\end{align*}

\end{proof}
\begin{remark} \label{PI1.p2}
In order to prove the above result for the PI1 method, one constructs $\Gt$ with the increments given by vector field evaluations as discussed in Remark~\ref{PI1.p1}. The Lemma then follows along the same lines as the proofs in \cite{E03,GottwaldMaclean13}.
\end{remark}

\medskip
\noindent 
%We have established in Lemma~\ref{g-G} that the vectorfield $\tilde{G}$ given by \eqref{define_G_twiddle} is close to the PI approximation $\tilde{g}$ of the slow dynamics over a time step of $\td$. \noindent
We are now in the position to establish the bound on the discretization error
\begin{align*}
|E_d^n| = |y^n-Y(t^n)|\; ,
\end{align*}
which we formulate in the following theorem.
\medskip
\begin{theorem}
\label{theorem.Ed}
Given assumptions ({\it A1})--({\it A6}), there exists a constant $C$ such that on a fixed time interval $T$, for each $n\Dt\le T$, the error between the solution of the projective integration scheme PI2 and the exact solutions of the reduced system is bounded by
\begin{align*}
\left| E_d^n \right| \le  
% \twoparts
%{\quad \;\;\;\;
C\left(
\td^P + M\dt + \left(\frac{\eps}{\td}+ \, e^{-\frac{aM\dt}{\varepsilon}} \right) |d^n_{\max}|
 +  \varepsilon
\right)  \;\; ,
%}
%{C\frac{\dt}{\dt^*}\left(
%\td + \varepsilon
%+\left(\frac{\varepsilon}{\td}+\eps+e^{-\frac{M\dt^*}{\varepsilon}} \right)|d^n| 
%\right)}\;\; ,
\end{align*}
where $\left|d^n_{\max}\right| := \max_{\substack{0 \le i <n\;, \\ 1 \le k \le P+1}} |d^{i,k}_{0}| $ is the maximal deviation of the fast variables from the approximate slow manifold over the increments and macrosteps. 
\end{theorem}
\medskip
%\begin{remark}
%The error estimate involves the well known exponential decay of the fast variable towards the approximate slow manifold, leading to a loss of memory of the fast initial condition; PI, however, involves an additional error term proportional to the maximal distance of the fast variable to the approximate slow manifold which involves no memory loss with $M\to \infty$. 
%\end{remark}

\begin{proof} 
Combining the bound on the discretization error obtained in Proposition~\ref{lemma_Gtwiddle}, 
\begin{align*}
|E^n_d| \le& \frac{e^{L_G t^n}}{L_G} 
 \Bigg\{  
C_P^* \td^P + C_2^*M\dt
 + \max_{0\le i \le n-1}\left|\frac{\tilde{g}(x^i,y^i)-\tilde{G}(x^i,y^i)}{\td}\right|\Bigg\} \; ,
\end{align*}
with Lemma~\ref{g-G} we obtain
\begin{align*}
|E^n_d| %thesis \le&  \frac{e^{L_G t^n}}{L_G} 
% \Bigg\{  
%C_P^* \td^P + C_2^*M\dt
% + 2 L_g \left((\frac{4}{a} - 1)\frac{\eps}{\td} + \frac{\Dt}{\td} \left(1-\frac{\dt}{\eps}\right)^{aM} \!  \right)|d^{n}_{\max}| 
% \\
% &\qquad\quad\; + 2 L_g L_h C_g \eps  \Bigg\}
% \\
 \le& \frac{e^{L_G t^n}}{L_G} 
 \Bigg\{  
C_P^* \td^P + C_2^*M\dt
 + 2 L_g \left(\frac{5}{a}\frac{\eps}{\td} + e^{-\frac{aM\dt}{\varepsilon}} \right)|d^{n}_{\max}| + 2 L_g L_h C_g  \eps\Bigg\} \;\; .
\end{align*}
\end{proof}

\medskip
\noindent
%Note that $\frac{\eps}{\td} < 1$ as $\eps < \Dt$ from ({\it A6}). 

\noindent
Theorem~\ref{theorem.main} now follows from Theorems~\ref{theorem.Ec}~and~\ref{theorem.Ed}.\\
\begin{remark}
Following the comments in Remarks~\ref{PI1.p1}~and~\ref{PI1.p2}, one can obtain the same bound for the PI1 method as obtained in Theorem~\ref{theorem.Ed} for the PI2 method.\\
\end{remark}

\noindent
Besides the parameters used in the numerical scheme, i.e. the macrostep size $\td$, the number of microsteps $M$ with microstep size $\dt$, and the time scale parameter $\eps$, the error bound also involves the maximal deviation of the fast variables from the approximate slow manifold $|d^n_{\max}|$. \\

\noindent
We now establish Theorem~\ref{theorem.sec} by bounding the distance of the fast variables from the slow manifold over the increments and macrosteps in the PI1 and PI2 schemes. This provides stability conditions for the fast variables in the PI1 and PI2 schemes, i.e. conditions under which $|d^n_{\max}|$ is finite. These stability conditions are crucial for the successful application of the seamless PI methods, since - as we shall see - the fast variables depart from the slow manifold at rate proportional to $\lambda \Dt/\eps \gg 1$ over the increments and macrosteps. %We are aware of no existing results that bound the reinitiali%sation error for the PI1 scheme; however a bound on the reini%tialisation is presented for the similar scheme in \cite{EngquistTsai05}, which we will discuss at the end of the next section.

\subsection{Stability of the fast variables}
%The dependence of this term on the parameters $M_j, \dt, \eps, \td$ as well as the orders P and p of the macro- and microsolver is complicated. Rather than bound $|d^n_{\max}|$ directly, we will produce stability conditions for the two methods under which it will be controlled.
%In Lemma~\ref{lemma_dnm} we established the rate of convergence of the fast variables to the approximate slow manifold over the microsolver. We must now obtain bounds on the rate of departure of the fast variables from the slow manifold over the macrosolver, so that we can obtain stability conditions for the PI1 and PI2 methods. \\
We first bound the distance of the fast variables from the slow manifold over the PI1 increments $\hat{k}_j$, which are employed in both PI1 and PI2.

%We now establish conditions for the two methods under which $|d^n_{\max}|$ will not diverge. 
%%%%%%%%%%%%%%%%%%%%%%%%%%%%%%%%%%%%%%%%%%%%%%%%%%%%%%%%

\medskip
\begin{lemma}
\label{reinit.incr}
Given assumptions {\it (A1)} and {\it (A4)}, the distance of the fast variables from the approximate slow manifold after the $j$-th increment of the $n$-th macrostep in the PI1 and PI2 methods, given by $|d^{n,j}_{0}| = |x^{n,j}_{0} - h_0(y^{n,j}_{0})|$, satisfies
\begin{align*}
|d^{n,j}_{0}| \le&  \left(\frac{\lambda \Dt}{\eps} \left(1-\frac{\dt}{\eps}\right)^{aM}\right)^j |d^{n}| + L_h C_g \left(1 + \lambda \right) \Dt \sum_{k=0}^{j-1} \left( \frac{ \lambda \Dt}{\eps}\left(1-\frac{\dt}{\eps}\right)^{aM} \right)^k \\
&+\mathcal{O}\left(\eps \, ,\, \left(1-\frac{\dt}{\eps}\right)^{M_1} \left|d^{n}\right|\right) \;\; ,
\end{align*}
where $d^{n,1}_{0} = d^n = x^n - h_0(y^n)$, and where we define $aM = \min_j M_j$ for the PI1 method to preserve the notation for both methods.
\end{lemma}
The first term in this result measures the combined effect of the convergence of the fast variables towards the slow manifold over the microsteps, proportional to $(1-{\dt/ \eps})^{aM}$, and the departure of the fast variables from the slow manifold over the subsequent increment, proportional to $\lambda\Dt/\eps$. 
\begin{proof}
The PI1 and PI2 methods (cf \eqref{def:znj0} and \eqref{def:2znj0}), respectively satisfy
\begin{align*}
x^{n,j+1}_{0} =& x^{n,1}_{M_1} + a_{j+1} \hat{k}_{x,j}(x^n,y^n) \\
%=&x^{n,1}_{M_1} + a_{j+1} \Dt \; f\left({x}^{n,j}_{M_j},{y}^{n,j}_{M_j},\eps \right)\\
=&x^{n,1}_{M_1} + a_{j+1} \Dt \; \frac{\Lambda}{\eps}(-  \, x^{n,j}_{M_j}+ h_0(y^{n,j}_{M_j})) \;\; .
\end{align*}
%on substituting \eqref{hatkx} and \eqref{basefast}. 
Then
\begin{align}
\nonumber \left| d^{n,j+1}_{0}\right| =& \left| x^{n,j+1}_{0} - h_0(y^{n,j+1}_{0}) \right| \\
\nonumber =& \left| x^{n,1}_{M_1} + a_{j+1} \Dt \; \frac{\Lambda}{\eps}(-  \, x^{n,j}_{M_j}+ h_0(y^{n,j}_{M_j})) - h_0(y^{n,j+1}_{0}) \right| \\
\nonumber =& \left| d^{n,1}_{M_1} - a_{j+1} \Dt \frac{\Lambda}{\eps} \; d^{n,j}_{M_j} + h_0(y^{n,1}_{M_1}) - h_0(y^{n,j+1}_{0}) \right| \\
\label{incr.i} \le& \frac{\lambda \Dt}{\eps} \; \left|d^{n,j}_{M_j}\right|  + L_h \left|y^{n,1}_{M_1} - y^{n,j+1}_{0} \right| +  \left|d^{n,1}_{M_1}\right|  \;\; .
\end{align}
Employing Lemma~\ref{lemma_dnm} with the Euler microsolver and substituting the slow component of \eqref{def:znj0} and the definition of $\hat{k}$, \eqref{hatky} and \eqref{2hatky} respectively, for $y^{n,j+1}_{0}$ we obtain
\begin{align*}
\left| d^{n,j+1}_{0}\right| \le& \frac{\lambda \Dt}{\eps} \;\left( \left(1-\frac{\dt}{\eps}\right)^{M_j} \left|d^{n,j}_{0}\right| + L_h C_g \eps \right)+ L_h C_g \Dt +\left(1-\frac{\dt}{\eps}\right)^{M_1} \left|d^{n}\right| + L_h C_g \eps\\
\le&  \frac{\lambda \Dt}{\eps} \; \left(1-\frac{\dt}{\eps}\right)^{a M} \left|d^{n,j}_{0}\right| + L_h C_g \left(1 + \lambda \right) \Dt +\left(1-\frac{\dt}{\eps}\right)^{M_1} \left|d^{n}\right| + L_h C_g \eps \;\; .
\end{align*}
Iterating this recursive relation concludes the Lemma.
%\begin{align*}
%\left| d^{n,j+1}_{0}\right| \le&  \left(\frac{\lambda \Dt}{\eps} \; \left(1-\frac{\dt}{\eps}\right)^{\min_j M_j}\right)^j \left|d^{n,1}_{0}\right| \\
%&+ L_h C_g \left(1 + \lambda \right) \Dt\frac{ \left(\frac{\lambda \Dt}{\eps} \; \left(1-\frac{\dt}{\eps}\right)^{\min_j M_j}\right)^j - 1}{\frac{\lambda \Dt}{\eps} \; \left(1-\frac{\dt}{\eps}\right)^{\min_j M_j}-1} \;\; .
%\end{align*}
\end{proof}

\noindent
We now prove Theorem~\ref{theorem.sec} in two parts. We first establish bounds on the distance of the fast variables from the slow manifold after one macrostep of the PI1 scheme in Lemma~\ref{d.PI1}, and then follow with the corresponding bound for the PI2 scheme in Lemma~\ref{d.PI2}.
\medskip
\begin{lemma}
\label{d.PI1}
Given assumptions {\it (A1)} and {\it (A4)}, the distance of the fast variables from the approximate slow manifold after the $n$-th macrostep in the PI1 scheme, given by $|d^{n+1}| = |x^{n+1} - h_0(y^{n+1})|$, satisfies the recurrence relation
\begin{align*}
|d^{n+1}| \le& \sum_{j=1}^P b_j \left( \frac{\lambda \Dt}{\eps} \left(1-\frac{\dt}{\eps}\right)^{aM}\right)^j |d^{n}| \\
& + L_h C_g (1+\lambda)\Dt  \sum_{j=0}^{P-1} \left( \frac{\lambda\Dt }{\eps}\!\! \left(1-\frac{\dt}{\eps}\right)^{aM} \right)^j  +\mathcal{O}\left(\eps \, ,\, \left(1-\frac{\dt}{\eps}\right)^{M_1} \left|d^{n}\right|\right) \;\;.
\end{align*}
In particular, the fast variables do not diverge if 
\begin{align*}
\frac{\lambda \Dt}{\eps}\left(1-\frac{\dt}{\eps}\right)^{a M} < 1\;.
\end{align*} 
When this condition is satisfied, the distance of the fast variables from the slow manifold after a macrostep can be written to lowest order as
\begin{align*}
|d^{n+1}| \le \sum_{j=1}^P b_j \left( \frac{\lambda \Dt}{\eps} \left(1-\frac{\dt}{\eps}\right)^{aM}\right)^j |d^{n}| 
 \; + \; L_h C_g (1+\lambda)\Dt \;\;.
\end{align*}
\end{lemma}

\begin{proof}
%The reinitia%lisation over the macrosteps in PI1 is not significantly different to the reinit%ialisation over the increments given in Lemma~\ref{reinit.incr}. 
We substitute the macrosolver \eqref{PI1_macrox} and increments \eqref{hatkx} into $|d^{n+1}|$, obtaining
\begin{align}
\nonumber |d^{n+1}| =& \left|x^{n+1}-h_0(y^{n+1})\right| \\
%thesis \nonumber =& \big|x^{n,1}_{M_1} + \sum_{j=1}^P b_j \hat{k}_{x,j}(x^n,y^n) -h_0(y^{n+1})\big|\\
%thesis \nonumber =& \big|x^{n,1}_{M_1} + \sum_{j=1}^P b_j \Dt \, f(x^{n,j}_{M_j}, y^{n,j}_{M_j},\eps) -h_0(y^{n+1})\big|\\
%thesis \nonumber =& \big|x^{n,1}_{M_1} - \frac{\lambda \Dt}{\eps} \sum_{j=1}^P b_j \left(x^{n,j}_{M_j} -  h_0(y^{n,j}_{M_j})\right) -h_0(y^{n+1})\big|\\
\nonumber =& \big|x^{n,1}_{M_1} - \frac{\lambda \Dt}{\eps} \sum_{j=1}^P b_j d^{n,j}_{M_j} -h_0(y^{n+1})\big|\\
\nonumber =& \big|d^{n,1}_{M_1} - \frac{\lambda \Dt}{\eps} \sum_{j=1}^P b_j d^{n,j}_{M_j} +h_0(y^{n,1}_{M_1}) -h_0(y^{n+1})\big|\\
\nonumber \le& \frac{\lambda \Dt}{\eps}  \sum_{j=1}^P b_j \big|d^{n,j}_{M_j}\big|  + \big|h_0(y^{n,1}_{M_1}) -h_0(y^{n+1})\big|+ |d^{n,1}_{M_1}| \;\;.
\end{align}
Employing Lemma~\ref{lemma_dnm}, Assumption ({\it A1}) on the Lipschitz continuity of the approximate slow manifold and \eqref{def:znj0}--\eqref{hatky} to bound $|h_0(y^{n,1}_{M_1}) - h_0(y^{n+1})|$ produces
\begin{align*}
|d^{n+1}| \le&\frac{\lambda \Dt}{\eps}  \sum_{j=1}^P b_j \left(\left(1-\frac{\dt}{\eps}\right)^{M_j}|d^{n,j}_{0}| + L_h C_g \eps \right)\\
&+ L_h\left|y^{n,1}_{M_1} -y^{n+1}\right| +\left(1-\frac{\dt}{\eps}\right)^{M_1} \left|d^{n}\right| + L_h C_g \eps \\
\le & \frac{\lambda \Dt}{\eps}  \sum_{j=1}^P b_j \left(1-\frac{\dt}{\eps}\right)^{aM}|d^{n,j}_{0}| + L_h C_g (1+\lambda)\Dt \\
&+\left(1-\frac{\dt}{\eps}\right)^{M_1} \left|d^{n}\right| + L_h C_g \eps\;\; .
\end{align*}
Substituting the bound from Lemma~\ref{reinit.incr} into $|d^{n,j}_{0}|$ yields
\begin{align}
\nonumber |d^{n+1}| \le & \sum_{j=1}^P b_j\Bigg(\left(\frac{\lambda \Dt}{\eps} \left(1-\frac{\dt}{\eps}\right)^{aM}\right)^{j} |d^{n}|+\!\! \Bigg(\!\! L_h C_g (1+\lambda)\Dt \\
\nonumber&\qquad\qquad +\left(1-\frac{\dt}{\eps}\right)^{\!\!M_1}\!\!\! \left|d^{n}\right| + L_h C_g \eps \! \Bigg)\! \sum_{k=0}^{j-1} \left( \frac{\lambda \Dt }{\eps}\left(1-\frac{\dt}{\eps}\right)^{aM} \right)^{\!\!k}\Bigg) \\
\nonumber& \le \sum_{j=1}^P b_j\left(\frac{\lambda \Dt}{\eps} \!\! \left(1-\frac{\dt}{\eps}\right)^{aM}\right)^{j} \!\!\!\! |d^{n}| + \Bigg(L_h C_g (1+\lambda)\Dt\\
\label{r.PI1}& \qquad\qquad +\left(1-\frac{\dt}{\eps}\right)^{M_1} \left|d^{n}\right| + L_h C_g \eps \Bigg) \sum_{j=0}^{P-1} \left( \frac{\lambda\Dt }{\eps}\!\! \left(1-\frac{\dt}{\eps}\right)^{aM} \right)^j  .
\end{align} 
The upper bound on $|d^n|$ diverges as n increases unless
\begin{align*}
\frac{\lambda \Dt}{\eps} \left(1-\frac{\dt}{\eps}\right)^{aM} <1 \;\;,
\end{align*}
completing the Lemma.
\end{proof}

%\begin{remark}\label{pat1}
%From equation \eqref{r.PI1} one can infer that $\frac{\lambda\Dt}{\eps}\left(1-\frac{\dt}{\eps}\right)^{aM}<1$ is a sufficient condition for the reinitia%lisation over the macrosteps to be bounded in the PI1 method. When this condition is satisfied, the reiniti%alisation in the PI1 method becomes, to lowest order,
%\begin{align*}
%|d^{n+1}| \le & \sum_{j=1}^P b_j\left(\frac{\lambda \Dt}{\eps} \!\! \left(1-\frac{\dt}{\eps}\right)^{aM}\right)^{j} \!\!\!\! |d^{n}| + L_h C_g (1+\lambda) \Dt   \;\; .
%\end{align*}
%\end{remark}

\begin{remark}\label{PI1.r}
The bound presented above for PI1 is not sharp. 
%It implies that the distance of the fast variables from the approximate slow manifold after one macrostep scales linearly with $\Dt$. However,
If the duration of the microsolver is sufficiently large with $M\dt \gg \eps$, or if the fast variables are initialised on the slow manifold with $|d^0=0|$, then the PI1 increments may be approximately tangent to the slow manifold and higher order accuracy in $\Dt$ can be achieved.\\
\end{remark}

\noindent
We now formulate analogous results for PI2.
\begin{lemma}
\label{d.PI2}
Given assumptions {\it (A1)}, {\it (A3)} and {\it (A4)}, the distance of the fast variables from the approximate slow manifold after the $n$-th macrostep in the PI2 scheme, given by $|d^{n+1}| = |x^{n+1} - h_0(y^{n+1})|$, satisfies the recurrence relation

\begin{align*}
|d^{n+1}| \le &  \ \frac{\left(1-\frac{\dt}{\eps}\right)^{aM}}{a}\Bigg[\sum_{j=1}^P b_j \left(\frac{\lambda \Dt}{\eps} \left(1-\frac{\dt}{\eps}\right)^{aM}\right)^j |d^{n}| \\
 &+ L_h C_g (1+\lambda)\Dt \sum_{j=0}^{P-1} \left( \frac{\Dt \lambda}{\eps}\left(1-\frac{\dt}{\eps}\right)^{aM} \right)^j\Bigg]+2L_{h'} C_g^2 \td^2 \\
 & +\mathcal{O}\left(\eps \, ,\, \left(1-\frac{\dt}{\eps}\right)^{M_1} \left|d^{n}\right|\right)  \;\; .
\end{align*}
In particular, the fast variables do not diverge if 
\begin{align*}
\frac{\left(1-\frac{\dt}{\eps}\right)^{aM}}{a} \left(\frac{\lambda \Dt}{\eps} \left(1-\frac{\dt}{\eps}\right)^{aM}\right)^j <1 \;\;\; \forall \; j.
\end{align*} 
When this condition is satisfied, the distance of the fast variables from the slow manifold after a macrostep can be written to lowest order as
\begin{align*}
|d^{n+1}| \le &  \ \frac{\left(1-\frac{\dt}{\eps}\right)^{aM}}{a}\sum_{j=1}^P b_j \left(\frac{\lambda \Dt}{\eps} \left(1-\frac{\dt}{\eps}\right)^{aM}\right)^j |d^{n}| +2L_{h'} C_g^2 \td^2 \;\; .
\end{align*}
\end{lemma} 
\noindent
We remark that the first bound presented in the above Lemma for PI2 is precisely $\left(1-\frac{\dt}{\eps}\right)^{aM}/a$ times the bound presented for PI1 in Lemma~\ref{d.PI1}, with an additional term proportional to $ \td^2$ and the curvature of the slow manifold, measured by $L_{h'}$.

\begin{proof}
We reformulate $x^{n+1}$, employing \eqref{macrox} and \eqref{def:kx} and Lemma~\ref{lemma_dnm} to estimate
\begin{align}
\nonumber |d^{n+1}| =& \left|x^{n+1}-h_0(y^{n+1})\right| \\
%\nonumber =& \big|x^{n,1}_{M_1} + \sum_{j=1}^P b_j k_{x,j}(x^n,y^n) -h_0(y^{n+1})\big|\\
\nonumber=&  \big|x^{n,1}_{M_1} + \sum_{j=1}^P \frac{b_j}{a_{j+1}}\left(x^{n,j+1}_{M_{j+1}} - x^{n,1}_{M_1}\right) -h_0(y^{n+1})\big| \\
\nonumber \le&  \big|h_0(y^{n,1}_{M_1}) - h_0(y^{n+1}) + \sum_{j=1}^P \frac{b_j}{a_{j+1}}\left(h_0(y^{n,j+1}_{M_{j+1}}) - h_0(y^{n,1}_{M_1})\right)\big| \\
\label{x.relax} & +\sum_{j=1}^P \frac{b_j}{a_{j+1}}\left|d^{n,j+1}_{M_{j+1}} + (a_{j+1}-1) d^{n,1}_{M_1}\right|  \;\; .
\end{align}
%In the first line of \eqref{x.relax}, the reinitia%lisation error is written as the difference between a weighted sum of chords from $h_0(y^{n,1}_{M_1})$ to points on the exact slow manifold, and the exact value of the slow manifold at the next step, $h_0(y^{n+1})$. The second line counts the accumulated errors from replacing the approximations $x^{n,j}_{M_j}$ with the exact value on the slow manifold. \\
The first term on the right-hand side of \eqref{x.relax} can be Taylor expanded to second order to obtain
\begin{align}
\nonumber h_0(y^{n,1}_{M_1}) - h_0(y^{n+1}) =& h_0(y^{n,1}_{M_1}) - h_0\left(y^{n,1}_{M_1} + \sum_{j=1}^P b_jk_{y,j}\right) \\
\nonumber =&- \D h_0\left(y^{n,1}_{M_1}\right) \left(\sum_{j=1}^P b_j k_{y,j}\right) \\
\label{t.heps}&-\sum_{|\alpha|=2}\frac{1}{\alpha!} \partial^\alpha h_0\left(y^{n,1}_{M_1}\right)\Bigg(\sum_{j=1}^P b_j k_{y,j}\Bigg)^{\!\alpha} +\mathcal{O}(\td^3) \;\;,
\end{align}
where we used multi-index notation to denote the second order derivatives of $h_0$. Similarly, the second term on the right-hand side of \eqref{x.relax} can be estimated by Taylor expanding the chord $h_0(y^{n,j+1}_{M_{j+1}}) - h_0(y^{n,1}_{M_1})$ to second order, employing \eqref{def:ky}, with
\begin{align}
\nonumber h_0(y^{n,j+1}_{M_{j+1}}) - h_0(y^{n,1}_{M_1}) =& h_0\left(y^{n,1}_{M_1} + (y^{n,j+1}_{M_{j+1}}-y^{n,1}_{M_1})\right) - h_0(y^{n,1}_{M_1}) \\
\nonumber =&\D h_0(y^{n,1}_{M_1})(y^{n,j+1}_{M_{j+1}}-y^{n,1}_{M_1}) \\
\nonumber &+\sum_{|\alpha|=2} \frac{1}{\alpha!}\partial^\alpha h_0\left(y^{n,1}_{M_1}\right)\left(y^{n,j+1}_{M_{j+1}}-y^{n,1}_{M_1}\right)^{\!\alpha}  +\mathcal{O}(\td^3)  \\
\nonumber =&a_{j+1}\D h_0(y^{n,1}_{M_1}) k_{y,j}\\
\label{t.micro}&+a_{j+1}^2\sum_{|\alpha|=2} \frac{1}{\alpha!}\partial^\alpha h_0\left(y^{n,1}_{M_1}\right)k^{\alpha}_{y,j} +\mathcal{O}(\td^3)   \;\;.
\end{align}
Substituting \eqref{t.micro} and \eqref{t.heps} into \eqref{x.relax} yields
\begin{align}
\nonumber |d^{n+1}| \le&  \left| \sum_{|\alpha|=2} \frac{1}{\alpha!}\partial^\alpha h_0\left(y^{n,1}_{M_1}\right)\right| \max_{|\alpha|=2}\left|\sum_{j=1}^P b_j a_{j+1} k^{\alpha}_{y,j} - \Bigg(\sum_{j=1}^P b_j k_{y,j}\Bigg)^{\!\alpha} \right| \\
\nonumber & +\sum_{j=1}^P \frac{b_j}{a_{j+1}}\left|d^{n,j+1}_{M_{j+1}} + (a_{j+1}-1) d^{n,1}_{M_1}\right|+\mathcal{O}(\td^3)  \\
\nonumber \le&  2L_{h'} \max_{\substack{|\alpha|=2, \\ 1 \le k \le P}}|k_{y,j}(x^n,y^n)|^\alpha\\
\nonumber & +\sum_{j=1}^P \frac{b_j}{a_{j+1}}\left(|d^{n,j+1}_{M_{j+1}}| + |d^{n,1}_{M_1}|\right)   +\mathcal{O}(\td^3)\;\;,
\end{align}
where we used that $\sum_{j=1}^P b_j=1$, $a_j\le 1$, and employed Assumption ({\it A3}) on the Lipshitz continuity of the Jacobian $Dh_0$. Employing Lemma~\ref{lemma_dnm} and recalling that the time step covered by each PI2 increment is $\td$, we obtain
\begin{align*}
|d^{n+1}| \le &2L_{h'} C_g^2 \td^2 + \sum_{j=1}^P \frac{b_j}{a_{j+1}}\left( \left(1-\frac{\dt}{\eps}\right)^{aM}\!\!\! \left(|d^{n,j+1}_{0}| + |d^{n,1}_{0}|\right) + 2L_h C_g \eps \right)   \;\;,
\end{align*}
which on substituting Lemma~\ref{reinit.incr} becomes
\begin{align*}
|d^{n+1}| \le &  \frac{\left(1-\frac{\dt}{\eps}\right)^{aM}}{a}\Bigg[\sum_{j=1}^P b_j \left(\frac{\lambda \Dt}{\eps} \left(1-\frac{\dt}{\eps}\right)^{aM}\right)^j |d^{n}| + \Bigg(\!\! L_h C_g (1+\lambda)\Dt 
\\& \qquad\quad\qquad+\left(1-\frac{\dt}{\eps}\right)^{M_1} \left|d^{n}\right| + L_h C_g \eps\! \Bigg) \! \sum_{j=0}^{P-1} \left( \frac{\Dt \lambda}{\eps}\left(1-\frac{\dt}{\eps}\right)^{aM} \right)^j\Bigg]\\
&+2L_{h'} C_g^2 \td^2  + \mathcal{O}(\eps) \;\; .
\end{align*}
The upper bound on $|d^n|$ diverges as n increases unless
\begin{align*}
\frac{\left(1-\frac{\dt}{\eps}\right)^{aM}}{a} \left(\frac{\lambda \Dt}{\eps} \left(1-\frac{\dt}{\eps}\right)^{aM}\right)^j <1 \;\;\; \forall \; j\;,
\end{align*}
completing the Lemma.
\end{proof}
%\begin{remark}\label{pat2}
%Similarly to the PI1 method, one can infer that a sufficient condition for the rein%itialisation over the macrosteps to be bounded in the PI2 method is 
%\begin{align*}
%\frac{\left(1-\frac{\dt}{\eps}\right)^{aM}}{a} \left(\frac{\lambda \Dt}{\eps} \left(1-\frac{\dt}{\eps}\right)^{aM}\right)^j <1 \;\;\; \forall \; j .
%\end{align*}
%When this condition is satisfied, the reinitia%lisation in the PI2 method becomes, to lowest order,
%\begin{align*}
%|d^{n+1}| \le &  \frac{\left(1-\frac{\dt}{\eps}\right)^{aM}}{a}\sum_{j=1}^P b_j \left(\frac{\lambda \Dt}{\eps} \left(1-\frac{\dt}{\eps}\right)^{aM}\right)^j |d^{n}| 
%\\& + L_h C_g (1+\lambda)\Dt  \;\; .
%\end{align*}
%If instead we apply the PI1 stability condition $\frac{\lambda \Dt}{\eps} \left(1-\frac{\dt}{\eps}\right)^{aM} < 1$ we obtain to lowest order the second order bound
%\begin{align*}
%|d^{n+1}| \le & 2L_{h'} C_g^2 \td^2 + \frac{\left(1-\frac{\dt}{\eps}\right)^{aM}}{a}\sum_{j=1}^P b_j \left(\frac{\lambda \Dt}{\eps} \left(1-\frac{\dt}{\eps}\right)^{aM}\right)^j |d^{n}| \;\;.
%\\& + \frac{\left(1-\frac{\dt}{\eps}\right)^{aM}}{a}\left(L_h C_g (1+\lambda)\Dt +\left(1-\frac{\dt}{\eps}\right)^{M_1} \left|d^{n}\right|\right) + \mathcal{O}(\eps)   \;\; .
%\end{align*}
%\end{remark}
Combining Lemmas~\ref{d.PI1} and \ref{d.PI2} yields Theorem \ref{theorem.sec}.\\

\section{Numerics}
\label{sec:numerics}
We now illustrate the key results of Theorem~\ref{theorem.main} with a fourth-order Runge-Kutta macrosolver, which we recall here for $P=4$ including the constants obtained in the proof. We employ a forward Euler microsolver with $p=1$ unless otherwise stated. \\
In order to compare PI1 and PI2 at the same computational cost, we choose the number of microsteps in the PI2 method proportionally lower so that the two methods take the same number of microsteps over one macrostep, with $M_j=M$ in PI1 and $M_j = \{ M, M/2, M/2, M, M\}$ in PI2. Recalling Theorem~\ref{theorem.main} for $P=4$, the discretisation error $|E^n_d| = |y^n-Y(t^n)|$ is bounded in the PI1 and PI2 method by
\begin{align*}
|E^n_d| \le& \frac{e^{L_G t^n}}{L_G} 
 \Bigg\{  
C_4^* \td^4 + C_2^*M_{\uR{1},\uR{2}}\dt
 + 2 L_g \left(9\frac{\eps}{\td} + e^{-\frac{M_{\uR{1},\uR{2}}\dt}{\varepsilon}} \right)|d^{n}_{\max}| + 2 L_g L_h C_g  \eps\Bigg\} \;\; ,
\end{align*}
%and in the PI2 method by
%\begin{align*}
%|E^n_d| \le& \frac{e^{L_G t^n}}{L_G} 
% \Bigg\{  
%C_4^* \td^4 + C_2^*\frac{M}{2}\dt
% + 2 L_g \left(9\frac{\eps}{\td} + e^{-\frac{M\dt}{2\varepsilon}} \right)|d^{n}_{\max}| + 2 L_g L_h C_g  \eps\Bigg\} \;\; ,
%\end{align*}
with $M_{\uR{1}}=M$ for PI1 and $M_{\uR{2}}=M/2$ for PI2, and where for $P=4$, $a=1/2$. However, the distance of the fast variables from the slow manifold after a macrostep scales differently in the two methods. Recalling Theorem~\ref{theorem.sec}, the distance of the fast variables from the approximate slow manifold $|d^n| = |x^n-h_0(y^n)|$ is bounded for stable applications of the PI1 method by
\begin{align*}
|d^{n+1}| \le& \sum_{j=1}^P b_j \left( \frac{\lambda \Dt}{\eps} \left(1-\frac{\dt}{\eps}\right)^{M}\right)^j |d^{n}|  + L_h C_g (1+\lambda)\Dt \;\;,
\end{align*}
and for stable applications of the PI2 method by
\begin{align*}
|d^{n+1}| \le &  \ 2\left(1-\frac{\dt}{\eps}\right)^{\frac{M}{2}}\sum_{j=1}^P b_j \left(\frac{\lambda \Dt}{\eps} \left(1-\frac{\dt}{\eps}\right)^{\frac{M}{2}}\right)^j |d^{n}| +2L_{h'} C_g^2 \td^2   \;\; .
\end{align*}

We show results for the multiscale system
\begin{align}
\label{toy_y}\dot{y}_{\eps} &= -x_\eps y_\eps -\alpha\; y_\eps^2 \\
\label{toy_x}\dot{x}_{\eps} &= \frac{-x_\eps+\sin^2(y_\eps)}{\eps} \; ,
\end{align}
which has stable fixed point at $(0,0)$. At lowest order in $\eps$, the associated slow reduced system is given by
\begin{align}
\label{toy_Y} \dot{Y} &= -Y\sin^2(Y) - \alpha\; Y^2 \; .
\end{align}
For higher order approximations of the slow manifold and the associated coordinate transformations relating $y$ and $Y$ the reader is referred to the useful webtool \cite{RobertsWeb} (see also \cite{Roberts08}).\\ The system (\ref{toy_y})-(\ref{toy_x}) with initial conditions $y_\eps(0) > 0$ is locally Lipschitz with Lipschitz constant $L_h = 1$ and $L_g={\rm{max}}(|x_\eps|+2\alpha |y_\eps|)$ where the maximum is taken over the local region around the initial conditions $(x_\eps(0),y_\eps(0))$ under consideration. The vectorfield of the slow dynamics (\ref{toy_y}) is locally bounded by $C_g = \max( |x_\eps y_\eps|+\alpha |y_\eps|^2)$, with the maximum taken over the same region. Note that the free parameter $\alpha$ controls the constants $C_2^*= \alpha y_\eps(0)^3(2\alpha+\sin(2y_\eps(0))) + \O(\sin^3(y_\eps(0)))$ and $C_4^* = 16\alpha^3 y_\eps(0)^3 + 8\alpha^4y_\eps(0)^5 + \O(\sin^5(y_\eps(0)))$. \\

\noindent
We first investigate how the discretization error $|E^n_d|$ scales with the macrostep size $\Dt$ in the PI1 and PI2 methods, when all other parameters are kept fixed (except $n$, to fix the final time T). Our analytical result predicts that, so long as $|d^n_{\max}|$ is small and the practical assumption $\eps < M\dt$ is satisfied, results will be divided into two regimes: for $C_P^*\td^P < C_2^*M\dt$, the bound for $|E^n_d|$ is dominated by the term proportional to $C_2^* M\dt$ and $|E^n_d|$ is independant of $\td$; for $C_P^*\td^P > C_2^*M\dt$, the scaling is $|E^n_d| \sim \td^P$. The slight advantage of the PI2 method in this case is that distributing the same total number of microsteps over more applications of the microsolver results in lower error due to $M\dt$. To keep the term proportional to $|d^n_{\max}|$ small in both cases, we choose parameters so that $\frac{\lambda \Dt}{\eps} \left(1-\frac{\dt}{\eps}\right)^{aM} < 1$. The predicted regimes are clearly visible in Figure \ref{EDt}, where results are presented for a range of macrostep sizes $\Dt$ for PI1 and PI2.
We choose $\alpha=0.2$, and the scale separation parameter $\eps=10^{-9}$. We use $M=40$ microsteps with microstep size $\dt = 0.4\eps$, while the number of iterations $n$ vary from $20$ to $10^5$ to keep $T=1$ fixed for all values of $\Dt$. Initial conditions are chosen to lie on the approximate slow manifold with $y^0 = 1$, $x^0 = \sin^2(1)$. The Lipschitz constants are $L_g=1.1$ and $L_h = 1$, the bound on the vector field of the slow dynamics is $C_g=2$, and the maximal derivatives of the reduced slow dynamics are $C_2^*=4$ and $C_4^* = 8$.\\

\begin{figure}
\includegraphics[width=\textwidth]{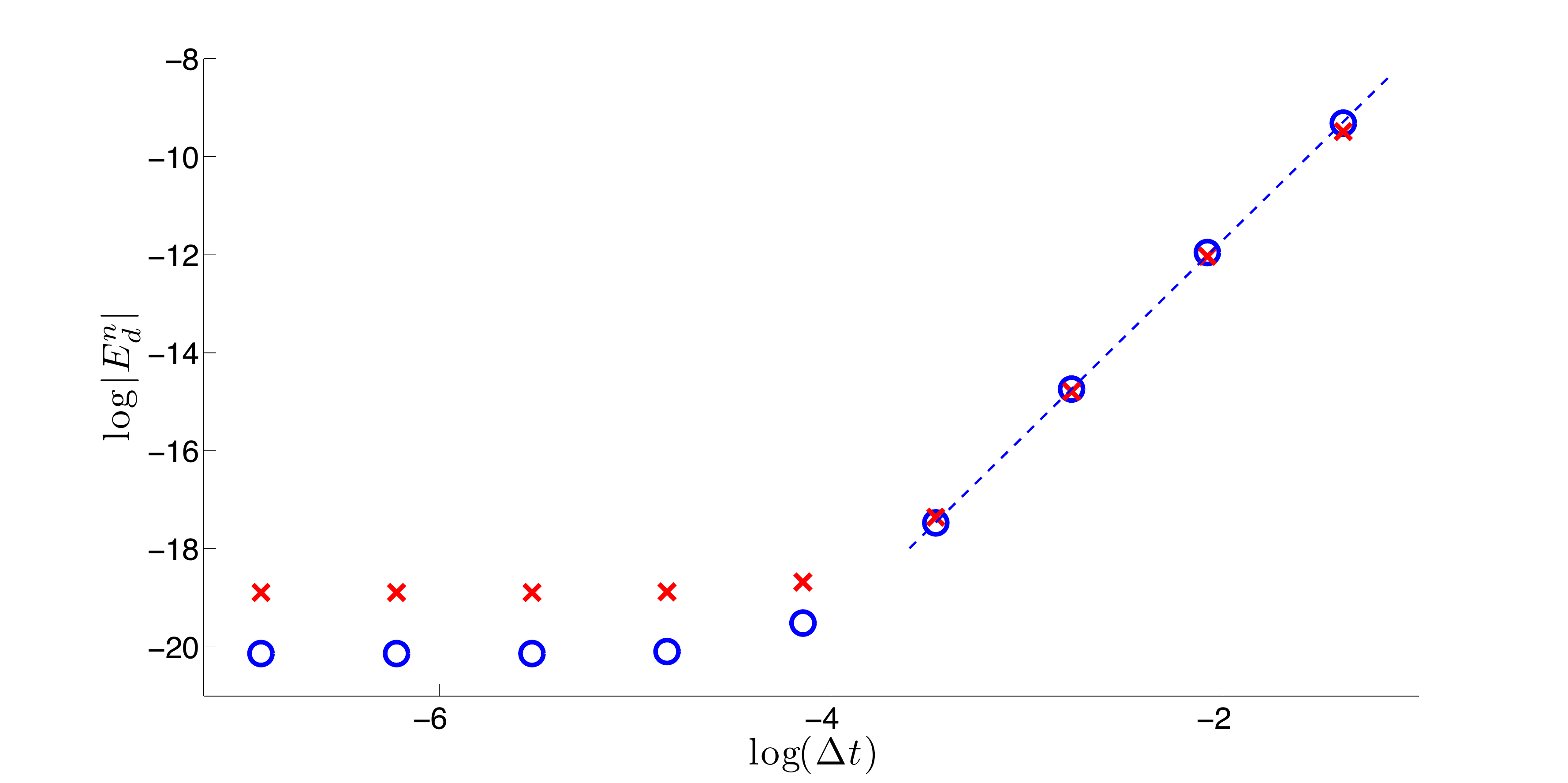}
\caption{Plot of $\log |E_d^n|$ versus $\log (\Dt)$ for fixed time of integration $T=1$ of the system \eqref{toy_y}-\eqref{toy_x}. The crosses represent results from the PI1 scheme and the circles represent results from the PI2 scheme. The dashed line is a linear regression line with a slope of $3.93$.}
\label{EDt}
\end{figure}

\noindent
We present results for the error scaling of $|E^n_d|$ with the microstep size $\dt$ in Figure \ref{dt}. To focus on the scaling with $M\dt$, we select parameters with $C^*_P \td^P < C^*_2 M\dt$, and control the distance of the fast variables from the slow manifold $|d^n_{\max}|$ by ensuring $\frac{\lambda \Dt}{\eps} \left(1-\frac{\dt}{\eps}\right)^{aM} < 1$. Figure \ref{dt} confirms our analytical result, that under the condition $C_P^*\td^P < C_2^*M\dt$, the discretization error scales like $E^n_d \sim M\dt$. The advantage of the PI2 method here is that distributing the microsteps over an additional application of the microsolver leads to an overall smaller error compared to PI1 due to the smaller drift of the slow variables over the microsolver. \\
We use a second-order Runge-Kutta microsolver (i.e. $p=2$), to demonstrate that the scaling with $\dt$ is not affected by the order of the microsolver. We choose $\alpha=1$, and the scale separation parameter $\eps=10^{-5}$. We use $M=100$ microsteps and $n=50$ iterations of each method. The macrostep size $\Dt$ varies from $0.0035$ to $0.0032$ to keep $\td$ fixed as $\dt$ increases. Initial conditions are chosen to lie on the approximate slow manifold with $y^0 = 5$, $x^0 = \sin^2(5)$. The Lipschitz constants are $L_g=11$ and $L_h = 1$, the bound on the vector field of the slow dynamics is $C_g=30$, and the maximal derivatives of the reduced slow dynamics are $C_2^*=50$ and $C_4^* = 2000$.\\

\begin{figure}
\includegraphics[width=\textwidth]{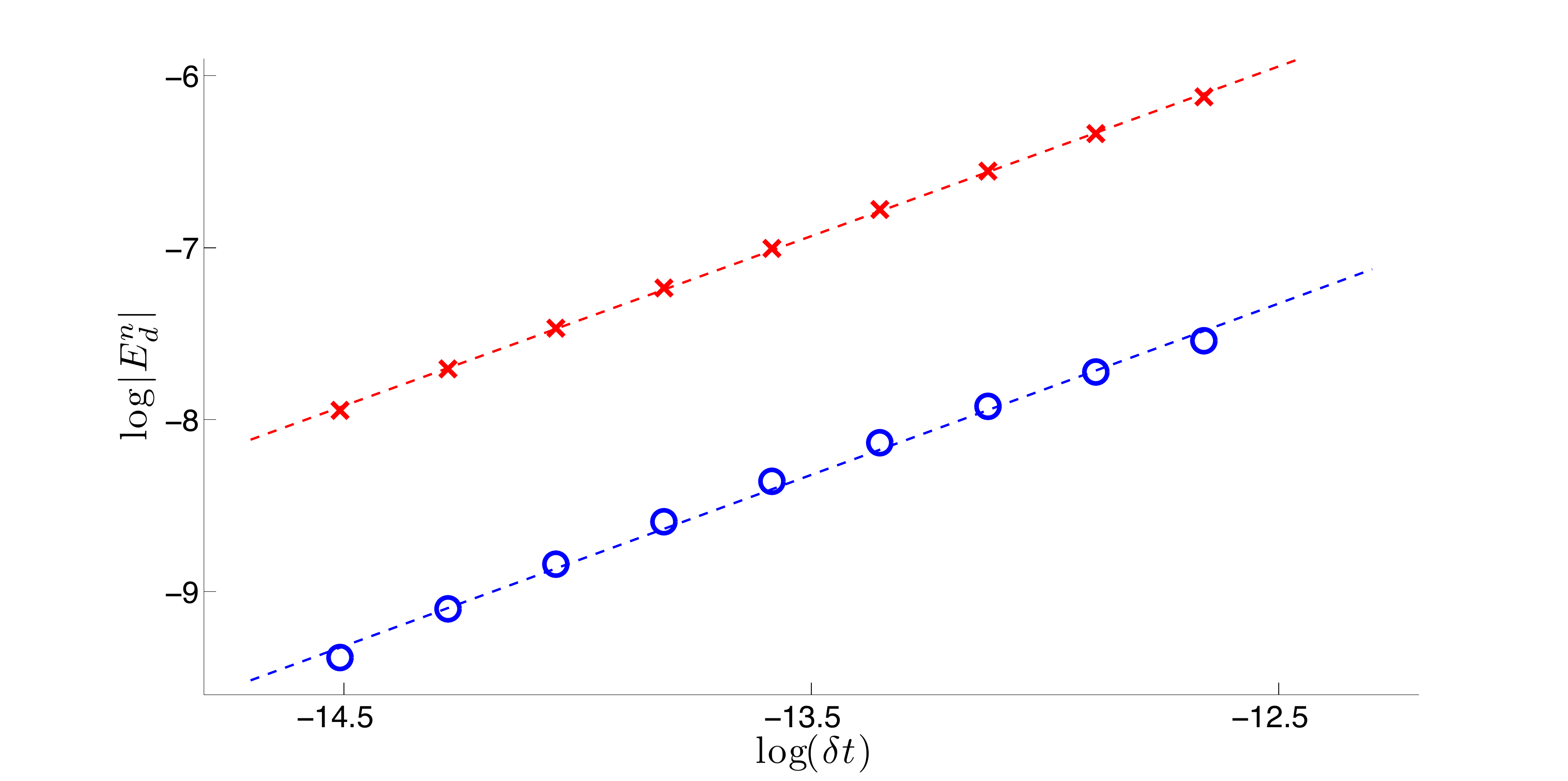}
\caption{Plot of $\log |E_d^n|$ versus $\log (\dt)$ for fixed time of integration $T=0.18$ of the system \eqref{toy_y}-\eqref{toy_x}. The crosses represent results from the PI1 scheme and the circles represent results from the PI2 scheme. The dashed lines are linear regression lines with a slopes of $0.99$ and $1.0$, respectively.} 
\label{dt}
\end{figure}

\noindent
We illustrate the linear scaling of $|E_d^n|$ with the maximal distance $|d^n_{\max}|$ of the fast variable from the approximate slow manifold after a macrostep in Figure~\ref{dn}. We do so by scaling the initial condition for the fast variables, $x^0$. To ensure that the error is not dominated by the initial initialization error $|d^{0}|$, we choose parameters which render the scheme unstable, allowing for divergence of the fast variables from the slow manifold over the macrosteps, i.e. $|d^{n}| > |d^{0}|$. Figure~\ref{dn} confirms clearly the linear dependence of $|E_d^n|$ on $|d^n_{\max}|$. 
We choose $\alpha=1$, and the scale separation parameter $\eps=10^{-4}$. We use $M=100$ microsteps with microstep size $\dt=0.01\eps$, and $n=5$ iterations of each method with macrostep size $\Dt=10^{-3}$. Initial conditions are $y^0 = 1$, $x^0 \in [\sin^2(1)+0.01, \sin^2(1)+1]$. The Lipschitz constants are $L_g=3$ and $L_h = 1$, the bound on the vector field of the slow dynamics is $C_g=2$, and the maximal derivatives of the reduced slow dynamics are $C_2^*=3$ and $C_4^* = 24$.\\

\begin{figure}
\includegraphics[width=\textwidth]{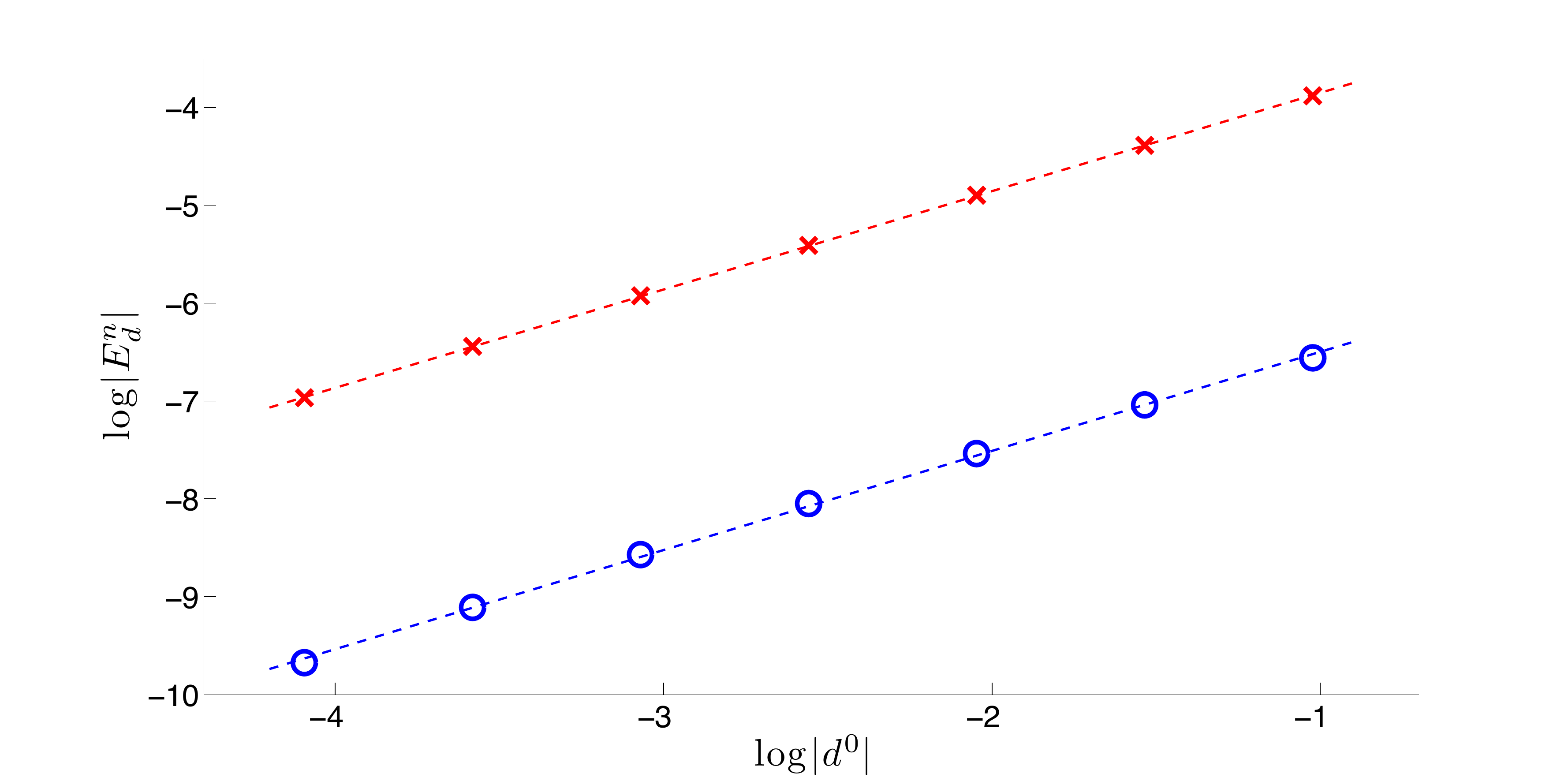}
\caption{Plot of $\log |E_d^n|$ versus $\log |d^0|$ for fixed time of integration $T=0.0056$ of the system \eqref{toy_y}-\eqref{toy_x}. The crosses represent results from the PI1 scheme and the circles represent results from the PI2 scheme. The dashed lines are linear regression lines with slopes of $1.01$.} 
\label{dn}
\end{figure}

\noindent
We investigate how $|d^n_{\max}|$, the maximal deviation of the fast variables from the slow manifold, scales with $\Dt$ in the PI1 and PI2 methods. We choose parameters satisfying $\frac{\lambda \Dt}{\eps}\left(1-\frac{\dt}{\eps}\right)^{aM} \ll 1$, so that the bounds presented for $|d^n|$ in Theorem~\ref{theorem.sec} imply
\begin{align*}
|d^{n}_{\max}| &\le L_h C_g (1+\lambda) \Dt + \mathcal{O}(\Dt^2) \;\;
\end{align*}
for the PI1 method, and 
\begin{align*}
|d^{n}_{\max}| \le &2L_{h'} C_g^2 \td^2 + \mathcal{O}\Big(\Dt^3, \frac{(1-\frac{\dt}{\eps})^{aM}}{a}\Dt \Big) \;\;
\end{align*}
for the PI2 method. As noted in Remark~\ref{PI1.r}, the bound for the PI1 method is not tight for systems with $|d^0|=0$. We therefore choose initial conditions off the slow manifold. Furthermore, to ensure that the initial error $|d^0|$ does not dominate the error $|d^n_{\max}|$, we record $|d^n_{\max}|$ after the first macrostep.
Figure \ref{dDt} clearly shows the linear dependence of $|d^n_{\max}|$ with the macrostep size $\Dt$ for PI1, and the quadratic dependence of $|d^n_{\max}|$ with the macrostep $\td$ for PI2. 
We choose again $\alpha=0.2$, and the scale separation parameter $\eps=10^{-9}$. We use $M=40$ microsteps with microstep size $\dt = 0.4\eps$, while the number of iterations $n$ vary from $20$ to $10^5$ to keep $T=1$ fixed for all values of $\Dt$. Initial conditions are $y^0 = 1$, $x^0 = \sin^2(1)+1$. The Lipschitz constants are $L_g=5$ and $L_h = 1$, the bound on the vector field of the slow dynamics is $C_g=2$, and the maximal derivatives of the reduced slow dynamics are $C_2^*=4$ and $C_4^* = 8$.\\

%\begin{figure}
%\includegraphics[width=\textwidth]{paper_EDt2}
%\label{EDt2}
%\end{figure}

\begin{figure}
\includegraphics[width=\textwidth]{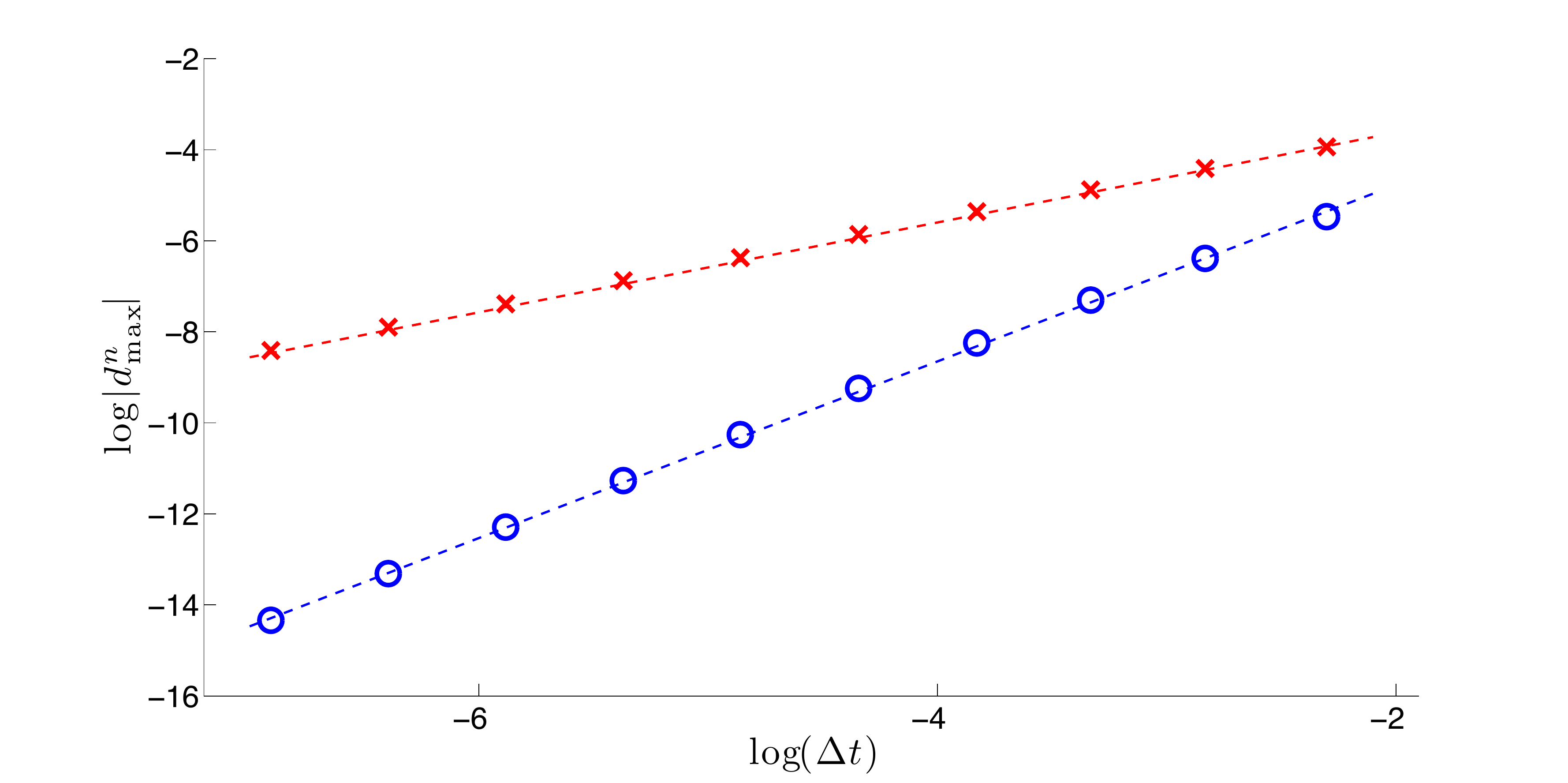}
\caption{Plot of $\log |d^n_{\max}|$ versus $\log (\Dt)$ for fixed time of integration $T=1$ of the system \eqref{toy_y}-\eqref{toy_x}. The crosses represent results from the PI1 scheme and the circles represent results from the PI2 scheme. The dashed lines are linear regression lines with a slopes of $0.98$ and $1.94$, respectively. } 
\label{dDt}
\end{figure}

Finally, we investigate how $\Delta y^{T,\Dt} = |y^{T,\Dt} - y^{T,\Dt/2}|$ scales with the macrostep size $\Dt$ where $y^{T,\Dt}$ are the outputs of the PI1 or PI2 methods with macrostep size $\Dt$ and final time $T$. In \cite{GearLee05} $\Delta y^{T,\Dt}$ was used as a measure of the numerical error.
%To do so, we adjust our notation to focus on the final time and macrostep size in each method. Denote by $y^{T,\Dt}$ the approximation to $y_\eps(T)$ produced by the PI1 or PI2 method with macrostep $\Dt$, assuming $T=n\Dt$, $n\in \mathbb{N}$. We illustrate the scaling of $\Delta y^{T,\Dt} = |y^{T,\Dt} - y^{T,\Dt/2}|$ with $\Dt$, keeping T and all other numerical parameters except n fixed, in Figure~\ref{dy}, employing the same numerical parameters as in Figure~\ref{EDt}. 
In Figure~\ref{dy} we show how $\Delta y^{T,\Dt}$ scales with $\Dt$ for a fourth-order Runge-Kutta macrosolver, employing the same numerical parameters as in Figure~\ref{EDt}.  It is seen that $|y^{T,\Dt} - y^{T,\Dt/2}| \sim \td^4$ for all values of $\Dt$ whereas the actual discretization error is dominated by $M\dt$ at the lower values of $\Dt$ (cf. Figure~\ref{EDt}). Hence, such proxies for the numerical error have to be treated with caution when evaluating PI methods.\\

\begin{figure}
\includegraphics[width=\textwidth]{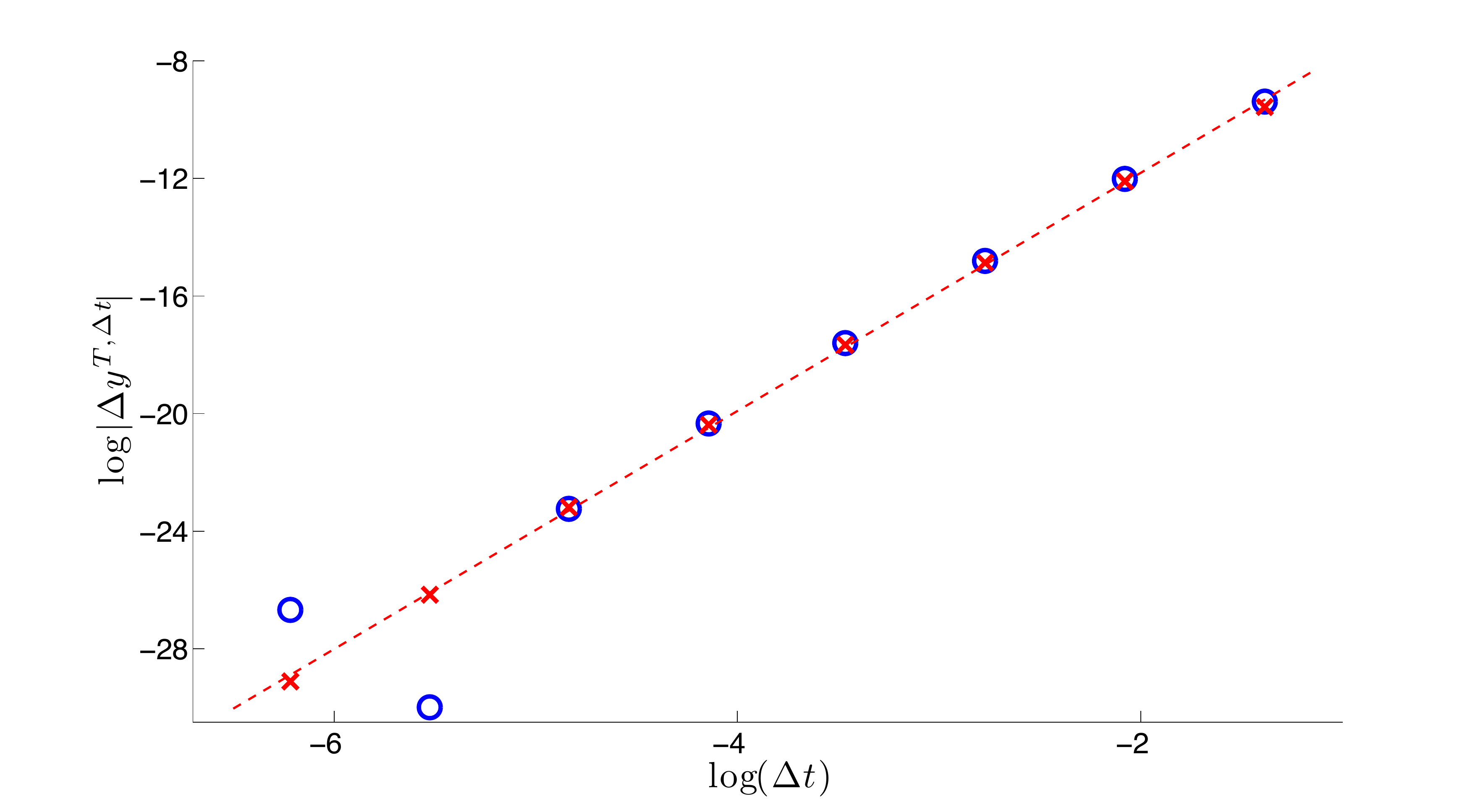}
\caption{Plot of $\log |\Delta y^{T,\Dt}|$ versus $\log |\Dt|$ for fixed time of integration $T=1$ of the system \eqref{toy_y}-\eqref{toy_x}. The crosses represent results from the PI1 scheme and the circles represent results from the PI2 scheme. The dashed line is a linear regression line (of the PI1 output) with slope of $4.05$.} 
\label{dy}
\end{figure}

We comment that the numerical results presented here are robust; in particular, we confirm that identical scalings can be produced from simulations of the Michaelis-Menten system employed in \cite{GearEtAl05}, although its fast dynamics does not satisfy our form \eqref{basefast}, and the Brusselator with rapidly replenished source employed in \cite{GearKevrekidis03}, where the approximate slow manifold is constant.

%%%%%%%%%%%%%%%%%%%%%%%%%%%%%%%%%%%

\section{Discussion}
\label{sec:summary}
We have introduced PI2, a seamless numerical multiscale method with a higher-order macrosolver, which is a slight modification of a standard implementation of a projective integration method, PI1, involving an additional application of the microsolver. \kone{ %The new method PI2 applies the microsolver before and after the evaluation of the increments. 
In both PI1 and PI2, each increment is rooted on the slow manifold. In PI1 the increments typically do not end on the slow manifold. In contrast, the additional application of the microsolver assures that in PI2 each increment also ends on the slow manifold, even for slow manifolds with non-vanishing curvature (see Figures~\ref{fig:PI1} and \ref{fig:PI2}). If the slow manifold is sufficiently linear over the course of one macrostep, the increments of PI2 then all lie approximately tangential to it.}\\ %If the slow manifold is sufficiently linear over the course of one macrostep, the additional application of the microsolver in PI2 assures that the increments lie approximately \kone{tangential} to the slow manifold. If the slow manifold is sufficiently nonlinear and varies significantly over the course of one macrostep, the additional application of the microsolver in PI2 assures that \kone{the tip and tail of each increment still lie close to} the slow manifold. By contrast, in PI1 the increments will in general point away from the slow manifold in the latter situation.\\  
We presented error bounds for the slow variables for both methods, expressed in Theorem~\ref{theorem.main}. \ktwo{The error bounds are not affected by the order of the microsolver used (though strictly speaking, we only considered explicit microsolver schemes). Hence the contribution of the microsolver to the error constitutes a bottleneck for PI methods, after which the error in the slow variables cannot be improved by adjusting the macrostep size or the order of the macro- or microsolver. Hence} there is no gain to be expected in the slow dynamics when microsolvers other than forward Euler schemes are used. \\
In Theorem~\ref{theorem.sec} we derived bounds for the unphysical deviation of the fast variables from the slow manifold, which may cause numerical instability \cite{EngquistTsai05},  and provided a stability criterion for the macrostep size.\\
% \ktwo{We note that a similar accuracy and stability analysis of the PI1 scheme applied to partial differential equations can be found in \cite{LafitteEtAl14}.}\\  

\noindent
The Theorems now allow us to compare the PI1 and PI2 schemes. A fair comparison requires that both schemes are operated at the same computational cost. Hence, PI2 utilises less microsteps per application of the microsolver during the construction of the increments as the total number of microsteps is distributed over one more application of the microsolver. Consequently, the absolute discretisation error of PI2 is smaller when compared to PI1. 
Theorem~\ref{theorem.sec} establishes that the PI2 method incurs less deviation from the slow manifold as the deviations scale quadratically with the macrostep size rather than linearly as for PI1. The improved stability can be attributed to the increments of PI2 pointing towards the slow manifold, enforced by the additional relaxation towards the slow manifold when constructing the increments.

\section*{Acknowledgments}
\noindent
Georg Gottwald acknowledges support from the Australian Research Council. John Maclean is supported by a University of Sydney Postgraduate Award.

%%%%%%%%%%%%%%%%%%%%%%%%%%%%%%%%%%%
\section*{References}
\bibliographystyle{elsarticle-num}
\bibliography{bibliography}

\begin{thebibliography}{10}
\expandafter\ifx\csname url\endcsname\relax
  \def\url#1{\texttt{#1}}\fi
\expandafter\ifx\csname urlprefix\endcsname\relax\def\urlprefix{URL }\fi
\expandafter\ifx\csname href\endcsname\relax
  \def\href#1#2{#2} \def\path#1{#1}\fi

\bibitem{KevrekidisGearEtAl03}
I.~G. Kevrekidis, C.~W. Gear, J.~M. Hyman, G.~K. Panagiotis, O.~Runborg,
  C.~Theodoropoulos, Equation-free, coarse-grained multiscale computation:
  Enabling microscopic simulators to perform system-level analysis, Comm. Math.
  Sci. 1~(4) (2003) 715--762.

\bibitem{GearKevrekidis03}
C.~Gear, I.~Kevrekidis, Projective methods for stiff differential equations:
  Problems with gaps in their eigenvalue spectrum, SIAM J. Sci. Comp. 24~(4)
  (2003) 1091--1106.

\bibitem{HummerKevrekidis03}
G.~Hummer, I.~Kevrekidis, Coarse molecular dynamics of a peptide fragment: Free
  energy, kinetics, and long-time dynamics computations, Journal of Chemical
  Physics 118~(23) (2003) 10762--10773.

\bibitem{GearLee05}
S.~L. Lee, C.~W. Gear, Second-order accurate projective integrators for
  multiscale problems, Journal of Computational and Applied Mathematics 201~(1)
  (2007) 258--274.

\bibitem{LustRooseVandekerckhove06}
C.~Vandekerckhove, D.~Roose, K.~Lust, Numerical stability analysis of an
  acceleration scheme for step size constrained time integrators, Journal of
  Computational and Applied Mathematics 200~(2) (2007) 761--777.

\bibitem{GivonEtAl06}
D.~Givon, I.~G. Kevrekidis, R.~Kupferman, Strong convergence of projective
  integration schemes for singularly perturbed stochastic differential systems,
  Comm. Math. Sci. 4~(4) (2006) 707--729.

\bibitem{KevrekidisSamaey09}
I.~Kevrekidis, G.~Samaey, Equation-free multiscale computation: algorithms and
  applications, Ann. Rev. Phys. Chem. 60 (2009) 321--344.

\bibitem{LafitteSamaey10}
P.~Lafitte, G.~Samaey, Asymptotic-preserving projective integration schemes for
  kinetic equations in the diffusion limit, SIAM J. Sci. Comput. 34~(2) (2010)
  A579--A602.

\bibitem{EEngquist03}
W.~E, B.~Engquist, The heterogeneous multiscale methods, Comm. Math. Sci. 1~(1)
  (2003) 87--132.

\bibitem{VandenEijnden03}
E.~Vanden-Eijnden, Numerical techniques for multi-scale dynamical systems with
  stochastic effects, Comm. Math. Sci. 1~(2) (2003) 385--391.

\bibitem{E03}
W.~E, Analysis of the heterogeneous multiscale method for ordinary differential
  equations, Comm. Math. Sci. 1~(3) (2003) 423--436.

\bibitem{EngquistTsai05}
B.~Engquist, Y.-H. Tsai, Heterogeneous multiscale methods for stiff ordinary
  differential equations, Mathematics of Computation 74~(252) (2005)
  1707--1742.

\bibitem{ELiuVandenEijnden05}
W.~E, D.~Liu, E.~Vanden-Eijnden, Analysis of multiscale methods for stochastic
  differential equations, Communications on Pure and Applied Mathematics
  58~(11) (2005) 1544--1585.

\bibitem{EEtAl07}
W.~E, B.~Engquist, X.~Li, W.~Ren, E.~Vanden-Eijnden, Heterogeneous multiscale
  methods: {A} review, Comm. Comp. Phys. 2~(3) (2007) 367--450.

\bibitem{Liu10}
D.~Liu, Analysis of multiscale methods for stochastic dynamical systems with
  multiple time scales, SIAM Multiscale Model. Simul. 8~(3) (2010) 944--964.

\bibitem{VandenEijnden07}
E.~Vanden-Eijnden, On {HMM}-like integrators and projective integration methods
  for systems with multiple time scales, Comm. Math. Sci 5~(2) (2007) 495--505.

\bibitem{EVandenEijnden08}
W.~E, E.~Vanden-Eijnden, Some critical issues for the ``{E}quation-{F}ree"
  approach to multiscale modeling, arXiv:0806.1621v1 [math.NA].

\bibitem{GottwaldMaclean13}
J.~Maclean, G.~A. Gottwald, On convergence of the projective integration method
  for stiff ordinary differential equations, Comm. Math. Sci. 12~(2) (2014)
  235--255.

\bibitem{GearEtAl02}
C.~Gear, I.~G. Kevrekidis, C.~Theodoropoulos, `coarse' integration/bifurcation
  analysis via microscopic simulators: micro-galerkin methods, Computers and
  Chemical Engineering 26 (2002) 941--963.

\bibitem{RooseVandekerckhove06}
C.~Vandekerckhove, D.~Roose, Accuracy analysis of acceleration schemes for
  stiff multiscale problems, Journal of Computational and Applied Mathematics
  211~(2) (2008) 181--200.

\bibitem{LafitteEtAl14}
P.~Lafitte, A.~Lejon, G.~Samaey, A high-order asymptotic-preserving scheme for
  kinetic equations using projective integration, preprint, arXiv:1404.6104v2.

\bibitem{GearEtAl05}
C.~Gear, T.~J. Kaper, I.~G. Kevrekidis, A.~Zagaris, Projecting to a slow
  manifold: Singularly perturbed systems and legacy codes, SIAM J. Appl. Dyn.
  Syst. 4~(3) (2005) 711--732.

\bibitem{ZagarisEtAl09}
A.~Zagaris, C.~Gear, T.~Kaper, I.~Kevrekidis, Analysis of the accuracy and
  convergence of equation-free projection to a slow manifold, Math. Mod. Num.
  Anal. 43~(4) (2009) 757--784.

\bibitem{VandekerckhoveEtAl11}
C.~Vandekerckhove, B.~Sonday, A.~Makeev, D.~Roose, I.~Kevrekidis, A common
  approach to the computation of coarse-scale steady states and to consistent
  initialization on a slow manifold., Computers \& Chemical Engineering 35~(10)
  (2011) 1949--1958.

\bibitem{ZagarisEtAl12}
A.~Zagaris, C.~Vandekerckhove, C.~Gear, T.~Kaper, I.~Kevrekidis, Stability and
  stabilization of the constrained runs schemes for equation-free projection to
  a slow manifold., Discrete and Continuous Dynamical Systems - Series A 32~(8)
  (2012) 2759--2803.

\bibitem{SiettosEtAl12}
C.~I. Siettos, C.~W. Gear, I.~G. Kevrekidis, An equation-free approach to
  agent-based computation: {B}ifurcation analysis and control of stationary
  states, EPL (Europhysics Letters) 99~(4) (2012) 48007.

\bibitem{Iserles}
A.~Iserles, A First Course in the Numerical Analysis of Differential Equations,
  Cambridge University Press, Cambridge, 2009.

\bibitem{Carr}
J.~Carr, {Applications of Centre Manifold Theory}, no.~35 in Applied
  Mathematical Sciences, Springer, 1981.

\bibitem{RobertsWeb}
A.~J. Roberts, Slow manifold of stochastic or deterministic multiscale
  differential equations,
  \url{http://www.maths.adelaide.edu.au/anthony.roberts/sdesm.php} (2008).

\bibitem{Roberts08}
A.~J. Roberts, Normal form transforms separate slow and fast modes in
  stochastic dynamical systems, Physica A 387~(1) (2008) 12--38.

\end{thebibliography}
%%%%%%%%%%%%%%%%%%%%%%%%%%%%%%%%%%%

\end{document}